\documentclass[12pt]{amsart}
\usepackage{packages}

\title[Regularity of minimizing currents at boundaries with multiplicity $Q$]{On the regularity of area minimizing currents at boundaries with arbitrary multiplicity}

\author{Ian Fleschler and Reinaldo Resende}

\thanks{\tiny{\emph{I. Fleschler: Dept of Mathematics, Princeton University, Princeton, USA. Email: imf@princeton.edu.}}}
\thanks{\tiny{\emph{R. Resende: Dept of Mathematics, Carnegie Mellon University, Pittsburgh, USA. Email: rresende@andrew.cmu.edu.}}}

\begin{document}

\begin{abstract}
In this paper, we consider an area-minimizing integral $m$-current $T$ within a submanifold $\Sigma$ of $\mathbb{R}^{m+n}$, taking a boundary $\Gamma$ with arbitrary multiplicity $Q \in \mathbb{N} \setminus \{0\}$, where $\Gamma$ and $\Sigma$ are $C^{3, \kappa}$. We prove a \emph{sharp} generalization of Allard's boundary regularity theorem to a higher multiplicity setting, precisely, we prove that the set of density $Q/2$ singular boundary points of $T$ is $\mathcal{H}^{m-3}$-rectifiable. As a consequence, we show that the entire boundary regular set, without any assumptions on the density, is open and dense in $\Gamma$ which is also \emph{dimensionally sharp}. Moreover, we prove that if $p \in \Gamma$ admits an open neighborhood in $\Gamma$ consisting of density $Q/2$ points with a tangent cone supported in a half $m$-plane, then $p$ is regular. 

Furthermore, we show that if the convex barrier condition is satisfied—namely, if $\Gamma$ is a closed manifold that lies at the boundary of a uniformly convex set and $\Sigma = \mathbb{R}^{m+n}$—then the entire boundary singular set is $\mathcal{H}^{m-3}$-rectifiable. Additionally, we investigate certain assumptions on $\Gamma$ that enable us to provide further information about the singular boundary set.
\end{abstract}

\maketitle

\vspace{-1cm}

\setcounter{tocdepth}{2} 
\tableofcontents

\section{Introduction}

The present work, together with \cite{ian2024uniqueness} and \cite{ian2024example}, provides a sharp answer to a long-standing open question posed in Allard’s 1969 PhD thesis \cite{AllPhD}, on the boundary regularity of area minimizing $m$-currents in arbitrary dimension, codimension, and boundary multiplicity, under a convexity assumption. Our work sharply extends Allard’s celebrated 1975 boundary regularity theorem \cite{AllB} to the much more delicate and intricate case of higher boundary multiplicity, in the minimizing setting. Before our work, even the existence of a single regular boundary point was not known in this setting. 

We develop a sharp regularity theory that answers Allard’s question in full generality by proving the sharp $\mathcal{H}^{m-3}$-rectifiability of the boundary singular set. The challenge of understanding higher multiplicity boundaries, which we address in Allard’s original form, was later highlighted in a broader framework by White in the famous collection of open problems from the 1984 AMS Summer Institute in Geometric Measure Theory \cite[Problem 4.19]{GMT_prob} and was revisited in the ICM 2022 survey by De Lellis \cite{de2022regularity}.

It is worth emphasizing that establishing the rectifiability of the singular set in geometric problems is usually significantly more difficult than obtaining the dimensional bound—and in this work, we tackle and resolve both problems simultaneously.

Moreover, we also establish a dimensionally sharp result showing that the boundary regular set is open and dense in the boundary, without the convexity assumption.

\subsection{Historical overview.} Consider an area minimizing integral $m$-currents $T$ within a $(m+\bar{n})$-submanifold $\Sigma\in C^{3,\kappa}$ of $\R^{m+n}, n\geq \bar{n}\geq 0$, and $m\geq 2$. The boundary $\Gamma \subset \Sigma$ is also of class $C^{3,\kappa}$ and $T$ takes the boundary $\Gamma$ with multiplicity $Q\in\N\setminus \{0\}$. We refer the reader to \Cref{S:preliminaries} for precise definitions. 

In the masterpiece \cite{HS}, Hardt and Simon prove that there are \emph{no} boundary singular points in the context of codimension and multiplicity $1$, i.e., $n=Q=1$. By a decomposition argument (which only works for $n=1$), in \cite{white1983regularity}, White shows that all boundary points are regular in the setting of codimension $1$ and arbitrary multiplicity, i.e., $n=1$ and $Q\geq 1$.

It is well known that the theory becomes significantly more involved when one takes a step forward and considers the higher codimension setting, i.e., $n>1$. In fact, in \cite[Theorem 1.8 (b)]{DDHM}, De Lellis, De Philippis, Hirsch, and Massaccesi, provide an example of a $2$ dimensional area-minimizing integral current $T$ in $\R^4$ with a boundary $\Gamma\in C^{\infty}$ (consisting of two disjoint closed curves) of multiplicity $Q=1$, such that the Hausdorff dimension of the boundary singular set $\Singb(T)$ is $1$. However, in this example the singular set has $\cH^1$-measure zero, leaving open the question of whether or not it is possible to prove that $\cH^{m-1}(\Singb(T))=0$ in general.

Assuming $\partial T = Q\a{\Gamma}$, we define a point $p\in\Gamma$ as \emph{one-sided} if the density of $T$ at $p$ is $Q/2$, and a point $q\in\Gamma$ as \emph{two-sided} if the density of $T$ restricted to an open neighborhood of $q$ in $\Gamma$ and intersected with $\Gamma$ is constant equal to $Q/2 + Q^\star$ for some $Q^\star\in\N\setminus \{0\}$. We refer the reader to \Cref{D:1-2-sided} and the accompanying pictures.

In the celebrated paper \cite{AllB}, Allard proves that if only points of density $1/2$ are allowed and $Q=1$, i.e., one-sided points in the multiplicity $1$ case, then all boundary points of $T$ are regular for any codimension $n \geq 1$. In particular, when the boundary lies on a convex barrier, defined as the boundary of a strictly convex open set, all boundary points are regular. However, the situation becomes more intricate if other types of boundary points are also allowed—for instance, \emph{two-sided} points. Indeed, the example in \cite{DDHM} mentioned above highlights how much less is expected (compared to \cite{AllB} and the case $n=1$) in terms of describing the boundary regular set of $T$. Nonetheless, in the groundbreaking work \cite{DDHM}, the authors adapt all the machinery introduced in \cite{Alm,DS1,DS2,DS3,DS4,DS5} (see also \cite{nardulli2024interior} for $\Sigma\in C^{2,\kappa}$) to prove that the boundary regular set of $T$ is open and dense in $\Gamma$, provided that the boundary multiplicity is $Q=1$ (cf. \Cref{T:open-dense}). In light of the example in \cite{DDHM}, this result is close to the best one can hope for.

When studying the case of arbitrary boundary multiplicity $Q\geq 1$ and arbitrary codimension $n\geq 1$, recent progress has been made in the context of $2$-dimensional currents. Specifically, in \cite{delellis2021allardtype}, for $m=2$ and considering \emph{only one-sided} points, De Lellis, Nardulli, and Steinbr\"{u}chel, showed that all boundary points of $T$ are regular (cf. \Cref{T:Rectifiability-1sided}). Their work relies fundamentally on two key results proven by them: (1) the uniqueness of tangent cones in this dimension $m=2$ setting, and (2) the regularity theory for $Q$-valued functions minimizing the Dirichlet energy with a prescribed $C^3$ boundary. Furthermore, in \cite{nardulli2022density}, the second-named author and Nardulli allow \emph{two-sided} points to exist in $\Gamma$ and rely on a regularity theory for $(Q-Q^\star/2)$-valued Dirichlet minimizers, along with the uniqueness of tangent cones in \cite{delellis2021uniqueness}, to prove that the boundary regular set of $T$ is open and dense in the setting $Q\geq 1$, $n\geq 1$, and $m=2$.

\subsection{Description of the main results}

This work aims at better understanding the general case, that is, $Q, m,$ and $n$, all arbitrary. One of the novelties here is that there is no boundary regularity theory developed for $Q$-valued Dirichlet minimizers for $m>2$ and also \emph{one-sided singular} points of $T$ may exist (see \Cref{T:ian-example}). This differs dramatically from every other previous investigation due to the fact that \emph{one-sided} points are always regular in the settings described above (i.e., $m=2$ or $Q=1$). Meanwhile, the uniqueness of tangent cones for $m=2$ obtained in \cite{delellis2021allardtype} is now known for general dimensions $m$, as well as an excess decay-type estimate with exponential decay rate, due to \cite{ian2024uniqueness} by the first named author (see \Cref{T:Uniqueness 1sided,T:excessdecay}). Such facts are crucial for the present article and will be exhaustively used in this work. 

Summarizing, for $m,n,Q$ arbitrary, the following results are proved in this article: 

\begin{enumerate}[\upshape (1)]
    \item the set of regular boundary points is open and dense in $\Gamma\in C^{3,\kappa}$, see \Cref{T:open-dense},

    \item the set of singular boundary points with density $Q/2$ of $T$ is $\cH^{m-3}$-rectifiable, see \Cref{T:Rectifiability-1sided},

    \item the set of boundary singular points of $Q$-valued Dirichlet minimizers is $\cH^{m-3}$-rectifiable, see \Cref{T:LinearProblem:Rectifiability},

    \item if $p$ is a density $Q/2$ flat point which has an open neighborhood in the boundary of density $Q/2$ flat points, then $p$ is regular, see \Cref{T:1sided-flat-regular}.
\end{enumerate}
%(1) $\Regb(T)$ is open and dense in $\Gamma\in C^{3,\kappa}$, see \Cref{T:open-dense}, (2) the set of boundary singular points with density $Q/2$ of $T$ is $\cH^{m-3}$-rectifiable, see \Cref{T:Rectifiability-1sided}, (3) the set of boundary singular points of $Q$-valued Dirichlet minimizers is $\cH^{m-3}$-rectifiable, see \Cref{T:LinearProblem:Rectifiability}, and (4) if $p$ is a density $Q/2$ flat point which has an open neighborhood in the boundary of density $Q/2$ flat points, then $p$ is regular, see \Cref{T:1sided-flat-regular}.

Item (1) is \emph{optimal} in the sense that there are examples of boundary singular sets with the same Hausdorff dimension as the boundary (see \Cref{T:4names example}). The results (2) and (3) are \emph{optimal} with respect to the Hausdorff dimensional bound (see \Cref{T:ian-example}) and, in addition to proving that the regular set has Hausdorff dimension bounded by $m-3$, we also establish a structural property (rectifiability) of the singular set. Item (4) is likewise \emph{optimal} in the sense that the assumption of an open neighborhood of \emph{flat} one-sided points is essential and cannot be removed (cf. \Cref{R:1sided-flat-assumption}).

The investigation of rectifiability of singular sets has been recently developed in the context of higher codimensions $n>1$ to tackle \emph{interior} regularity in \cite{de2023fineI, de2023fineII, de2023fineIII}, see also \cite{krummel2023analysisI,krummel2023analysisII}. The approach used in this work resembles that of \cite{de2023fineI, de2023fineII}, which, while inspired by the seminal work \cite{NV} by Naber and Valtorta, differs \emph{significantly} from it.

For the first time, this machinery is applied to study boundary regularity and it introduces significant conceptual differences compared to the interior setting.

The estimates derived through the Naber-Valtorta framework are identical for both linear (Dirichlet minimizing) and non-linear (area minimizing) problems. In fact, the $\cH^{m-3}$-rectifiability for the linear problem is also established in this article. However, we will not provide an extensive treatment of the linear problem before presenting the proofs for the non-linear case.

It must be remarked that the novelty of this work, which the authors find surprising, is that in this framework proving the $(m-3)$-Hausdorff dimensional bound for the one-sided singular set does not seem significantly easier to prove than proving the $\cH^{m-3}$-rectifiability of the boundary singular set.

The results, loosely stated in this Introduction, are described in a precise manner in what follows.

\subsection{Definition of area minimality and one- and two-sided regular points}

The following assumptions will always be in place: $\Sigma$ and $\Gamma$ are a $(m+\bar{n})$-submanifold and a $(m-1)$-submanifold of $\R^{m+n}$, respectively. They are both of class $C^{3,\kappa}$, for some $\kappa>0$, and satisfy $\Gamma\subset\Sigma$ and $\bar{n}\geq 1$. The space of integral $k$-dimensional currents in $\R^{m+n}$ will be denoted by $\mathbf{I}_k(U)$ for $k\in \{0,\ldots, m+n\}$, the mass of a current $T$ is denoted by $\mathbf{M}(T)$, and $\spt(T)$ denote its support.

\begin{definition}\label{D:AreaMin}
Let A current $T \in\mathbf{I}_m(U)$ with $\spt(T)\subset\Sigma$ is \emph{area minimizing in $\Sigma\cap U$}, if it holds that $\mathbf{M}(T) \leq \mathbf{M}(T+\partial \tilde{T})$, for all $\tilde{T}\in\mathbf{I}_{m+1}(U)$ with $\spt(\tilde{T})\subset\Sigma\cap U$. 
\end{definition}

The definitions of one-sided and two-sided points are provided below. Here, for a smooth orientable $k$-manifold, we denote the naturally induced $k$-current by $\a{M}$, and $\Theta^m(T, p)$ denotes the density of $\|T\|$ (the total variation of $T$) at $p$ with respect to the $m$-dimensional Hausdorff measure. For precise basic definitions, we refer the reader to \Cref{S:preliminaries}.

\begin{definition}\label{D:1-2-sided}
Let $T \in\mathbf{I}_m(U)$ with $\spt(T)\subset\Sigma$, $\partial T={Q}\a{\Gamma}, {Q}\in\N\setminus\{0\}$, then $p\in\Gamma$ is called 
\begin{enumerate}[\upshape (A)]
    \item an \emph{one-sided point for $T$} if $\Theta^m(T,p) = Q/2$,

    \item a \emph{two-sided point for $T$} if there exist $Q^\star\in \N\setminus\{0\}$ and an open neighborhood $Z$ of $p$ such that $\Theta^m(T,q) = Q/2 + Q^\star$ for every $q\in \Gamma\cap Z$.
\end{enumerate}
\end{definition}

\begin{figure}[H]
\centering
\begin{subfigure}{.5\linewidth}
\centering
\captionsetup{width=.9\linewidth}
\includegraphics[width=.99\linewidth]{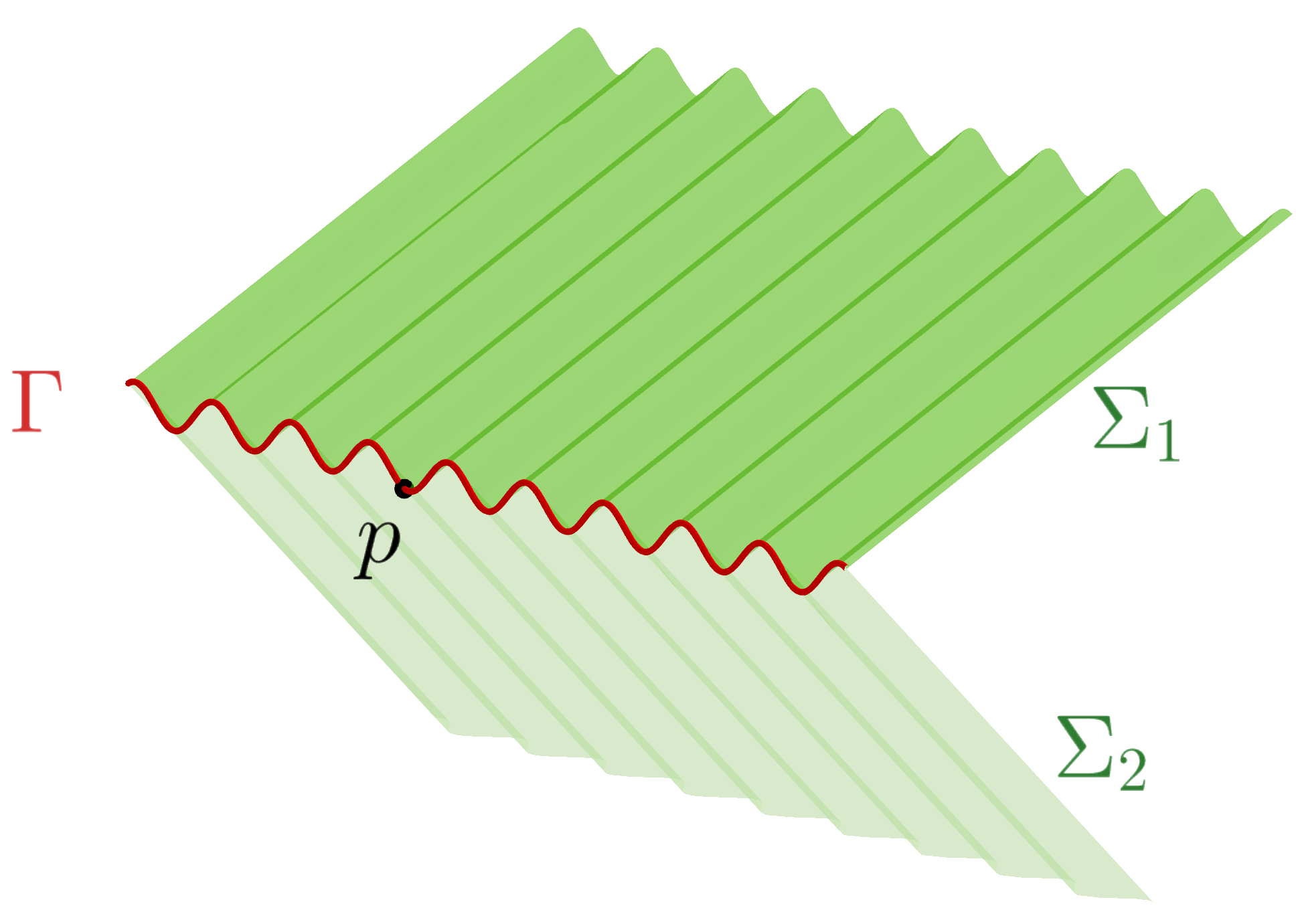}
\caption{If $T$ is given by $\a{\Sigma_1} + \a{\Sigma_2}$, then $p\in\Gamma$ is a \emph{one-sided}.}
\label{F: 1sided}
\end{subfigure}
%\hspace{1cm}
\begin{subfigure}{.4\linewidth}
\centering
\captionsetup{width=.9\linewidth}
\includegraphics[width=.99\linewidth]{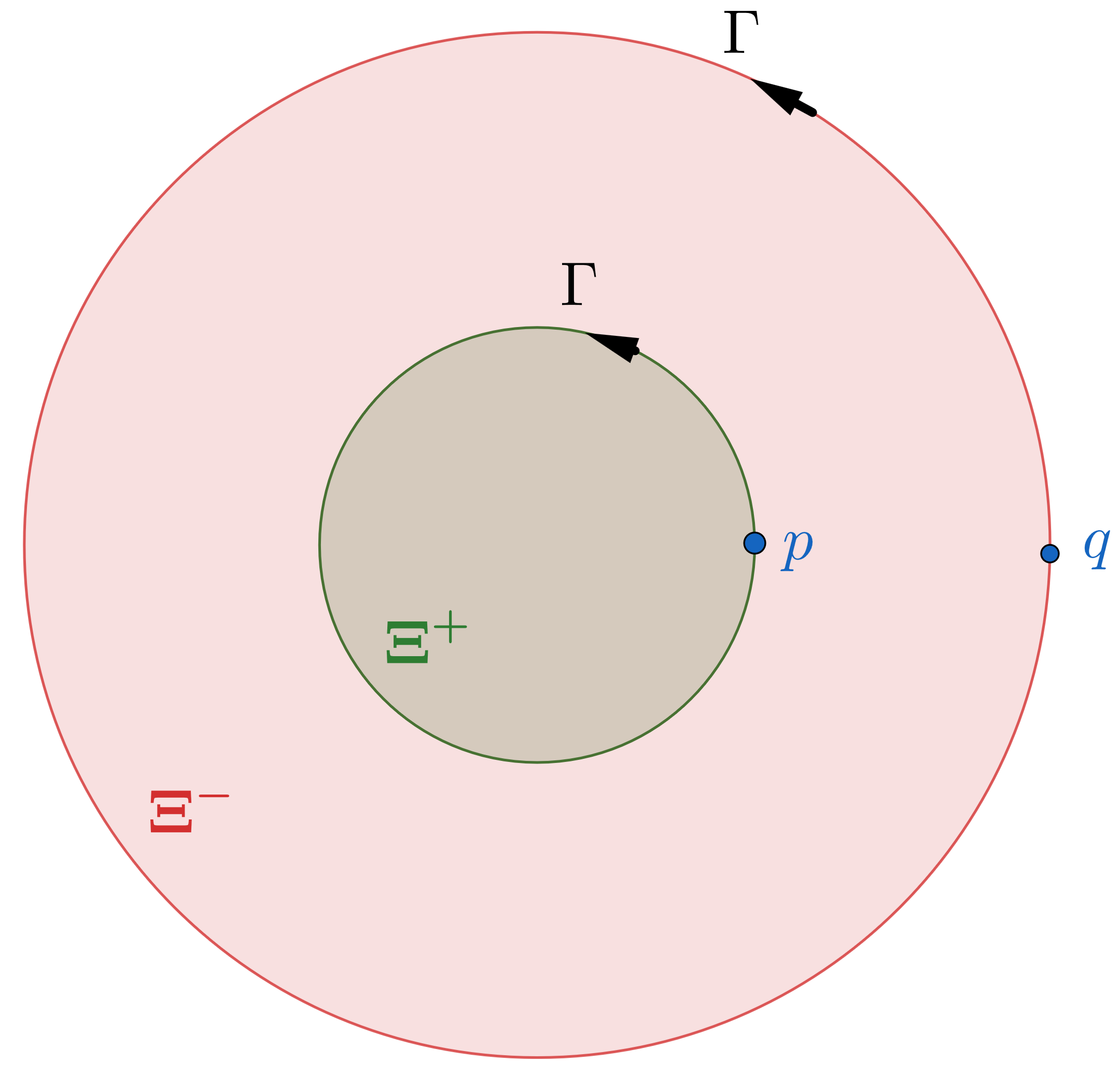}
\caption{If $T$ is given by $2\a{\Xi^+} + \a{\Xi^-}$, then $p\in\Gamma$ is \emph{two-sided} and $q$ is \emph{one-sided}.}
\label{F: 1 and 2sided}
\end{subfigure}
\end{figure} 

\begin{figure}[H]
    \centering
    \includegraphics[width=0.5\linewidth]{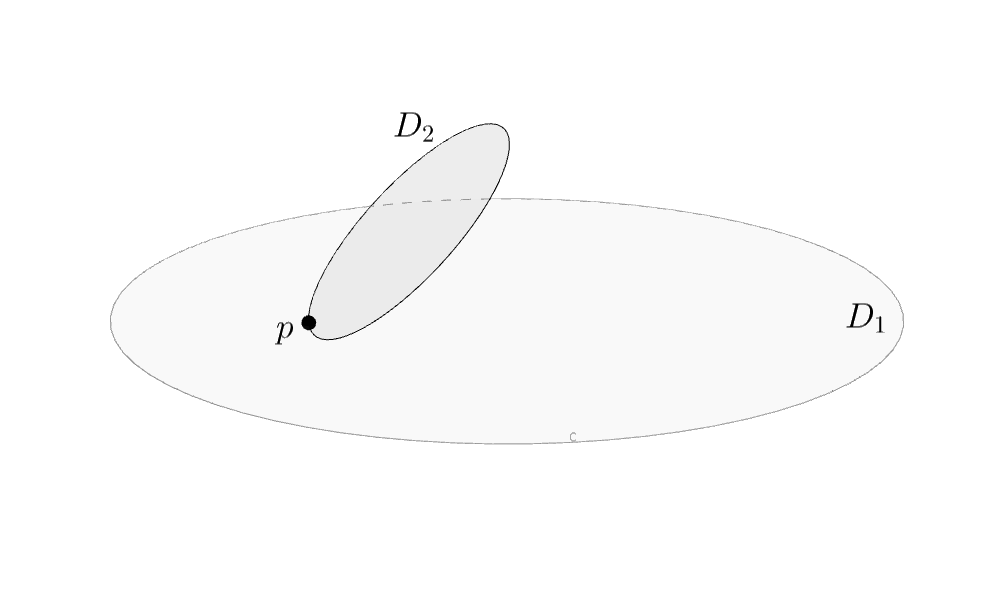}
    \caption{If $T$ is given by $\a{D_1} + \a{D_2}$, then $p$ is neither a one-sided nor a two-sided point.}
\end{figure}

Regular one-sided points have been studied in \cite{AllB} for $Q=1$ and in \cite{delellis2021allardtype} for $m=2$ and regular two-sided points have been explored in \cite{DDHM} for $Q=1$ and in \cite{nardulli2022density}. The two notions of regular boundary points in an arbitrary boundary multiplicity setting are defined below.

\begin{definition}[Regular one-sided boundary points]\label{D:Reg-1sided}
Let $T \in\mathbf{I}_m(U)$ with $\spt(T)\subset\Sigma$, $\partial T={Q}\a{\Gamma}, {Q}\in\N\setminus\{0\}$. Then $p\in\Gamma$ is called an {\emph{one-sided regular boundary point}} if there are:
\begin{enumerate}[(i)]
    \item a finite number $\Sigma_{1}, \ldots, \Sigma_{J}$ of oriented embedded $m$-submanifolds of $U\cap\Sigma$,
    
    \item and a finite number of positive integers $k_{1}, \ldots, k_{J}$ such that:
        \begin{enumerate}
            \item $\partial \Sigma_{j} \cap U=\Gamma \cap U=\Gamma_{i} \cap U$ (in the sense of differential topology) for every $j$,
            
            \item $\Sigma_{j} \cap \Sigma_{l}=\Gamma \cap U$ for every $j \neq l$,
            
            \item for all $j \neq l$ and at each $q \in \Gamma$ the tangent cones to $\Lambda_{j}$ and $\Lambda_{l}$ are distinct,
            
            \item $T \res U=\sum_{j} k_{j} \a{\Sigma_{j}}$ and $\sum_{j} k_{j}=Q$.            
        \end{enumerate}
\end{enumerate}
The set $\Regb^1(T)$ of {\emph{one-sided regular boundary points}} is a relatively open subset of $\Gamma$ and the set of \emph{one-sided singular boundary points} is defined as 
\[\Singb^1(T) := \Singb(T) \cap \left\{q\in\Gamma:\Theta^m(T,q)=Q/2 \right \}.\]
\end{definition}

\begin{definition}[Regular two-sided boundary points]\label{D:2sided regular}
Let $T \in\mathbf{I}_m(U)$ with $\spt(T)\subset\Sigma$, $\partial T={Q}\a{\Gamma}, {Q}\in\N\setminus\{0\}$. Then we say that $p\in\Gamma$ is a {\emph{two-sided regular boundary point for $T$}} if there exist a neighborhood $U \ni p$ and a smooth manifold $\Xi \subset U \cap \Sigma$ such that $\spt(T) \cap U = \Xi$. The set of such points will be denoted by $\Regb^2(T)$. 
\end{definition} 

\begin{remark}\label{R:2sided regular definition}
We record the following consequence of the Constancy Lemma (\cite[Theorem 2.34]{simon2014introduction}) and the definition above (see also Picture (B) in \Cref{D:1-2-sided}). Let $p\in\Regb^2(T)\subset \Gamma$, and let $U$ and $\Xi$ be as in the definition above. Then, there exists $Q^\star\in \N\setminus \{0\}$ such that $\Gamma\cap U$ divides $\Xi$ into two disjoint regular submanifolds, $\Xi^+$ and $\Xi^-$, of $\Xi$ with boundaries $(Q+Q^\star)\a{\Gamma}$ and $-Q^\star\a{\Gamma}$, respectively. Moreover, we have
\[T\res U = (Q+Q^\star)\a{\Xi^+} + Q^\star\a{\Xi^-}.\]    
\end{remark}

We know collect all singular and regular points in the definition below.

\begin{definition}[Singular points]
Let $T \in\mathbf{I}_m(U)$ with $\spt(T)\subset\Sigma$, $\partial T={Q}\a{\Gamma}, {Q}\in\N\setminus\{0\}$. We denote the set of regular boundary points by $\Regb(T):=\Regb^1(T)\mathring{\cup}\Regb^2(T)$ and the set of singular boundary points by $\Singb(T):= \Gamma\setminus (\Regb^1(T)\mathring{\cup}\Regb^2(T))$.
\end{definition}
\begin{remark}
Notice that we only have $\Singb^1(T)\subset \Singb(T)$.
\end{remark}

\subsection{Main results of this work}

Next, in this subsection, the central assumptions and main results of this work will be stated.

\begin{assumption}\label{A:general}
Let $\kappa \in (0,1]$ and integers $m,n\geq 2$, $\bar{n}\geq 1$, $Q\geq 1$. Consider $\Sigma$ an embedded $(m+\bar{n})$-manifold of class $C^{3,\kappa}$ in $\R^{m+n}$, $\Gamma$ a $C^{3,\kappa}$ embedded oriented $(m-1)$-submanifold of $\Sigma$ without boundary, and assume $p\in\Gamma$. Let $T$ be an integral $m$-dimensional area minimizing current in $\Sigma\cap\ball{p}{R}$ with boundary $\partial T \res \ball{p}{R}=Q\a{\Gamma \cap \ball{p}{R}}$.
\end{assumption}

We will prove a version of the boundary regularity theorem obtained in by Allard in \cite{AllB} (for $Q=1$ and only allowing for \emph{one-sided} points) on a higher multiplicity boundary, that is, $Q>1$. We call $q \in \Gamma$ a \emph{flat} boundary point if the tangent cone to $T$ at $q$ is supported in a half $m$-plane (cf. \Cref{D:flat bdr points}).

\begin{resultado}\label{T:1sided-flat-regular}
Under \Cref{A:general}, assume that $q\in\Gamma$ is a one-sided \emph{flat} point of $T$ and there is a neighborhood of $q$ in $\Gamma\cap\ball{p}{r}$ containing \emph{only} one-sided \emph{flat} points, then $q$ is a regular point.
\end{resultado}

\Cref{T:1sided-flat-regular} is close to the best one can hope for, i.e., it is not possible to get an entire counterpart of \cite{AllB} for $Q>1$, since one-sided singular points may exist. In fact, in \cite{ian2024example}, the first named author provides an example of this behavior which is stated below.

In order to state it, define a point $p\in\Gamma$ to be an \textit{essential one-sided singularity} for $T$ if $\Theta(T,0)=Q/2$ and $T$ cannot be written as $T \res \bB_{\rho}=\sum Q_i \a{\Sigma_i}$, where $\Sigma_i$ are smooth minimal surfaces and $\partial \Sigma_i=\Gamma \cap \bB_{\rho}$, for any radius $\rho>0$.

\begin{theorem}\label{T:ian-example}
There exists a $3$d area minimizing current $T$ in $\bB_1\subset\R^5$ (i.e., $n=2$) with $T$ and $\Gamma$ satisfying \Cref{A:general} with $\Gamma$ analytic (i.e. $\Gamma$ is a real-analytic submanifold of $\R^5$), and $Q=2$, such that $\Theta^3(T,p)= Q/2 = 1$ for every $p \in \Gamma \cap \bB_1$ and $0$ is an essential one-sided boundary singularity. Moreover, the tangent cone to $T$ at $0$ is flat.
\end{theorem}
\begin{remark}
As observed in \Cref{R:1sided set is open}, the set of one-sided points is open. This example then demonstrates that, in \Cref{T:1sided-flat-regular}, it is not possible to remove the assumption that \emph{all} points in the open neighborhood of $q$ in $\Gamma$ are \emph{flat}.
\end{remark}

In light of this example, the natural follow-up question to be asked is: What can we say about the boundary singular set? In fact, this question is posed in \cite{de2022regularity} and is fully answered in the subsequent theorems.
 
\begin{resultado}\label{T:Rectifiability-1sided}
Under \Cref{A:general}. If $m\geq 3$, the following holds:
\begin{equation*}
    \Singb^1(T) \mbox{ is }\cH^{m-3}\mbox{-rectifiable}.
\end{equation*}
In particular, if $m=3$, we have that $\Singb^1(T)$ is countable. Moreover, if $m=2$, we have $\Singb^1(T) = \varnothing$.
\end{resultado}
\begin{remark}
The fact that there are \emph{no} one-sided singular points in the two dimensional case is proved in \cite{delellis2021allardtype}.
\end{remark}

Recalling that in the interior setting, the Hausdorff dimension of the singular set is bounded by $m-2$, we observe that in the first dimension where singularities appear, i.e., $m=2$, the singular set is discrete. As proved in the series of papers \cite{DSS1,DSS2,DSS3,DSS4}, addressing this problem is \emph{considerably} more intricate, requiring the introduction of several new techniques.

In our boundary setting, singularities first appear for $m=3$, as shown in \Cref{T:ian-example}, and we establish the countability of the one-sided singular set in \Cref{T:Rectifiability-1sided}. The problem that \Cref{T:Rectifiability-1sided} leaves open is the following.

\begin{oq}
Under \Cref{A:general}. If $m= 3$, is $\Singb^1(T)$ a discrete set? 
\end{oq}

We do not pursue further details on this question here, but it is worth noting that even in the linearized setting (i.e., for $Q$-valued Dirichlet minimizers with collapsed and regular boundary data), the problem remains entirely open.

A consequence of \Cref{T:Rectifiability-1sided}, along with streamlined versions of the theorems by the second named author and Nardulli in \cite{nardulli2022density}, is that the whole regular boundary set is open and dense.

\begin{resultado}\label{T:open-dense}
Under \Cref{A:general}, it follows that
\begin{equation*}
    \Regb(T)\mbox{ is an open and dense subset of }\Gamma\cap\ball{p}{R}.
\end{equation*}
\end{resultado}

As mentioned in the introduction, the theorem above is optimal in the sense that there cannot be a Hausdorff dimensional estimate for the singular set in this general context. Indeed, as shown in \cite[Theorem 1.8 (b)]{DDHM}, we have the following example.

\begin{theorem}\label{T:4names example}
There exists a $2d$ area minimizing current $T$ in $\oball{1}\subset \R^4$ with a boundary $\Gamma$ (consisting of two disjoint closed $C^\infty$ curves) and $\partial T = \a{\Gamma}$, i.e., $Q=1$, such that the Hausdorff dimension of the boundary singular set equals $1$ and $\cH^1(\Singb(T)) =0$.
\end{theorem}

Although this example has a ``large" singular set, it remains an open problem to understand, even when $Q=1$, the following question.

\begin{oq}
Under \Cref{A:general}, is it true that $\Singb(T)$ is $\cH^{m-1}$-null?
\end{oq}

This was originally conjectured in \cite[Conjecture 2.8]{DDHM}; however, it remains a widely open question, as \Cref{T:4names example} is known not to generalize in a way that produces a singular set of positive measure. Even the counterpart of this question in the linearized problem (i.e., for $(Q-Q^\star/2)$-valued Dirichlet minimizers with regular interface) remains completely open.

%\begin{ack*} \emph{
\subsection{Acknowledgments} Both authors are grateful to Camillo De Lellis for his support throughout the project, for encouraging their collaboration, and for providing feedback on an earlier version of this manuscript.%}

%\emph{
Parts of this work were carried out during a two-week visit by I. Fleschler to the Center for Nonlinear Analysis at Carnegie Mellon University and during a one-week visit by R. Resende to the Institute for Advanced Study, with the support of Paul Minter's Clay Research Fellowship.%}
%\end{ack*}

\section{Preliminaries}

\subsection{Basic definitions and notation}\label{S:preliminaries}

Denote \(\ball{p}{r}: = \{x\in\R^{m+n}:|x-p|<r\}\), fix \(\pi_0 := \R^m\times\{0\}\subset\R^{m+n}\) and \(\bp := \bp_{\pi_0}\). For any linear subspace \(\pi \subset \R^{m+n}, \pi^{\perp}\) is its orthogonal complement in \(\R^{m+n}\), \(\bp_\pi\) is the orthogonal projection on \(\pi\). The tilted disk is defined as \(\mathrm{B}_r(p,\pi) := \ball{p}{r}\cap(p+\pi)\) and the tilted cylinder as the set \[\mathbf{C}_r(p,\pi):= \left\{(x+y): x \in \mathrm{B}_r(p,\pi), y \in \pi^{\perp}\right\}.\]

Set \(\cyl{p}{r}:= \mathbf{C}_r(p,\pi_0)\) and \(\bball{p}{r}:= \mathrm{B}_r(p,\pi_0)\). The half-ball will be defined as
\[ \mathrm{B}^+_r(p) := \bball{p}{r}\cap \{x\in \R^m: x_m\geq 0\}. \]

Moreover, $\obball{r}:= \bball{0}{r}$, $\mathrm{B}_r^+:= \mathrm{B}^+_r(0)$, and \(\ocyl{r}:=\mathbf{C}_r(0)\). For any vector subspace of $\R^{m+n}$, define its $\varepsilon$-neighborhood by \[\ball{V}{\varepsilon}:=\{x\in\R^{m+n}: \dist(x,V) <\varepsilon\}.\]

It will always be assumed that each linear $k$-dimensional subspace \(\pi\) of \(\R^{m+n}\) is oriented by a \(k\)-vector \(\vec{\pi}:=v_1 \wedge \cdots \wedge v_k\), where \((v_i)_{i\in\{1,...,k\}}\) is an orthonormal base of \(\pi\) and, with abuse of notation, we write \(\left|\pi_2-\pi_1\right|\) standing for \(\left|\vec{\pi}_2-\vec{\pi}_1\right|\), where \(|\cdot|\) is the norm associated with the canonical inner product of \(k\)-vectors. 

Denote the $k$-dimensional Hausdorff measure in $\R^{m+n}$ by $\cH^k$. For notation and basic results on multivalued functions, the reader is referred to \cite{DS1}, here $\cA_Q(\R^{m+n})$ denotes the set of $Q$-points (i.e., the sum of $Q$ Dirac masses) in $\R^{m+n}$ for $Q\geq 1$.

Let us fix some notation that will be used throughout this work regarding currents and refer the reader to the treatise \cite{Fed}. The space of $k$-dimensional integral currents in $U\subset\R^{m+n}$ is denoted by $\mathbf{I}_k(U)$ for $k\in\{1,\ldots, m+n\}$. Define the upper and lower $d$-density of a nonnegative Borel measure at $p\in\R^{m+n}$ as follows:
\begin{equation*}
    \Theta^{d,*}(\mu,p) := \limsup_{r\to 0}\frac{\mu(\ball{p}{r})}{\omega_d r^d},\mbox { and }\Theta^{d}_*(\mu,p) := \liminf_{r\to 0}\frac{\mu(\ball{p}{r})}{\omega_d r^d}.
\end{equation*}
Whenever the limit exists, the $d$-density of $\mu$ at $p$ is $\Theta^{d}(\mu,p) = \Theta^{d,*}(\mu,p) = \Theta^{d}_*(\mu,p)$, where $\omega_m := \cH^m(\obball{1})$. If $T\in\mathbf{I}_m(U)$, $p\in U$, and $\|\cdot\|$ denotes the mass of the current $T$, we denote $\Theta^{m}(T,p) := \Theta^{m}(\|T\|,p)$. 

A measure $\mu$ is said to be $d$-upper regular if for some universal constant $C$, $\mu(\ball{p}{r}) \leq Cr^d$ for every ball $\ball{p}{r}.$  

Given a $k$-rectifiable set $M\subset\R^{m+n}$, it naturally induces a $k$-current which is denoted by $\a{M}$. The boundary of a $m$-current $T$ is taken with multiplicity $Q^{\star}$ if $T=Q^{\star}\a{\Gamma}$ for some $(m-1)$-rectifiable set $\Gamma$. 

We denote by $\bA_M$ the $C^0$-norm of the second fundamental form of a smooth submanifold $M$ of $\R^{m+n}$.

The classical definitions for the excess and height that will be throughly used in this work are given below. Let $T\in\mathbf{I}_m(U)$ with $\spt(T)\subset\Sigma$, $\Gamma\subset\Sigma$, $\partial T = Q\a{\Gamma}$, $\pi$ a $m$-dimensional plane, $p\in \R^{m+n}$, and $r>0$. The \emph{cylindrical excess relative to the plane $\pi$} is the quantity 
\begin{equation*}
    \bE (T, \cyl{p}{r}, \pi) := \frac{1}{\omega_m r^m} \int_{\cyl{p}{r}} \frac{|\vec{T} (x) - \vec{\pi}|^2}{2} \, d\|T\| (x),
\end{equation*} 
and, for $p\in\Gamma$, the {\emph{cylindrical excess}} the quantity 
\begin{equation*}
    \bE^{\flat} (T, \cyl{p}{r}) := \min \{\bE (T, \cyl{p}{r}, \pi): T_p\Gamma\subset \pi\subset T_p \Sigma\}   .
\end{equation*}

Define \emph{spherical excess relative to the plane $\pi$} to be the quantity 
\begin{equation*}
    \bE (T, \ball{p}{r}, \pi) := \frac{1}{\omega_m r^m} \int_{\ball{p}{r}} \frac{|\vec{T} (x) - \vec{\pi}|^2}{2} \, d\|T\| (x),
\end{equation*}
and, for $p\in\Gamma$, the {\emph{spherical excess}} to be the number
\begin{equation*}
    \bE^{\flat} (T, \ball{p}{r}) := \min \{\bE (T, \ball{p}{r}, \pi): T_p\Gamma\subset \pi\subset\R^{m+n}\} .   
\end{equation*}

The \emph{height of $T$ in a set $G \subset \mathbb{R}^{m+n}$ with respect to a plane $\pi$} is defined as 
\begin{equation*}
    \mathbf{h}(T, G, \pi):=\operatorname{diam}(\bp_{\pi}^\perp(\spt(T)\cap G)) = \sup \left\{\left|\mathbf{p}_{\pi}^{\perp}(q-p)\right|: q, p \in \operatorname{spt}(T) \cap G\right\},
\end{equation*}
where $\mathbf{p}_{\pi}^{\perp}$ denotes the orthogonal projection onto $\pi^\perp$.

\subsubsection{Tangent cones.} For $x\in\R^{m+n}$ and $r>0$, define $\iota_{x, r}(y):= \frac{y-x}{r}$. For any current $T\in\mathbf{I}_m(U)$, let us define the rescaled current $T$ at $x$ at scale $r$ as ${\iota_{x, r}}_{\sharp} T=:T_{x,r}$ and $T_r =: {\iota_{0,r}}_{\sharp} T$. The current $T_{x}$ is called a blow-up of $T$ at $x$, if there exists a sequence of radii $r_j\to0$ such that $T_{x,r_j}\to T_x$ in the weak topology. A well-known consequence of the monotonicity formula tells us that if $T$ is area minimizing in $\Sigma$ and $\partial T = Q\a{\Gamma}$ for $\Gamma\in C^{1,\kappa}$, then $T_{p,r_j}$ converges to an oriented cone about $p$ for any $r_j\downarrow 0$. An oriented cone about $p$ is understood as a current $C$ such that $C_{p,r} = C$ for any $r>0$. 

Some classes of cones obtained by limits of area-minimizing currents will be of special interest in what follows; essentially, they are: one-sided flat cones (half $m$-planes), two-sided flat cones ($m$-planes), and open books (a union of half $m$-planes). Their precise definitions are given below.

\begin{definition}[Open books]\label{D:open-book}
We say that $C$ is an {\emph{open book with boundary $\a{V}$ and multiplicity ${Q}$}}, if $\partial C = {Q}\a{V}$ and there exist $N\in\N\setminus\{0\}$, ${Q}_1,\dots, {Q}_N\in\N\setminus\{0\}$ and $\pi_1,\dots, \pi_N$ distinct $m$-dimensional half-planes such that 
\begin{enumerate}[\upshape (a)]
\item $\partial\a{\pi_i} = \a{V}, \forall i\in\{1, \dots, N\}$,

\item $T_0 = \sum_{i=1}^{N}{Q}_i\a{\pi_i}$ with ${Q} = \sum_{i=1}^{N}{Q}_i$.
\end{enumerate}
If there exist $i\neq j$ such that $\pi_i\neq\pi_j$, we say that $C$ is a {\emph{genuine open book with boundary $\a{V}$ and multiplicity $Q$}}.
\end{definition}

One-sided boundary flat cones will also be called \emph{non-genuine open books}.

\begin{definition}[Flat cones]\label{def:flat-cone}
We say that $C$ is a {\emph{flat cone with boundary $\a{V}$ and multiplicity ${Q}$}}, if $\partial C = {Q}\a{V}$ and there exist a $m$-dimensional closed plane $\pi$, $Q^{\textup{int}}\in\N$ and $Q,Q^{\star}\in\N\setminus\{0\}, Q\le Q^\star$, such that 
\begin{enumerate}[\upshape (a)]
    \item $\spt(C) = \pi$ is an $m$-dimensional subspace,

    \item $\partial\a{\pi^+} = -\partial\a{\pi^-} = \a{V}$,

    \item $C = Q^{\textup{int}}\a{\pi} + Q^\star\a{\pi^+} + (-1)^k({Q^\star}-Q)\a{\pi^-}$, $k\in\{0,1\}$.
\end{enumerate}
If $Q^{\textup{int}}=0$ and either $Q=Q^\star$ or $k=1$, we call $C$ an {\emph{one-sided boundary flat cone with multiplicity $Q$}}. If $k=0$ and $C$ is not a one-sided boundary flat cone, we say that $C$ is a {\emph{two-sided boundary flat cone with multiplicity $Q$}}.
\end{definition}

For $x\in\R^{m+n}$ and $r>0$, define $\iota_{x, r}(y):= \frac{y-x}{r}$ and, for any $T\in\mathbf{I}_m(U)$, the rescaling of $T$ at $x$ at scale $r$ as ${\iota_{x, r}}_{\sharp} T=:T_{x,r}$ and $T_r =: {\iota_{0,r}}_{\sharp} T$.

\begin{definition}\label{D:flat bdr points}
Let $T\in\mathbf{I}_{m}(\R^{m+n})$ with $\spt(T)\subset\Sigma$, $\partial T=Q\a{\Gamma}$, and $p\in\Gamma$. Assume that $r_j\downarrow 0$ and $T_{p,r_j}$ weakly converge to an oriented cone $C$ with $\partial C = Q\a{T_p \Gamma}$. We say that 
\begin{enumerate}[\upshape (a)]
    \item $p$ is a {\emph{boundary flat point of $T$}} provided $C$ is a flat cone,
    
    \item $p$ is a {\emph{one-sided boundary flat point of $T$}} provided $C$ is an open book which is non genuine,
    
    \item $p$ is a {\emph{two-sided boundary flat point of $T$}} provided $C$ is a two-sided boundary flat cone.
\end{enumerate}
\end{definition}

\subsubsection{Jones' $\beta_2$ coefficients.}\label{S:Jones beta numbers} We define the Jones' $\beta_2$ coefficient that measures how close the support of $\mu$ is to a plane in $L^2(\mu)$ as follows.

Given a finite non-negative Radon measure $\mu \in \RR^{m+n}$ and $k \in \left\{1,...,m+n-1 \right\}$, we define the \emph{$k$-dimensional Jones' $\beta_2$ coefficient of $\mu$} as 
\begin{equation*}
    \beta_{2,\mu}^{k}(x,r):= \inf_{\textup{affine k-planes }L} \left[r^{-k-2} \int _{\ball{x}{r}}  \dist(y,L)^2 d\mu(y)\right]^{1/2}.
\end{equation*}

We will often omit $\mu$ from the notation.

\subsection{Outline of the proofs}

The results established by the first-named author in \cite{ian2024uniqueness} are pivotal for the development of this article, and they are presented here for completeness. Although these results were originally proven under more general assumptions, we have opted to state them in a specific form that aligns with our current application.

\begin{theorem}[Uniqueness of the tangent cone]\label{T:Uniqueness 1sided}
Assume \Cref{A:general} and, in addition, $\Theta(T,p)=Q/2$. The tangent cone at $p$ is an open book and is unique, i.e. it does not depend on the blow-up sequence. Moreover, $T_{p,r}$ converges (in flat distance, cf. \cite[\textsection 7, Chapter 6]{simon2014introduction}) as $r\to 0$ to the unique open book faster than $r^\alpha$ for some $\alpha\in (0,1)$. %has a power rate of convergence to the unique tangent cone. 
\end{theorem}

The following excess decay, whose proof is essentially contained in \cite{ian2024uniqueness}, will be repeatedly used. 

\begin{theorem}\label{T:excessdecay}
Assume \Cref{A:general} and, in addition, $\Theta(T,p)=Q/2$. Assume that $T$ has a flat one-sided tangent cone $Q\a{\pi}$ at $p$. There exists $\varepsilon(Q,m,n)>0$ such that if
\begin{equation*}
\max\left\{\mathbf{E}(T,\bB_R(p),\pi), \mathbf{A}_{\Gamma},\mathbf{A}_{\Sigma}^2\right\}<\varepsilon
\end{equation*}
then there exists $C(Q,m,n)$ and $\alpha(Q,m,n)>0$ such that
\begin{equation*}
\mathbf{E}(T,\bB_r(p),\pi) \leq C(Q,m,n) r^{\alpha}, \mbox{ for every }r\in (0,R).
\end{equation*}
\end{theorem}

To investigate the $\cH^{m-3}$-rectifiability of the singular set and prove \Cref{T:Rectifiability-1sided}, the first step is to employ a suitable decomposition of $T$ (see \Cref{T:decomposition-1sided}), which allows us to split the current into different sheets. Each of these sheets has a boundary multiplicity smaller than that of $T$ and does not share the same tangent cone.

Using this decomposition, the problem of proving the $\cH^{m-3}$-rectifiability of the set of one-sided singular points (those with density $Q/2$) is reduced, in \Cref{S:main-results-proofs}, to establishing the $\cH^{m-3}$-rectifiability of the subset of one-sided \emph{flat} singular points (which possess at least one tangent cone supported on a plane, cf. \Cref{D:flat bdr points}). Relying on the result of the first named author and De Lellis \cite{iancamillorectifiability}, one can reduce the problem of proving rectifiability of a set to showing the rectifiability of Frostman (upper-regular) measures supported on it. Then, an application of Azzam and Tolsa's main theorem of \cite{azzam2015characterization} (see \Cref{T:NaberValtorta-AzzaamTolsa})  reduces the proof of $\HH^{m-3}$-rectifiability to establishing a Dini-type condition for the Jones' $\beta_2^{m-3}$ coefficient for Frostman measures supported on pieces of the set of one-sided flat singular points.

To analyze one-sided flat singular boundary points, we utilize the standard machinery to address regularity in this setting—namely, the study of Lipschitz approximations, center manifolds, and normal approximation (cf. \Cref{P:1sided-StrongLA} and \Cref{T:CM,T:LocalNA})—adapting the approach used for $m=2$ in \cite{delellis2021allardtype}. As observed in \cite{de2023fineI,de2023fineII,de2023fineIII} (for flat interior singular points), the ``classical" estimates in \cite{DS4,DS5} are insufficient. The same holds in the boundary case: the estimates in \cite{delellis2021allardtype} are not sharp enough to yield the necessary bounds. To address this, we provide finer estimates in \Cref{S:UniversalFreqDef} for quantities related to the universal frequency function (\Cref{S:UniversalFreqDef}), improving on those in \cite{delellis2021allardtype,DDHM,nardulli2022density}.

When working with the universal frequency function, we must consider the intervals of flattening, following the approach in \cite{DS5,delellis2021allardtype}. In contrast, in boundary settings considered in \cite{DDHM} and in the work of the second named author and Nardulli \cite{nardulli2022density}—where they study collapsed two-sided points—there is no need to introduce intervals of flattening, as prior results ensure an almost quadratic excess decay rate. However, for one-sided points, the excess decay rate established in \cite{ian2024uniqueness} is certainly not almost quadratic, and there is no reason to expect such a rate in general.

The analysis of the fine properties of the universal frequency function begins by establishing the almost monotonicity of the frequency (see \Cref{T:AlmMonotFreq}), with errors quantitatively controlled by the excess and the norms of both the second fundamental form of the boundary and the second fundamental form of the ambient manifold. In contrast to \cite{de2023fineI}, by \cite{ian2024uniqueness}, it is already known that every boundary one-sided point has a unique tangent cone, which serves as the starting point for the regularity theory. The frequency values $\bI(T,p)$ are then established in \Cref{T:UniversalLowerBound-FreqFnc} to be always bounded below by a universal geometric constant $1+ \alpha/2 >1$, where $\alpha$ is the exponent obtained from the excess decay result in \cite{ian2024uniqueness}.

Although the presence of the boundary introduces both technical and conceptual challenges, we manage to follow the approach in \cite{de2023fineII} in an analogous way. Taking into account the lower bound for the frequency values of $1+\alpha/2 > 1$, the procedure from \cite{de2023fineII} can be implemented to bound Jones' coefficient $\beta_2^{m-3}$ (for certain Frostman measures) in terms of frequency pinching, up to an error (see \Cref{T:JonesBeta-LowFreq,T:JonesBeta-HighFreq}).

The bounds for $\beta_2^{m-3}$ mentioned above and established \Cref{T:JonesBeta-LowFreq,T:JonesBeta-HighFreq} are a consequence of the classification of homogeneous $2$-dimensional solutions for the linear problem, proved in \cite{delellis2021allardtype}. This classification asserts that every $2$-dimensional homogeneous multivalued Dirichlet minimizer with zero boundary data on a flat boundary must be linear (see \Cref{t:homogeneous2d}).

In fact, if the set of flat one-sided singular points is not sufficiently close to an $(m-3)$-dimensional plane (i.e., $\beta_2^{m-3}$ is ``large"), we can construct a solution to the linear problem that depends only on \emph{two} variables and has a frequency function at least $1+\alpha/2>1$ (see \Cref{T:JonesBeta-HighFreq}). This contradicts the aforementioned classification, which implies that the frequency must be exactly $1$ (\Cref{t:homogeneous2d}). Consequently, we obtain the desired control on the Jones' $\beta_2^{m-3}$ coefficient (up to an error) via the frequency pinching (\Cref{T:JonesBeta-LowFreq,T:JonesBeta-HighFreq}).

Another key ingredient in obtaining the bound mentioned above in $\beta^{m-3}_2$ is achieving a precise understanding of the frequency values along different center manifolds—i.e., the fine monotonicity formula for the universal frequency—and, moreover, being able to compare the frequency values at different points at the same scale. These properties have been absent in previous works that addressed boundary regularity, namely \cite{DDHM,delellis2021allardtype,nardulli2022density}.

\subsection{Convex barrier}

The convex barrier condition in Euclidean spaces, defined in the following, ensures the applicability of \Cref{T:Rectifiability-1sided} (see \Cref{T:ConvexBarriers}) for the whole boundary singular set. 

\begin{assumption}[Convex barrier condition] \label{convexbarrier} Let $\Omega \subset \R^{m+n}$ be a domain such that $\partial \Omega$ is a $C^2$ uniformly convex submanifold of $\RR^{m+n}$. The set \emph{$\partial \Omega$ is a convex barrier for $T$} if $\partial T = \sum_{i=1}^N Q_i \a{\Gamma_i}$, $Q_i$ are positive integers, and $\Gamma_i \subset \partial \Omega$ are disjoint $C^2$ closed oriented submanifolds of $\partial \Omega$. 
\end{assumption}

The following notation for wedges will be used from now on.

\begin{definition}
Given an $(m-1)$-dimensional plane $V \subset \R^{m+n}$, we denote by $\mathbf{p}_V$ the orthogonal projection onto $V$. Given a unit vector $\nu$ normal to $V$ and an angle $\vartheta \in\left(0, \pi/2\right)$, the \emph{wedge with spine $V$, axis $\nu$, and opening angle $\vartheta$}, is defined as the set
\begin{equation*}
W(V, \nu, \vartheta):=\left\{y\in \R^{m+n}:\left|y-\mathbf{p}_V(y)- \langle y, \nu\rangle \nu\right| \leq (\tan \vartheta )\langle y, \nu\rangle \right\}.
\end{equation*}
\end{definition}

A local version of the convex barrier condition is defined below and is actually implied by the global version defined above in \Cref{convexbarrier}. 

\begin{assumption}[Local convex barrier condition]\label{a:convexbarrierlocal}
Let $T, \Sigma$ and $\Gamma$ be as in \Cref{A:general}. Assume that there exists $\nu: \Gamma \rightarrow \mathbb{S}^{n+1} \cap T\Sigma$ a map such that $\nu(q) \perp T_q \Gamma$ and $\nu(q) \in T_q \Sigma$. Then \emph{$T$ satisfies the local convex barrier assumption} if 
\begin{equation*}
 \operatorname{spt}(T) \subset \bigcap_{q \in \Gamma}\left(q+W\left(T_q \Gamma, \nu(q), \vartheta\right)\right) \cap \Sigma.
\end{equation*}
\end{assumption}

A much more detailed discussion of both convex barrier conditions has been explored in \cite[Section 4]{ian2024uniqueness} where the first named author proved \Cref{T:ConvexBarriers}, which shows that the theory in the present work can be applied to the convex barrier setting.

\begin{theorem}\label{T:ConvexBarriers}
Let $T$ be as in \Cref{convexbarrier} (convex barrier) or \Cref{a:convexbarrierlocal} (local convex barrier), then every tangent cone to $T$ at a boundary point $q$ is an open book. Moreover, it holds:
\begin{itemize}
    \item if $q \in \Gamma_i$ in the first case, then $\Theta(T,q)=Q_i/2$,

    \item if $q \in \Gamma$ in the second case, then $\Theta(T,q)=Q/2$.
\end{itemize}
\end{theorem}

\subsection{Structural properties}

Although, in general, the regular boundary set is only dense in $\Gamma$ (\Cref{T:open-dense}), certain assumptions may allow us to derive additional information about boundary points. A question originally posed by Almgren and answered, for $Q=1$ in \cite[Theorem 1.9 and Corollary 1.10]{DDHM}, is: is it possible for the set of density $1/2$ points to be empty when $\Gamma$ is connected?

Here, we present the counterpart of these results in the arbitrary boundary multiplicity setting $Q>1$, with proofs provided at the end of this subsection.

\begin{proposition}\label{P: conn components of bdr have good dim bound}
Let $T$, $Q$, $\Sigma$, and $\Gamma$, be as in \Cref{A:general} and $\Gamma^{\prime} \subset \Gamma$ be a connected component of $\Gamma$. In addition, suppose that $\Gamma$ and $\spt(T)$ are compactly supported in $\ball{p}{R}$. If
\[\mbox{there exists }p\in \Gamma^{\prime} \cap \Regb(T)\mbox{ with }\Theta(T, p)>\frac{Q}{2}, 
\]
then there is no one-sided point in $\Gamma^\prime$ and $\dim_H\left(\Singb(T) \cap \Gamma^{\prime}\right)\leq m-2$. Moreover, if $m=2$, then $\Singb(T) \cap \Gamma^{\prime}$ consists of finitely many points.
\end{proposition}
\begin{remark}
In the case $\Sigma=\R^{m+n}$, a portion of $\Gamma$ is contained within the real analytic minimal surface $\Regi(T)$. This condition is quite rigid, and indeed, under reasonable definitions for generic $C^\infty$ boundaries, a generic boundary data will not have any ``two-sided" regular point.
\end{remark}

Next, we establish that the connectedness of the boundary $\Gamma$ guarantees only the existence of one-sided points.

\begin{proposition}\label{P: conn bdr implies one-sidade info}
Let $T$, $Q$, $\Sigma$, and $\Gamma$, be as in \Cref{A:general}. In addition, suppose that 
\[ \Gamma\mbox{ is connected and both }\Gamma\mbox{ and }\spt(T)\mbox{ are compactly supported in }\ball{p}{R}.\]
Then $\Gamma$ only contains one-sided points and, in particular, $\Regb(T)$ is contained in the set of density $\frac{Q}{2}$ points.
\end{proposition} 

Both of the above theorems will follow as consequences of the following structural decomposition argument. This decomposition follows from the density of the boundary regular set (cf. \Cref{T:open-dense}) and provides a decomposition of the area-minimizing current $T$ into its \emph{irreducible} components.

\begin{theorem}[Irreducible components of $T$]\label{T:decomp thm 4names}
Let $T$, $Q$, $\Sigma$, and $\Gamma$, be as in \Cref{A:general}. In addition, suppose that 
\[ \Gamma\mbox{ and }\spt(T)\mbox{ are compactly supported in }\ball{p}{R}.\]
Let $\Gamma_1, \cdots, \Gamma_{L_0}$ be the connected components of $\Gamma$. Then there exist $L\in\N$, $Q_j\in \N\setminus\{0\}$, and \emph{area minimizing} integral $m$-currents $T_j$ satisfying $T = \sum_{j=1}^{L} Q_j T_j$ where, for each $j\in\{1,\ldots,L\}$, we have
\begin{enumerate}[\upshape (i)]
    \item $\partial T_j = \sum_{i=1}^{L_0}\sigma_{ij}\a{\Gamma_i}$ with $\sigma_{ij}\in\{-1,0,1,\cdots, Q\}$,

    \item  $T_j = \a{\Lambda_j}$ with $\Lambda_1, \cdots, \Lambda_L$ being the connected components (accordingly oriented) of $\Regi(T)$. 
\end{enumerate}
Moreover, for each $i\in\{1,\cdots, L_0\}$, only one of the two following alternatives holds:
\begin{enumerate}
    \item[\upshape (iii)]
    $\Gamma_i$ is one-sided, i.e. there is $J_i \subset \{1, \dots, L\}$ such that $\sigma_{ij} \geq 1$ for all $j \in J_i$, $\sigma_{ij} = 0$ for all $j \notin J_i$, and $\sum_{j \in J_i} \sigma_{ij}Q_j = Q$. Moreover, it holds
    \[ \Gamma_i\cap \Regb(T)\subset \Regb^1(T).\]
    
    \item[\upshape (iv)] 
    $\Gamma_i$ is two-sided, i.e. there exist $p(i), n(i)\in\{1,\cdots, L\}$ such that $\sigma_{ip(i)} = - \sigma_{in(i)} = 1$ and $\sigma_{ij}=0$, for all $j\notin \{ p(i),n(i) \}$. 
    
    In this case, $Q_{p(i)} - Q_{n(i)} = Q$ and $T_{p(i)}+T_{n(i)}$ is area minimizing. Moreover, it holds
    \[ \Gamma_i\cap \Regb(T)\subset \Regb^2(T).\]
\end{enumerate}
\end{theorem}

\begin{proof}[Proof of \Cref{T:decomp thm 4names}]
We omit the proof of this theorem, as it follows \cite[Theorem 2.1]{DDHM} mutatis mutandis, but now relying on \Cref{T:open-dense}.
\end{proof}

\begin{proof}[Proof of \Cref{P: conn bdr implies one-sidade info}]
The proof follows \cite[Corollary 1.10]{DDHM} with a minor modification that we describe now. Since $\Gamma$ is connected, we have $L_0 = 1$ in \Cref{T:decomp thm 4names} and, thus, the decomposition therein consists of \emph{exactly} two currents $T_{p(i)}$ and $T_{n(i)}$ in case $\Gamma$ is two-sided (as in \Cref{T:decomp thm 4names}(iv)). However, in this case, $S:= (Q_{p(i)} - Q_{n(i)})T_{p(i)}$ would satisfy $\partial S = Q\a{\Gamma}$ and, by \Cref{T:decomp thm 4names}(ii), the mass of $S$ would be strictly smaller than the mass of $T$ which contradicts the minimality of $T$. Hence, the only possibility is that $\Gamma$ is one-sided (as in \Cref{T:decomp thm 4names}(iii)).
\end{proof}
\begin{proof}[Proof of \Cref{P: conn components of bdr have good dim bound}]
The proof is the same as in \cite[Theorem 1.9]{DDHM}. By \Cref{T:decomp thm 4names}(iv), $\Gamma^\prime$ is necessarily two-sided, therefore $S:=T_{p(1)}+T_{n(1)}$ is area-minimizing. Since all points of $\Gamma^\prime$ are interior points of $S$, by the interior regularity theory, the interior singular set of $S$ has Hausdorff dimension at most $m-2$ (and is discrete if $m=2$, since $\Gamma$ is compactly supported, it has to be finite). The boundary regularity is then a straightforward consequence of \Cref{D:Reg-1sided} and \Cref{R:2sided regular definition}.
\end{proof}

\section{Rectifiability for the linear problem}\label{S:linear-problem}

The goal of this section is to present the rectifiability criteria based on the Jones' $\beta_2$ coefficients (\Cref{T:NaberValtorta-AzzaamTolsa}), establish key results on the linear problem (i.e., Dirichlet minimizers) that are fundamental for the development of the rest of the paper (\Cref{S:dir-min}), and state the rectifiability of the boundary singular set of Dirichlet minimizers (\Cref{T:LinearProblem:Rectifiability}).

The novelty in the Jones' $\beta_2$ coefficient (see \Cref{S:Jones beta numbers} for the definition) lies in the fact that its finiteness is equivalent to the rectifiability of the measure $\mu$, this is a celebrated result proved in \cite{azzam2015characterization}.

\begin{theorem}[Theorem 1.1 of \cite{azzam2015characterization}]\label{T:NaberValtorta-AzzaamTolsa}
Let $\mu$ be a finite Borel measure in $\R^{m+n}$ such that $0<\Theta^{k, *}(\mu,x)<\infty$ for $\mu$-a.e. $x \in \R^{m+n}$. If
\begin{equation*}
    \int_0^1 \beta^k_{2,\mu}(x, r)^2 \frac{\mathrm{d} r}{r}<\infty \quad \text { for } \mu \mbox{-a.e. } x \in \R^{m+n},    
\end{equation*}
then $\mu$ is $d$-rectifiable. 
\end{theorem}

Our main interest is to apply \Cref{T:NaberValtorta-AzzaamTolsa} for Frostman measures $\mu$, as in \cite[Theorem 1.1]{iancamillorectifiability}, supported on pieces of the boundary singular set of an area minimizing current $T$ or a $Q$-valued function. 

Next, we define the frequency function as introduced in \cite[Chapter 4]{DDHM} for boundary points, see also \cite{delellis2021allardtype} and \cite{nardulli2022density}. Let $u:\Omega\to\cA_Q(\R^n)$ and $\Gamma := \partial\Omega$, then define  $I_{u,p}(r)$ is the value of the smoothed frequency function of $u$ at $p\in\Gamma$ at scale $r>0$, i.e., 
\begin{equation*}
\begin{aligned}
    I_{u,p}(r) := \frac{rD_{u,p}(r)}{H_{u,p}(r)},& \ D_{u,p}(r):= \int_{\Omega} |Du(y)|^2\phi\left(\frac{d(y,p)}{r}\right), \\
    &\mbox{ and } H_{u,p}(r) := - \int_{\Omega} \frac{|u(y)|^2}{d(y,p)}\phi^\prime\left(\frac{d(y,p)}{r}\right)|\nabla d(y,p)|^2,
\end{aligned}
\end{equation*}
where the cut-off $\phi$ is chosen as before, i.e.,
\begin{equation*}
    \phi(t) := \begin{cases}
        1, & t\in [0,1/2],\\
        2-2t, & t\in [1/2,1],\\
        0, &\mbox{otherwise},
    \end{cases}
\end{equation*}
and the function $d$ is a good distorted distance function with good asymptotics and such that $\nabla_y d$ is tangent to $\Gamma$, as given in \cite[Lemma 4.25]{DDHM}. Set $I_{u,p}(0):=\lim_{r\downarrow 0}I_{u,p}(r)$ whenever the limit exists.

\subsection{\texorpdfstring{$Q$}{Lg}-valued Dir-minimizers with zero boundary data on a flat boundary}\label{S:dir-min}

We develop a set of tools for the case where the Dirichlet minimizer has zero boundary data on a flat boundary. These results will be utilized throughout the paper. In particular, we establish the following counterpart (to the boundary setting) of \cite[Lemma 5.4]{de2018rectifiability}.

\begin{lemma}\label{l:twovariables}
Let $u$ be a Dir minimizing function in $\obball{s}^+$ with zero boundary data on $\obball{s}\cap\{x\in\R^m:x_m=0\}$ and $\mathrm{B}^+_r(p)\subset \obball{s}^+$, $p\in\obball{s}^+$%, $s>r$
. If $u|_{\mathrm{B}^+_r(p)}$ is a function of only the two variables in $\textup{span}(v, x_m)$ with
$v \in \R^{m-1}$, then $u$ is a function of only $\textup{span}(v, x_m)$ in the whole ball $\obball{s}^{+}$.
\end{lemma}
\begin{proof}
We denote by $\Singi(u)$ the interior singular set of $u$. By interior regularity theory, see \cite{DS1}, we know that $\mathrm{dim}_H(\Singi(u)) \leq m-2$. Thus, $\obball{s}^+ \setminus \Singi(u)$ is a path-wise connected set. Let $V=\textup{span}(v,x_m)$. By assumption, we deduce that $\nabla_{V^{\perp}} u (x) = 0$ for any $x \in \halfbball{p}{r}$. Since $u$ is given by $Q$ harmonic sheets on $\obball{s}^+\setminus \Singi(u)$, the gradient $\nabla_{V^{\perp}} u(x)$ vanishes for any $x\in \obball{s}^+ \setminus \Singi(u)$. Since $u \in W^{1,2}(\obball{s}^+)$, we easily obtain that $\nabla_{V^{\perp}}u \equiv 0$ in $\obball{s}^+$ which ensures that $u$ is a $Q$-valued function which only depends on the variables in $V$.
\end{proof}

We will need to use that, in the $2$-dimensional case, the only possible frequency value for the linear problem is $1$. We state this result in a slightly different form than \cite[Theorem 5.9]{delellis2021allardtype}, assuming homogeneity in our formulation. 

\begin{theorem}[Theorem 5.9 of \cite{delellis2021allardtype}]\label{t:homogeneous2d}
Let $u$ be a nontrivial homogeneous Dirichlet minimizing function on $\obball{s}^+\subset\R^2$ with zero boundary value on $[-1,1]\times \{0\}$ and assume that $\etaa \circ u \equiv 0$. Then $I_u(r) \equiv 1$ and $u$ is linear, i.e., $u = \sum_iQ_i\a{L_i}$ where each $L_i:\R^2\to\R^n$ is linear, $L_i|_{\{x_2=0\}}\equiv 0$, $\sum_iQ_i=Q$, and $L_i\neq L_j$ on $\R^2\setminus\{x_2=0\}$.
\end{theorem}

We use \Cref{t:homogeneous2d} to show a crucial lower bound for the derivatives of Dirichlet minimizing functions with zero boundary data on a flat boundary that will play a crucial role within our program to prove rectifiability.

\begin{lemma}\label{l:linearproblemderivativebound}
Let $u$ be a nontrivial Dir minimizing function in $\obball{s}^+$. Let $\alpha>0$ and $\Lambda>0$. 
Assume that  
\begin{equation*}
    I_u(s)<\Lambda, \ \int_{\obball{s}^+} |u|^2 = 1,\mbox{ and }I_{u}(0) \geq 1+\alpha/2.
\end{equation*}
If $\{v_1,...,v_{m-1}\}$ is an orthonormal basis of $\R^{m-1}$ and $s>r>0$, we obtain that
\begin{equation*}
    \sum_{i=1}^{m-2}\int_{\obball{s}^+\setminus \overline{\mathrm{B}_r^+}}|\partial_{v_i} u|^2 \geq c(\Lambda)>0.
\end{equation*}
\end{lemma}
\begin{proof}
We argue by contradiction, take a contradiction sequence $u_k$ as in the statement and assume that
\begin{equation}\label{e:contradiction-normal-derivatives}
    \sum_{i=1}^{m-2}\int_{\obball{s}^+\setminus \overline{\mathrm{B}_r^+}}|\partial_{v_i} u_k|^2 \leq 1/k.
\end{equation}

Up to subsequences, $\{u_k\}$ converge in $W^{1,2}$ to $u_0$ such that
\begin{equation*}
    \int_{\obball{s}^+} |u_0|^2=1, \int_{\obball{s}^+} |Du_0|^2 \leq C\Lambda, \mbox{ and }I_{u_0}(0) \geq 1+\alpha/2.
\end{equation*}

Additionally, from \eqref{e:contradiction-normal-derivatives}, we also have
\begin{equation*}
    \sum_{i=1}^{m-2}\int_{\obball{s}^+\setminus \overline{\mathrm{B}_r^+}}|\partial_{v_i} u_0|^2 = 0,
\end{equation*}
which, in turn, implies that $u_0$ is a function of only $\textup{span}(v_{m-1},x_{m})$ in $\obball{s}^+\setminus \overline{\mathrm{B}_r^+}$. Noting that the proof of \Cref{l:twovariables} goes through with $\obball{s}^+\setminus \overline{\mathrm{B}_r^+}$ in place of $\halfbball{p}{r}$, we obtain that $u_0$ must be a function of only two variables, namely $\textup{span}(v_{m-1},x_m)$, on $\obball{s}^{+}$. 

To conclude, we take $h$ as a blow-up of $u_0$ at $0$. The function $h$ is a homogeneous Dirichlet minimizing function in $\obball{s}^+$ that is still a function of only the two variables in $\textup{span}(v_{m-1},x_{m})$. Thus, $h$ is a homogeneous Dirichlet minimizer on $\textup{span}(v_{m-1},x_{m})$ with zero boundary value on $\{x_m=0\}$. By \Cref{t:homogeneous2d}, $I_h(0) = 1$. However, since $I_{u_0}(0) \geq 1+\alpha/2$, we obtain $I_h(0)\geq 1+\alpha/2$, which gives a contradiction.  
\end{proof}

We state a lemma that will be useful in some later arguments.

\begin{lemma}\label{L:constantfrequencyannulus}
Let $u$ be a nontrivial Dirichlet minimizer in $\obball{1}^+$ with $u \equiv Q\a{0}$ on $\obball{1} \cap \left\{x_m=0 \right\}$. Suppose further that $I_{u}(s)=I_{u}(t)$ for $0<s<t<1$. Then, the function $u$ is $I_u(s)$-homogeneous in $\obball{1}^+$.
\end{lemma}
\begin{proof}
Since $I_{u}$ is monotone, it must be constant in $[s,t]$. As in \cite[Corollary 3.16]{DS1}, we then have, for almost every $r \in [s,t]$ and almost every $y$, that 
\begin{equation*}
    |y|= u_i(y)=\lambda_r\partial_{\nu}u_i(y),
\end{equation*} 
with $\lambda_r$ being a constant. We recall that the classical frequency $I$ (that is, given the choice $\phi \equiv 1$) is monotone with a flat boundary as in \cite{ian2024uniqueness}. The above equality implies $I(s)=I(t)$, but as in \cite[Corollary 3.16]{DS1}, we obtain that
\begin{equation*}
    I(r)=\frac{r\int_{\partial \obball{r}}\sum_i \langle \partial_{\nu}u_i, u_i\rangle}{\int_{\partial \obball{r}}\sum_i |u_i|^2}= r\lambda_r^{-1}.
\end{equation*}

Thus, $\lambda_r$ must be the same constant for almost every $r$ in $[s,t]$. Set $I_0 := I(s)$. This in turn leads to $I_0 u_i(y)=r\partial_{\nu}u_i(y)$. Defining $\sigma_y:=\left\{ry: s \leq y  \leq t\right\}$, we get  
\begin{equation*}
    (u_i|_{\sigma_y})'(r)= \frac{I_0}{r }u_i|_{\sigma_y}(r).
\end{equation*}

This implies that $u_i|_{\sigma_y}$ is $I_0$-homogeneous in $[s,t]$. By unique continuation in the interior regular set of $u$ (where it agrees with a classical harmonic function), the function $u$ must be $I_0$-homogeneous in $\obball{1}^+$.
\end{proof}

\subsection{Rectifiability for the linear problem}

We now state the boundary rectifiability result for Dirichlet minimizers. The counterpart of this result for the \emph{interior} singular set is stated in \cite{de2018rectifiability} and in \cite{krummel2017fine}.

\begin{theorem}\label{T:LinearProblem:Rectifiability}
Let $u : \Omega \subset \R^m \to \cA_Q(\R^n)$ be a Dirichlet minimizer, $ u|_{\partial\Omega} = Q\a{g}$, $\Omega\in C^2$, and $g\in C^2$. Then, the boundary singular set of $u$ is $\cH^{m-3}$-rectifiable. 
\end{theorem}

The proof of this result follows by repeating the arguments used to establish the rectifiability of the nonlinear problem (\Cref{T:Rectifiability-1sided}), but with many of the error terms absent. Since this proof represents a simpler case of the analysis we carry out later and proceeds in an entirely analogous manner, we omit its details.

\section{Reduction of main theorems to rectifiability of one-sided flat singular points}

\subsection{Excess decay and decomposition for two-sided points with trapped density}

For two-sided points, we establish an excess decay at an almost-quadratic rate and, as a consequence, the uniqueness of tangent cones. This result holds under a density condition (see \eqref{E:density-trapped}) that is satisfied in a relatively \emph{open} neighborhood in $\Gamma$, which is, a priori, more general than the condition assumed in \cite{DDHM,nardulli2022density}. This refinement strengthens previously known results in \cite[Theorems 4.3 and 4.5]{nardulli2022density} and, for $Q=1$, in \cite[Theorem 6.3]{DDHM}. 

Building on this, we derive a decomposition of the current at two-sided points under the same density assumption; see \Cref{P:2sided-decomposition}.

\begin{proposition}[Excess decay and uniqueness of tangent cones at two-sided points]\label{P:2sided-excessdecay+uniqueness}
Let $T,\Sigma$, and $\Gamma$ be as in \Cref{A:general}, $Q^\star\in\N\setminus\{0\}$, and $p_0\in\Gamma$. For any $\varepsilon>0$, there exists a constant $\delta = \delta(\varepsilon,Q,Q^*,m,n,\overline{n}) >0$ with the following property. If, for some $r_0>0$, the following holds
\begin{equation}\label{E:density-trapped}
     \left| \frac{\|T\|\left(\ball{q}{\rho}\right)}{\omega_m \rho^m} - \left( \frac{Q}{2} + Q^\star\right)\right| < \delta, \quad \forall q\in\Gamma\cap\ball{p_0}{r_0}, \forall \rho\in (0,r_0),
\end{equation}
then, for any $q\in\Gamma\cap\ball{p_0}{r_0}$ and any $\rho\in (0,r_0/2)$, the following is true:
\begin{enumerate}[\upshape (i)]
    \item there exists a $m$-dimensional plane $\pi(q)$ with $T_q\Gamma\subset\pi(q)\subset T_q\Sigma$ and 
    \begin{equation*}
        \bE^\flat (T, \ball{q}{\rho}, \pi(q)) = \bE^\flat (T, \ball{q}{\rho}) \leq C \left(\frac{\rho}{r_0}\right)^{2-2\varepsilon}  \bE^\flat (T, \ball{p_0}{r_0}) + C \rho^{2-2\varepsilon}r^{2\varepsilon} \bA^2,
    \end{equation*}

    \item $T_q:=(Q+Q^\star)\a{\pi(q)^+} + Q^\star\a{\pi(q)^-}$ is the \emph{unique} tangent cone to $T$ at $q$ where $\pi(q)^+$ and $\pi(q)^-$ are the two connected components of $\pi(q)\setminus T_q\Gamma$ oriented such that $\partial\a{\pi(q)^+} = -\partial\a{\pi(q)^-} = \a{T_q\Gamma}$.
\end{enumerate}
\end{proposition}
\begin{proof}
To prove \cite[Lemma 4.2]{nardulli2022density} the only property that is used is \eqref{E:density-trapped}, the full strength of the definition of collapsed points (i.e., lower bound for the density and existence of a flat tangent cone) is not used. We state here the consequence of \cite[Lemma 4.2]{nardulli2022density}: for every $\varepsilon>0$, there exist $\varepsilon_0=\varepsilon_0(\varepsilon, Q, Q^\star,m,n)>0$ and $M_0=M_0 (\varepsilon, Q, Q^\star,m,n)>0$ with the following property. If we define the function $\theta$ as $\theta(\sigma) := \max\{\bE^\flat (T, \ball{q}{\sigma}) ,M_0 \bA^2 \sigma^2\}$ for $\sigma\in (0,r_0)$, we obtain
\begin{equation}\label{eq:lemma-milder-decay:small-excess-assump}
    \bA^2 \sigma^2 + \bE^\flat (T, \ball{q}{4 \sigma}) < \varepsilon_0 \quad \Rightarrow \quad \theta(\sigma) \leq \max \{ 2^{-4+4\varepsilon} \theta (4\sigma), 2^{-2+2\varepsilon}  \theta(2\sigma)\},
\end{equation}
where $\bA := \max\{\bA_\Gamma^2, \bA_\Sigma^2\}$. 

We will show that \eqref{E:density-trapped} implies the smallness condition on the left-hand side of \eqref{eq:lemma-milder-decay:small-excess-assump}. Assume by contradiction that there is a sequence of currents $T_k^\prime$ such that \eqref{E:density-trapped} holds for $T_k^\prime$ with $\delta = 1/k$, for some $\sigma_k\in (0,r_0)$, and
\[\bA^2_k \sigma_k^2 + \bE^\flat (T_k^\prime, \ball{q_k}{4 \sigma_k}) \geq \varepsilon_0\mbox{ for some }\varepsilon_0>0\mbox{ and }q_k\in\Gamma_k\cap\ball{p_0}{r_0}.\]

Define $T_k = (T_k^\prime)_{0,4\sigma_k}$ and recall that $\bE^\flat (T_k, \oball{1}(q_k)) = \bE^\flat (T_k^\prime, \ball{q_k}{4 \sigma_k}) \geq \varepsilon_0$. Moreover, by \eqref{E:density-trapped} with $\rho = s\sigma_k$ and rescaling, we obtain
\begin{equation}\label{E:density-trapped_v2}
     \left| \frac{\|T_k\|\left(\ball{q}{s}\right)}{\omega_m s^m} - \left( \frac{Q}{2} + Q^\star\right)\right| < \frac{1}{k}, \quad \forall q\in\Gamma\cap\ball{p_0}{r_0}, \forall s\in \left(0,\frac14\right).
\end{equation}

We now take a limit of the sequence $T_k$, the boundaries $\Gamma_k$, and the ambient manifolds $\Sigma_k$. The currents will converge to a current $T_\infty$ and the boundaries $\Gamma_k$ converge to a flat boundary $L_\infty$. Since we have \eqref{E:density-trapped_v2}, the mass ratio of $T_\infty$ will be constantly equal to $Q/2+Q^*$ at every point of $L_\infty$ and in every radius $s<1$. Thus, it must be a two-sided flat cone. Up to a subsequence, $q_k \rightarrow q_\infty$ and then $\bE^\flat (T_\infty, \ball{q_\infty}{4r_0}) \geq \varepsilon_0$, which is a contradiction since $T_\infty$ is a flat two-sided cone, i.e., supported in a $m$-plane. 

Hence, for any $\varepsilon >0$, there is a choice of $\delta = \delta(\varepsilon)$ such that the left-hand side of \eqref{eq:lemma-milder-decay:small-excess-assump} holds. Thus, for any $\varepsilon>0$, there exists $\delta=\delta(\varepsilon)>0$ such that \eqref{E:density-trapped} imply
\begin{equation*}
     \theta(\sigma) \leq \max \{ 2^{-4+4\varepsilon} \theta (4\sigma), 2^{-2+2\varepsilon}  \theta(2\sigma)\}.
\end{equation*}

From this, the very same proof of \cite[Theorem 4.3]{nardulli2022density} provides (i) and (ii). 
\end{proof}

We establish a decomposition for $T$ at two-sided points as a direct consequence of the uniqueness of tangent cones at such points (\Cref{P:2sided-excessdecay+uniqueness}). Although this result \emph{is not} used elsewhere in this article, it is of independent interest. Notably, the significance of this decomposition lies in its elementary nature. Indeed, it does not rely on the full regularity theory at two-sided points with density close to $Q/2+Q^\star$ (i.e., constructing center manifolds, defining frequency functions, etc., see \Cref{P:2sided-min-dens-regular}); rather, it follows directly from an application of \Cref{P:2sided-excessdecay+uniqueness}.

\begin{proposition}[Decomposition at two-sided points]\label{P:2sided-decomposition}
Let $T,\Sigma$, and $\Gamma$ be as in \Cref{A:general}, $Q^\star\in\N\setminus\{0\}$, and $p_0\in\Gamma$. There exists $\delta>0$ with the following property. If, for some $r_0\in (0, r/4)$, the following holds
\begin{equation}
     \left| \frac{\|T\|\left(\ball{q}{\rho}\right)}{\omega_m \rho^m} - \left( \frac{Q}{2} + Q^\star\right)\right| < \delta, \quad \forall q\in\Gamma\cap\ball{p_0}{r_0}, \forall \rho\in (0,r_0),
\end{equation}
then $T \res \ball{p_0}{r_0} =T_1+T_2$ where $\partial T_1=(Q+Q^*)\a{\Gamma}$, $\partial T_2=-Q^* \a{\Gamma}$
with $\textup{spt}(T_1) \cap \textup{spt}(T_2)=\Gamma$, and every $q\in\Gamma\cap\ball{p_0}{r_0}$ is a one-sided flat boundary point for $T_1$ and $T_2$.
\end{proposition}
\begin{proof}
For $\theta_1,\theta_2\in\R, \theta_2>\theta_1$, we define the wedge region:
\begin{equation*}
    W_{\theta_1}^{\theta_2} := \left\{x \in \ball{p_0}{r_0}:  \left\langle \frac{x-\bp_{\Gamma}(x)}{|x-\bp_{\Gamma}(x)|},N_{\bp_{\Gamma}(x)}\Gamma \right\rangle  \in [\theta_1,\theta_2] \right\},
\end{equation*}
where, for any $q\in\Gamma$, $N_q\Gamma := (T_p\Gamma)^\perp$. Since $\Gamma$ is $C^{1,1}$, we find that $\bp_{\Gamma}$ is well-defined and $\mathrm{dist}(\cdot,\Gamma)$ is differentiable provided that $r_0$ is small enough. We will show that 
\begin{equation}\label{empty-wedge}
    \left(W^{\pi/4}_{-\pi/4}\cup W^{3\pi/4}_{5\pi/4}\right)\cap\spt(T) = \varnothing,
\end{equation}
which then allow us to conclude the proof since this means that 
\[\spt(T)\cap\ball{p_0}{r_0} \subset W_{\pi/4}^{3\pi/4} \cup W^{7\pi/4}_{5\pi/4}\]
and, hence, we can write $T=T_1+T_2$ with 
\[\spt(T_1)\setminus \Gamma \subset W_{\pi/4}^{3\pi/4}\mbox{ and }\spt(T_2)\setminus \Gamma \subset W^{7\pi/4}_{5\pi/4}.\] 

Additionally, we obtain that $\spt(\partial T_1) \subset \Gamma$ and $\spt(\partial T_2) \subset \Gamma$ and thus $\partial T_1=(Q+Q^*) \a{\Gamma}$ and $\partial T_2=-Q^*\a{\Gamma}$. 

To prove \eqref{empty-wedge}, we proceed by contradiction as follows. Assume the existence of 
\[x \in \left(W^{\pi/4}_{-\pi/4}\cap W^{3\pi/4}_{5\pi/4}\right)\cap\spt(T).\] 

We set $p:=\bp_{\Gamma}(x)$, let $\rho:=|x-p|$ and denote $\pi(p)$ the unique tangent cone to $T$ at $p$ (given by \Cref{P:2sided-excessdecay+uniqueness}). By the height bound in \cite[Corollary 4.7]{nardulli2022density}, denoting the projection onto $\pi^\perp$ by $\bp^\perp_{\pi}$ for any $m$-plane $\pi$, we have that 
\begin{equation*}
    \abs{\bp^\perp_{\pi(q)}(x) - p } \le C (r_0^{-1} \bE^\flat(T,\ball{p}{4r_0})^\frac{1}{2} + \bA)^{\frac{1}{2}} \rho^{3/2}.
\end{equation*}

Since $x\in W^{\pi/4}_{-\pi/4}\cap W^{3\pi/4}_{5\pi/4}$, we know that $\abs{ x - \bp_{\pi(p)}(x) } \leq C\abs{ x - \bp_{\pi(p)}^\perp(x) }$. This together with the last displayed inequality imply that 
\begin{equation*}
    \rho = \abs{x - p} \leq C\abs{ x - \bp_{\pi(p)}(x) } \leq C (r_0^{-1} \bE^\flat(T,\ball{p}{4r_0})^\frac{1}{2} + \bA)^{\frac{1}{2}} \rho^{3/2},
\end{equation*}
which gives us a contradiction.
\end{proof}

\subsection{Regularity of two-sided points with trapped density}

The following theorem builds on the theory developed in \cite{nardulli2022density,DDHM}, where the authors study the notion of collapsed points and establish that such points are always regular for $T$. However, we provide a refined version of this result. Specifically, the definition of collapsed points in \cite{nardulli2022density} requires two conditions: (1) the existence of a flat tangent cone at the given point, say $p_0$, and (2) $p_0$ to attain the minimum density in some relatively \emph{open} neighborhood of $\Gamma$, with this density being $\frac{Q}{2}+Q^\star$ for integers $Q,Q^\star\in\N\setminus\{0\}$. Instead of assuming both (1) and (2), we \emph{only} require that the density is $\delta$-close to $\frac{Q}{2}+Q^\star$ to establish regularity.

\begin{proposition}[Regularity of two-sided points with density close to $(\frac{Q}{2}+Q^\star)$]\label{P:2sided-min-dens-regular}
Let $T,\Sigma$, and $\Gamma$ be as in \Cref{A:general}, $Q^\star\in\N\setminus\{0\}$, and $p_0\in\Gamma$. For any $\varepsilon>0$, there exists a constant $\delta = \delta(\varepsilon) >0$ with the following property. If, for some $r_0>0$, the following holds
\begin{equation*}
     \left| \frac{\|T\|\left(\ball{q}{\rho}\right)}{\omega_m \rho^m} - \left( \frac{Q}{2} + Q^\star\right)\right| < \delta, \quad \forall q\in\Gamma\cap\ball{p_0}{r_0}, \forall \rho\in (0,r_0),
\end{equation*}
then $p_0$ must be a two-sided boundary regular point for $T$.
\end{proposition}
\begin{proof}
Appealing to \Cref{P:2sided-excessdecay+uniqueness}, we can readily obtain that
\begin{equation*}
    \Theta^m(T,p_0) = \min\left\{ \Theta^m(T,q): q\in \ball{p}{r_p}\cap\Gamma\cap \mathscr{P}_{m-1}^* \right\} > \frac{Q}{2},
\end{equation*}
where we define $\mathscr{P}_{m-1}^*$ to be only the points in the top $(m-1)$-stratum of $\Gamma$ with respect to $T$, i.e., 
\begin{equation}\label{E:def top stratum}
    \mathscr{P}_{m-1}^* = \mathscr{P}_{m-1}(T,\Gamma) \setminus \bigcup_{k=1}^{m-2}\mathscr{P}_k(T,\Gamma),    
\end{equation} 
we refer the reader to \cite[Section 3]{nardulli2022density}. Since $p_0$ is two-sided and belongs to $\mathscr{P}_{m-1}^*$, by \cite[Lemma 2.12]{nardulli2022density}, we see that $T_{p_0,r}$ blows up to a two-sided boundary flat cone. Hence, 
\begin{equation*}
    \Theta^m(T,p_0) = \frac{Q}{2}+Q^\star, \mbox{ for some }Q^{\star}\in \mathbb{N}\setminus\{0\}. 
\end{equation*}

We now prove that $p_0$ is a collapsed point as described above, see also \cite[Definition 2.17]{nardulli2022density}. Given any $x\in \Gamma\cap\ball{p}{r_p}$, we have $\Theta^m(T,x) \geq \Theta^m(T,p_0)$ if $x\in\mathscr{P}_{m-1}^*$, since $p_0$ is a minimum density point in $\mathscr{P}_{m-1}^*$. If $x\notin \mathscr{P}_{m-1}^*$, we use \cite[Theorem 3.2]{White97} (see also \cite[Section 3]{nardulli2022density}) to obtain that $\Gamma\cap\ball{p_0}{r_{p_0}}$ and $\mathscr{P}_{m-1}^*\cap\ball{p_0}{r_{p_0}}$ are $\cH^{m-1}$-equivalent, then we can take $x_k\in \mathscr{P}_{m-1}^*\cap\ball{p_0}{r_{p_0}}$ to converge to $x$. Next, using the upper semicontinuity of the density restricted to the boundary (\cite[Proposition 1.8]{nardulli2022density}), it holds that
\begin{equation*}
    \Theta^m(T,x) \geq \limsup_{k\to +\infty}\Theta^m(T,x_k) \geq \Theta^m(T,p).
\end{equation*}

Therefore, $p_0$ is a collapsed boundary point of $T$. We can now apply \cite[Theorem 2.21]{nardulli2022density} to see that $p_0$ is a regular boundary point for $T$ and conclude the proof. Note that \cite[Theorem 2.21]{nardulli2022density} is stated for $\Sigma=\R^{m+n}$, one can see (cf. \cite{DDHM}) that the proof works for any $T$ minimizing area in an arbitrary $C^{3,\kappa}$-submanifold $\Sigma$ of $\R^{m+n}$.
\end{proof}

\subsection{Decomposition at one-sided points}

In this section, we first develop the essential machinery for one-sided points, which will enable us to carry out our main reduction argument. In \Cref{S:main-results-proofs}, we then apply these tools to reduce the proof of the main theorems to the study of one-sided flat points.

\subsubsection{Decomposition theorem}

We present the following result, which ensures that the density is forced to be equal to $Q/2$, provided that it is sufficiently close to this value (with respect to a uniform constant). This result is established in \cite{ian2024uniqueness}.

\begin{theorem}[Density jump]\label{T:density-jump}
Under \Cref{A:general}, we have $\Theta^m(T,p) \geq Q/2$. The equality case holds if and only if the tangent cone to $T$ at $p$ is an open book and unique. Moreover, there exists a constant $\varepsilon = \varepsilon(m,n,Q)>0$ such that if $\Theta^m(T,p)>Q/2$ then $\Theta^m(T,p)>Q/2+\varepsilon$.  
\end{theorem}
\begin{remark}\label{R:1sided set is open}
Notice that this result, together with the upper semicontinuity of the density restricted to $\Gamma$ (\cite[Proposition 1.8]{nardulli2022density}) implies that the set of one-sided boundary points is a relatively open set in $\Gamma$.
\end{remark}

The following decomposition result is stated in \cite{delellis2021allardtype} for $2$-dimensional area minimizing integral currents and generalized in \cite{ian2024uniqueness} to arbitrary dimensions.

\begin{theorem}[Decomposition at one-sided points theorem]\label{T:decomposition-1sided} 
Let $T$ be an area minimizing current as in Assumption \ref{A:general}, assume that $\Theta(T,0)=Q/2$, and let $C=\sum_{i=1}^N Q_i \a{H_i}$ be the unique tangent cone to $T$ at $0$ (where the representation of $C$ is such that the half-planes are distinct). Then there exists $\rho>0$ and area minimizing currents $T_1,T_2,...,T_N$ in $\bB_{\rho}$ such that
\begin{equation*}
T \res \bB_{\rho}=\sum_{i=1}^N T_i,\mbox{ with }\partial T_i \res \bB_{\rho}=Q_i \a{\Gamma},
\end{equation*}
and where the supports of $T_i$ only intersect at $\Gamma$. Moreover, the (unique) tangent cone at $0$ of $T_i$ is $Q_i \a{H_i}$.
\end{theorem}

\subsection{Main reduction argument and proof of the main results}\label{S:main-results-proofs}

We are now prepared to reduce the main results to establish the rectifiability of the set of singular \emph{flat} one-sided points. For the readers' convenience, we restate the main results here.

\begin{theorem*}[\Cref{T:Rectifiability-1sided}]
Under \Cref{A:general}, we have that $\Singb^1(T)$ is $\cH^{m-3}$-rectifiable.
\end{theorem*}

This theorem will be derived as a corollary of the following.

\begin{proposition}[Rectifiability of the set of one-sided flat singular points]\label{P:Rectifiability-FlatSingular}
Under \Cref{A:general}, we have
\begin{equation*}
        \Singb^1(T) \cap \left\{q\in\Gamma: q\mbox{ is flat}\right \}\mbox{ is }\cH^{m-3}\mbox{-rectifiable.}
\end{equation*}
\end{proposition}

The proof of \Cref{P:Rectifiability-FlatSingular} will be the focus of the remaining sections of this article. We now demonstrate how \Cref{T:Rectifiability-1sided} follows from \Cref{P:Rectifiability-FlatSingular}.

\begin{proof}[Proof of \Cref{T:Rectifiability-1sided}]
We argue by induction. The case $Q=1$ is given by Allard's boundary regularity result, see \cite{AllB}. Indeed, in this case, we do not have any singular boundary point. Assume now that the result holds for $Q'\geq 1$ and take $Q>Q^\prime$. We define
\begin{equation*}
    E_i:= \Singb(T) \cap \left\{q\in\Gamma:\Theta^m(T,q)=Q/2 \mbox{ and }T_q\mbox{ has }i\mbox{ sheets}\right\},
\end{equation*}
where $T_q$ is the tangent cone to $T$ at $q$. We can cover each $E_i, i>2$, with a countable family of open balls whose centers belong to $E_i$, namely $\{B_j^i\}_{j\in\N}$, each of which is given by \Cref{T:decomposition-1sided}. Indeed, since $i>2$, any $q\in E_i$ is a one-sided singular point which is not flat and, thus, we can apply \Cref{T:decomposition-1sided}. Notice that for every $q \in \Gamma \cap B_j^i$, we have $\Theta^m(T,q)=Q/2$. From this covering, we can write $T$ locally in each of these balls as
\begin{equation*}
    T\res B_j^i = T_1^{j,i} + T_2^{j,i}, \ j\in\N,
\end{equation*}
where $T_1^{j,i}$ and $T_2^{j,i}$ are area minimizing in $B_j^i$, $Q^{j,i}_1,Q_2^{j,i}>0$ are two integers, $\partial T^{j,i}_1 = Q^{j,i}_1\a{\Gamma\cap B_j^i}$, $\partial T^{j,i}_2 = Q^{j,i}_2\a{\Gamma\cap B_j^i}$, and $Q^{j,i}_1+Q_2^{j,i} = Q$. Now, we can apply our inductive hypothesis to each $T^{j,i}_1$ and $T^{j,i}_2$ with $i>2$ to get the $\cH^{m-3}$-rectifiability of the corresponding set of singular one-sided points. Given that
\begin{equation*}
    E_i = \bigcup_{j\geq 1} \left(\Singb(T^{j,i}_1)\cup\Singb(T^{j,i}_2)\right)\cap B_j^i,\mbox{ and }E = \bigcup_{i=1}^Q E_i,
\end{equation*}
to obtain the rectifiability of $E$, we are only left to prove rectifiability of $E_1$ which is exactly what we claim in \Cref{P:Rectifiability-FlatSingular}. Thus, we conclude the proof.
\end{proof}

As a consequence of \Cref{T:Rectifiability-1sided}, we can establish an Allard-type regularity result, provided there exists a relatively \emph{open} neighborhood of $\Gamma$ of \emph{flat} one-sided points. Furthermore, given the counterexample in \Cref{T:ian-example}, we know that this assumption cannot be eliminated.

Additionally, we recall that for $m=2$, Allard's theorem holds for any $Q$, meaning that one-sided points are automatically regular, as shown in \cite{delellis2021allardtype}.

For the readers' convenience, we restate \Cref{T:1sided-flat-regular} here.

\begin{theorem*}[\Cref{T:1sided-flat-regular}]
Under \Cref{A:general}, assume that $q\in\Gamma$ is a one-sided flat point of $T$ and there is a neighborhood of $q$ in $\Gamma$ containing \emph{only} one-sided \emph{flat} points, then $q$ is a regular point.
\end{theorem*}
\begin{proof}
Let $q$ be a one-sided flat point and $r>0$ such that $\ball{q}r \cap \Gamma$ consists of one-sided flat points. By \Cref{T:Rectifiability-1sided}, there is a one-sided regular point $x\in \ball{q}{r} \cap \Gamma$. Using the fact that $x$ is a flat and regular point, there exist $\rho>0$ and a \emph{single} smooth manifold $M$ such that $T \res \ball{x}{\rho}= Q\a{M}$ (cf. \Cref{D:Reg-1sided}). 

Next, we extend $M$ to $\tilde{M}$, a connected component of $\textup{spt}(T) \setminus \left(\Gamma \cup  \Singi(T)\right)$ outside $\ball{x}{\rho}$ where $\Singi(T)$ denotes the interior singular set of $T$. Since $\dim_H(\Singi(T))\leq m-2$, we have that $\spt(\partial \a{\tilde{M}}) \subset \Gamma.$ By the Constancy Lemma, $\partial \a{\tilde{M}}=\a{\Gamma \cap \ball{q}{r}}.$

We will show that $T$ agrees with $Q\a{\tilde{M}}$ on a neighborhood of $q$. It suffices to show that 
\[\ball{q}{r} \cap \textup{spt}(T-Q\a{\tilde{M}}) =\varnothing.\] 

Assume by contradiction that $y$ belongs to such a set. Since $T-Q\a{\tilde{M}}$ is area minimizing, by the monotonicity formula, we have $\Theta^m(T-Q\a{\tilde{M}},y) \geq 1$. Additionally, $\Theta^m(T,y)=\Theta^m(Q\a{M},y)=Q$, which follows by construction $y\in\spt(T)\cap\tilde{M}$ and the Hausdorff dimensional bound for the interior singular set. These two facts are in contradiction, and hence we conclude the proof.
\end{proof}

We employ a reduction argument to demonstrate that it suffices to prove \Cref{T:open-dense} by assuming that every boundary point is a one-sided flat point for $T$ within a small ball, and then apply \Cref{P:Rectifiability-FlatSingular}. For the readers' convenience, we restate it here.

\begin{theorem*}[\Cref{T:open-dense}]
Under \Cref{A:general}, we have that $\Regb(T)$ is an open and dense subset of $\Gamma\cap\ball{p}{R}$.
\end{theorem*}

\begin{proof}[Proof of \Cref{T:open-dense}]
Assume that there is a singular boundary point $x$ for $T$ and $r>0$ such that $\ball{x}{r}\cap\Gamma\subset \Singb(T)$. We choose $p_0\in\ball{x}{r}\cap\Gamma$ to have minimum density in the $(m-1)$-stratum (this is possible due to \cite[Lemma 2.12]{nardulli2022density}), i.e., 
\begin{equation}\label{E:main-reduction:choice-p}
    \Theta^m(T,p_0) = \min \left\{\Theta^m(T,q): q\in\ball{x}{r}\cap\Gamma\cap\mathscr{P}_{m-1}^*\right\},
\end{equation}
where $\mathscr{P}_{m-1}^*$ is defined as above in \eqref{E:def top stratum}. 

We divide into two cases. If $\Theta^m(T,p_0) = Q/2 +Q^\star$, for some $Q^\star\in\N\setminus\{0\}$, then there exists $r_1>0$ such that 
\begin{equation}\label{E:min density is collapsed}
    \Theta^m(T,q) \geq Q/2 +Q^\star\mbox{ for any }q\in \ball{p_0}{r_1}\cap\Gamma \subset \ball{x}{r}.
\end{equation}

In fact, if not, there is $\{x_k\}_k\subset \ball{p_0}{1/k}\cap \Gamma\subset \ball{x}{r}$ such that $\Theta^m(T,x_k) < Q/2 +Q^\star$ and $\Theta^m(T,x_k) \to Q/2 +Q^\star$. Recall that $\mathscr{P}_{m-1}^*$ is $\cH^{m-1}$-equivalent to $\Gamma$, and take a sequence $\{y_j^k\}_j \subset \mathscr{P}_{m-1}^*\cap \ball{x}{r}$ such that $y_j^k \to x_k$ as $j$ grows. Therefore, we have
\[ Q/2 +Q^\star > \Theta^m(T,x_k) \geq \limsup_j \Theta^m(T,y^k_j) \geq Q/2 + Q^\star, \]
where the last inequality follows from $y_j^k \in \mathscr{P}_{m-1}^*\cap \ball{x}{r}$ and that $p_0$ is a minimum density point within this set, cf. \eqref{E:main-reduction:choice-p}. The last displayed inequality provides the contradiction and concludes the proof of \eqref{E:min density is collapsed}. 

By the upper semicontinuity of the density restricted to $\Gamma$ and \eqref{E:min density is collapsed}, given $\delta>0$, there is an open neighborhood $U$ of $p_0$ such that $\Theta^m(T,q)$ is $\delta$-close to $\Theta^m(T,p_0) = Q/2 +Q^\star$ for any $q\in U\cap\Gamma$ and thus \Cref{P:2sided-min-dens-regular} implies that $p_0$ is regular, which contradicts the fact that $\ball{x}{r}\cap\Gamma$ only has singular points.

Next, assume $\Theta^m(T,p_0) = Q/2$. Recall that, by upper semicontinuity of the density (\cite[Proposition 1.8]{nardulli2022density}) and the density jump (\Cref{T:density-jump}), we obtain that the set of one-sided boundary points is relatively open in $\Gamma$ and, therefore, there exists $r_0>0$ such that $\ball{p_0}{r_0}\subset\ball{x}{r}$ and
\begin{equation*}
    \Theta^m(T,q) = \frac{Q}{2}, \mbox{ for every } q\in\ball{p_0}{r_0}\cap\Gamma\subset\Singb(T).
\end{equation*}

Now, by \Cref{T:Rectifiability-1sided}, the Hausdorff dimension of $\Singb(T)\cap\{q\in\Gamma:\Theta^m(T,q)=Q/2\}$ is capped at $m-3$ which then gives us a contradiction since $\ball{p_0}{r_0}\cap\Gamma$ is supposed to only have flat singular points of $T$. Hence, we conclude the proof of \Cref{T:open-dense}.
\end{proof}

\section{Universal frequency function \texorpdfstring{$\bI$}{Lg} at one-sided flat points}

\subsection{Center manifold and normal approximation at one-sided flat points}\label{S:1sided-CM-NA}

In this subsection, we construct the center manifold (\Cref{T:CM}) and the normal approximation (\Cref{T:LocalNA}) around a flat one-sided singular point of an area-minimizing current $T$. These constructions are well established and have previously been formulated in different contexts.

Next, we introduce the intervals of flattening (\Cref{S:IntervalsFlattening}), which identify the scales at which $T$ decays almost quadratically toward its tangent cone. Using these intervals, we define the universal frequency function (\Cref{S:UniversalFreqDef}), a fundamental tool for analyzing regularity in this setting. Although most of the frequency function estimates in this section are well known in the two-dimensional case—see \cite{delellis2021allardtype}—certain refinements are necessary to integrate them into the framework used to address the rectifiability problem.

We begin by establishing the existence of a Lipschitz approximation for an area-minimizing current in a neighborhood of a one-sided flat singular point.

We first recall and introduce the following notation. The function $\phi$ is defined on $\bp_{\pi_0}(\Gamma)$ and it satisfies $\mathrm{graph}(\phi)\cap\oball{2} = \Gamma\cap\oball{2}$. We define $\bA:=\max\{\bA_{\Sigma}, \bA_{\Gamma}\}$ where $\bA_\Gamma$ and $\bA_\Sigma$ are the $C^0$-norms of the second fundamental forms of $\Gamma$ and $\Sigma$.

\begin{proposition}\label{P:1sided-StrongLA}
Let $T, \Sigma,$ and $\Gamma$ be as in \Cref{A:general} with $p=0$ and $R=2$, assume  in addition that
\begin{itemize}
   \item $T_0\Gamma = \R^{m-1}\times \{0\}\subset\pi_0 \subset T_0\Sigma$,

   \item $\bp_{\pi_0}(\Gamma)$ splits $\obball{2}$ into two disjoint connected open sets, which we denote by $\obball{2}^+$ and $\obball{2}^-$,

   \item $(\bp_{\pi_0})_{\#}T\res\ocyl{2} = Q\a{\obball{2}^+}$.
\end{itemize}
Then there are geometric constants $\gamma_{la}>0$, $\varepsilon_{la}>0$, and $C>0$ with the following properties. If $E := \bE(T,\ocyl{2},\pi_0) < \varepsilon_{la}$, there exist a set $K\subset \obball{3/2}^+$ and a function $f$ defined on $\obball{3/2}^+$ taking values in $\cA_Q(\R^n)$ such that
\begin{align}
    f|_{\partial\obball{2}^+} = Q\a{\phi} \mbox{ and } &\mathrm{spt}(\bG_f) \subset \Sigma , \label{T:1sided-LA:bdr-condition}       \\
    \mathrm{Lip}(f) &\leq C(E + r^2\bA^2)^{\gamma_{la}}, \label{T:1sided-StrongLA:Lip} \\ 
    \mathrm{osc}(f) &\leq C\bh(T,\ocyl{2}) + Cr (E + r^2\bA^2)^{\frac{1}{2}}, \label{T:1sided-StrongLA:osc} \\
    \bG_f\res(K\times \R^n) &= T\res (K\times \R^n), \label{T:1sided-StrongLA:T=G_f}\\
    \cH^m(\obball{1}\setminus K) + \be_T(\obball{1}\setminus K) &\leq Cr^m(E + r^2\bA^2)^{1+\gamma_{la}}, \label{T:1sided-StrongLA:bad-set}\\
    \biggl| \|T\|(A\times\R^n) - Q\cH^m(A\setminus\obball{2}) - \frac{1}{2}\int_A |Df|^2 \biggr| &\leq Cr^m (E + r^2\bA^2)^{1+\gamma_{la}} \label{T:1sided-StrongLA:taylor}.
\end{align}
\end{proposition}
\begin{proof}
The theory in \cite[Sections 6, 7, 8, and 9]{delellis2021allardtype} is already developed in an arbitrary dimension. Note also that \cite[Lemma 7.6]{delellis2021allardtype} (stated in 2d) can be replaced by its counterpart in any dimension \cite[Lemma 7.1]{delellis2021uniqueness}.
\end{proof}

Before starting the construction of the center manifold and normal approximation at flat one-sided points, we note that \cite{delellis2021allardtype} will be frequently referenced throughout. Although the authors consider $\Sigma = \mathbb{R}^{m+n}$, it is straightforward to verify that all their results remain valid for an arbitrary $\Sigma \subset \mathbb{R}^{m+n}$ of class $C^{3,\kappa}$, introducing certain errors that we will explicitly state below.

\begin{assumption}\label{A:CM}
Under \Cref{A:general} with $p=0$ and $R=2$, we assume that $\Sigma\cap\oball{2} = \mathrm{graph}(\Psi)$ and $\Gamma\cap\ball{0}{2} = \mathrm{graph}(\phi)$ where $\Psi: \oball{2}\cap\R^{m+\bar{n}} \to \R^{n-\bar{n}}$ and $\phi: \oball{2} \cap \R^{m-1}\to \R^{n+1}$. We assume that
\begin{enumerate}[\upshape (i)]
    \item $\max\left\{\bE^\flat (T, \oball{2}),\|\Psi\|_{C^{3,\kappa}}^2, \|\phi\|_{C^{3,\kappa}}^2\right\} =: \mbold{0} < \varepsilon_{cm},$ where $\varepsilon_{cm} = \varepsilon_{cm}(m,n,\bar{n},Q)>0$ is a small constant,

    \item Assume $0$ is a one-sided flat boundary point for $T$ and the tangent cone to $T$ at $0$ is $T_0 = Q\a{\pi_0^+}$.
\end{enumerate}
\end{assumption}

We begin by describing a Whitney-type decomposition of $\halfobball{3/2}$, which denotes the connected component of $\obball{3/2}\setminus\bp_{\pi_0}(\Gamma)$ with nonempty intersection with ${\bp_{\pi_0}}_{\sharp}(T\res\oball{3/2})$, with cubes whose sides are parallel to the coordinate axes and have side length $2 \ell(L)$. The center of any such cube $L$ will be denoted by $c_L$ and its sidelength will be denoted by $2 \ell (L)$. We introduce a family of dyadic cubes $L\subset \pi_0$ in the following way: for $j\geq N_0$, where $N_0$ is an integer whose choice will be specified below,
we introduce the families 
\begin{equation*}
\sC_j:=\{L:\,L\text{ is a dyadic cube of side }\ell(L)=2^{-j}\text{ and } \halfobball{3/2}\cap L\neq\varnothing\} ,
\end{equation*}

For each $L$ define a radius $r_L:=M_0\sqrt m\ell(L),$ with $M_0\geq 1$ to be chosen later.  
We then subdivide $\sC := \cup_j \mathscr{C}_j$ into, respectively, \emph{boundary cubes} and \emph{non-boundary cubes}
\begin{align*}
\sC^\flat & :=\{L\in \sC:\,\dist(c_L,\bp_{\pi_0}(\Gamma))<64 r_L\}, \quad \sC_j^\flat = \mathscr{C}^\flat \cap \mathscr{C}_j,\\
\sC^\natural & :=\{L\in \sC:\,\dist(c_L,\bp_{\pi_0}(\Gamma))\ge 64 r_L\}, \quad \mathscr{C}^\natural_j = \mathscr{C}^\natural \cap \mathscr{C}_j.
\end{align*} If $H, L \in \mathscr{C}$ we say that:
\begin{itemize}
    \item $H$ is a \emph{{descendant}}\index{Descendant (in a Whitney decomposition)} of $L$ and $L$ is an \emph{{ancestor}} of $H$, if $H\subset L$;
    \item $H$ is a \emph{{child}} of $L$ and $L$ is the \emph{{parent}} of $H$, if $H\subset L$ and $\ell (H) = \frac{1}{2} \ell (L)$;
    \item $H$ and $L$ are \emph{{neighbors}} if $\frac{1}{2} \ell (L) \leq \ell (H) \leq \ell (L)$ and $H\cap L \neq \varnothing$. 
\end{itemize}

\begin{lemma}\label{l:child_parent}
Let $H$ be a boundary cube. Then any ancestor $L$ and any neighbor $L$ with $\ell (L) = 2 \ell (H)$ are necessarily a boundary cube. In particular: the descendant of a non-boundary cube is a non-boundary cube.
\end{lemma}

\begin{itemize}
    \item If $L\in\sC^\natural$, then we define the {\emph{non-boundary satellite ball $\bB_L = \ball{p_L}{64r_L}$}} where $p_L\in\spt(T)$ such that $\bp_{\pi_0}(p_L)=c_L$, such \(p_L\) is a priori not unique, and $\pi_L$ is a plane which minimizes the excess in $\bB_L$, namely $\bE (T, \bB_L)= \bE (T, \bB_L, \pi_L)$ and $\pi_L\subset T_{p_L}\Sigma$,

    \item If $L\in \mathscr{C}^\flat$, then we define the {\emph{boundary satellite ball $\bB^\flat_L = \ball{p_L^\flat}{2^7 64r_L}$}} where $p_L^\flat\in\Gamma$ is such that $|\bp_{\pi_0}(p_L^\flat)-c_L|=\dist(c_L,\bp_{\pi_0}(\Gamma))$. Note that in this case, the point $p^\flat_L$ is uniquely determined because $\Gamma$ is $C^{3,\kappa}$ and $\mbold{0} $ is assumed to be sufficiently small. Likewise $\pi_L$ is a plane that minimizes the excess $\bE^\flat$, that is, such that $\bE^\flat (T,\bB^\flat_L) = \bE (T, \bB^\flat_L, \pi_L)$ and $T_{p^\flat_L} \Gamma \subset \pi_L \subset T_{p^\flat}\Sigma$.
\end{itemize}

\subsubsection{Stopping conditions}\label{S:Stopping-condition}

We introduce the set up to decide whether or not a cube is such that $T$ has good enough properties around it. These properties are mainly based upon how small the excess or height of $T$ are with respect to the size of a given cube. 

Let \(C_\be, C_\bh\) be two large positive constants that will be specified later. We define
\begin{enumerate}[\upshape (i)]
    \item $\sW_{N_0}^{\be}:=\{L\in\sC_{N_0}^\natural\cup\sC_{N_0}^\flat:\, \bE(T, \bB_L)> C_\be \mbold{0} \ell(L)^{2-\alpha_\be}\}$,
    \item $\sW_{N_0}^{\bh}:=\{L\in\sC_{N_0}^\natural\cup\sC_{N_0}^\flat:\, \bh(T, \bB_L,\pi_L)> C_\bh\mbold{0}^{\sfrac{1}{2m}} \ell(L)^{1+\alpha_\bh}\}$,
    \item $\sS_{N_0}:=\sC_{N_0}\setminus \left(\sW_{N_0}^{\be}\cup\sW_{N_0}^{\bh}\right)$.
\end{enumerate}

Assume inductively that for a certain step $j\geq N_0+1$, we have defined
the families $\sW_{j-1} := \sW_{j-1}^\be\cup\sW_{j-1}^\bh$ and $\sS_{j-1}$. We set the following stopping criteria:

\begin{enumerate}[\upshape (i)]
    \item $\sW_j^{\be}:=\{L \mbox{ child of } K \in\sS_{j-1}:\, \bE(T, \bB_L)> C_\be\mbold{0} \ell(L)^{2-2 \alpha_\be}\}$,

    \item $\sW_j^{\bh}:=\{L \mbox{ child of } K\in \sS_{j-1}:\, L\not\in \sW_j^\be\mbox{ and } \bh(T, \bB_L,\pi_L)> C_\bh\mbold{0}^{\sfrac{1}{2m}} \ell(L)^{1+\alpha_\bh}\}$,
 
    \item $\sW_j^{\mathbf{n}}:=\{L\mbox{ child of } K\in \sS_{j-1}:\, L\not\in \sW_j^\be\cup \sW_j^\bh \mbox{ but } \exists L'\in \sW_{j-1} \text{ with }L\cap L' \neq \varnothing\}$.
\end{enumerate}

Thus, we define $\sS_j:=\left\{ L\in \mathscr{C}_j \mbox{ child of } K\in \sS_{j-1}\right\} \setminus \sW_j $. Furthermore, we set
\begin{equation*}
\sW  :=\bigcup_{j\ge N_0} \sW_j, \ \ \sS :=\bigcup_{j\ge N_0} \sS_j, \ \ \bS  :=\bigcap_{j\ge N_0}\Big( \bigcup_{L\in \sS_j} L\Big)=\halfobball{3/2}\setminus \bigcup_{H\in \sW} H\, .
\end{equation*}

\begin{remark}
Unlike the two-sided case, as in \cite{DDHM,nardulli2022density}, we cannot state that every boundary cube is \emph{nonstopping} since we do not have an almost-quadratic excess decay (which is strongly used in \cite{DDHM,nardulli2022density} to prove this affirmation). Hence, we have to include boundary cubes in the stopping conditions.
\end{remark}

Since $(\sW,\bS)$ is a Whitney decomposition of $[-3/2,3/2]^m$, it is not hard to see that 
\begin{equation}\label{E:separation-Whitney}
    \mathrm{sep}(\bS, L):= \inf\left\{ |x-y|: x\in L, y\in\bS\right\} \geq 2\ell(L), \mbox{ for any }L\in\sW.
\end{equation}

We now assume an order of hierarchy for all the parameters involved in the construction of the center manifold that will follow.

\begin{assumption}\label{A:ParametersCM}
$T$ and $\Gamma$ are as in Assumptions \ref{A:CM} and we also assume that
\begin{enumerate}[\upshape (i)]
    \item $0<\alpha_\bh \leq \sfrac{1}{2m}$ and $\alpha_\be>0$ is positive but small, depending only on $\alpha_{\bh}$,
    \item $M_0$ is larger than a suitable constant, depending only upon $\alpha_\be$,
    \item $2^{N_0} \geq C (m,n, M_0)$, 
    \item $C_\be$ is sufficiently large depending upon $\alpha_\be$, $\alpha_\bh$, $M_0$ and $N_0$,
    \item $C_\bh$ is sufficiently large depending upon $\alpha_\be, \alpha_\bh, M_0, N_0$ and $C_\be$,
    \item Item (i) of \Cref{A:CM} holds with $\varepsilon_{par}$ sufficiently small depending upon all the other parameters above.
\end{enumerate}
\end{assumption}

\subsubsection{Existence of the center manifold}

In this subsection, we state the existence of the center manifold which acts as an ``average" of the sheets of the current $T$. Furthermore, the center manifold $\cM$ is a classical submanifold of $\Sigma$ of regularity $C^{3,\kappa}$ which allows us to study further properties of the current $T$.

\begin{theorem}[Theorem 10.16 \cite{delellis2021allardtype}]\label{T:CM}
Under Assumptions \ref{A:CM} and \ref{A:ParametersCM}, there is a $\kappa := \kappa(\alpha_{\be}, \alpha_{\bh},Q,Q^\star,m,\bar{n},n)>0$, such that

\begin{enumerate}[\upshape (i)]
    \item $\varphi_j\in C^{3,\kappa}$ with $\|\varphi_j\|_{C^{3, \kappa}( \halfobball{3/2})} \leq C \mathbf{m_0}^{\sfrac{1}{2}}$,  for some $C:=C(\alpha_\be, \alpha_\bh, M_0, C_{\be}, C_{\bh})>0$,

    \item If \(i\le j\), $L\in \sW_{i-1}$ and $H$ is a cube concentric to $L$ with $\ell (H) = \frac{9}{8} \ell (L)$, then $\varphi_j = \varphi_i$ on $H$,

    \item $\varphi_j$ converges in the $C^3$-topology to a map ${\bphi}: \halfobball{3/2}\to \R^n$, whose graph is a $C^{3,\kappa}$-submanifold $\cM$ of $\Sigma$, which will be called the {{\em center manifold}},

    \item $\partial(\mathrm{graph}(\bphi)) = \Gamma$, i.e., $\partial \cM \cap \bC_{3/2} = \Gamma \cap \bC_{3/2}$,

    \item For any $q\in \partial \cM \cap \bC_{3/2}$, we have $T_q \cM = \pi^+ (q)$ where $\pi(q)$ is the supporting $m$-plane for the unique tangent cone to $T$ at $q$.
\end{enumerate}
\end{theorem}
\begin{proof}
We will omit the proof because it follows the same lines as in \cite[Theorem 10.16]{delellis2021allardtype}. In fact, the authors, in \cite{nardulli2022density}, do not use the fact that $m=2$ to prove the existence of $\cM$ as above. 
\end{proof}

\subsubsection{Normal approximations}

With the center manifold $\cM$ in hand, our objective is to construct approximations of $T$—akin to those in \Cref{P:1sided-StrongLA}—by functions defined on $\cM$ and taking values in its normal bundle. These approximations, known as $\cM$-normal approximations, are formally defined in \Cref{D:NA} and their existence is established in \Cref{T:LocalNA}.

Let us define the graph parametrization map of $\cM$ as $\bPhi(x):=(x,\bphi(x))$. We will call {\em contact set} the subset $\bPhi(\bS)$. For every cube $L\in \sW$ we associate a {\em Whitney region $\cL$ on $\cM$} as follows: 
\begin{itemize}
    \item $\cL:= \bPhi(H\cap \obball{1})$ where $H$ is the cube concentric to $L$ such that $\ell(H)=\frac{17}{16}\ell(L)$.
\end{itemize}

\begin{assumption}\label{A:NA}
Under Assumptions \ref{A:CM} and \ref{A:ParametersCM}. We possibly further reduce $\varepsilon_{cm}$, call it $\varepsilon_{na}$, and assume $\mbold{0} < \varepsilon_{na}$, to be such that, if
\begin{align*}
\bU := \{q\in \R^{m+n} :\,\exists! q' = \bp (q) \in \cM \mbox{ s.t. $|q-q'|<1$} \mbox{ and $q-q'\in T_{q'}^\perp\cM$}\}\, 
\end{align*}
where $T_{q'}^\perp\cM:=(T_{q'}\cM)^\perp$, then the map $\bp$ extends to a Lipschitz map to the closure $\overline{\bU}$ which is $C^{2,\kappa}$ on \(\bU\setminus \bp^{-1} (\Gamma)\) and $\bp^{-1} (q') = q' + \overline{\mathrm{B}_1 (0, (T_{q'} \cM)^\perp)}$ for all $q'\in \overline{\cM}$.
\end{assumption}

The following corollary states strong properties on the behavior of the current $T$ whenever we look at points on the contact set.

\begin{corollary}\label{C:2.2}
Under Assumptions \ref{A:CM} and \ref{A:ParametersCM},, we have
\begin{enumerate}[\upshape (i)]
    \item\label{C:2.2:i} $\spt(\partial(T\res\bU))\subset \partial_l\bU$, $\spt_\sharp (T\res [-\frac{7}{2},\frac{7}{2}]\times\mathbb{R}^n)\subset\bU$, and $\bp(T\res\bU) = Q\a{\cM}$,
    \item\label{C:2.2:ii} $\spt(\langle T, \bp, \boldsymbol{\Phi}(q)\rangle) \subset\left\{y:|\bPhi(q)-y| \leq C \mbold{0}^{1 / 2 m} \ell(L)^{1+\alpha_\bh}\right\}$ for every $q \in L \in \sW$,
    \item\label{C:2.2:iii} $\langle T, \bp, p\rangle=Q \a{ p }$ for every $p \in \bPhi(\bS)\cup\left(\Gamma\cap\partial\cM\right)$.
\end{enumerate}
\end{corollary}

Next, we define the $\cM$-normal approximations of $T$ defined on $\cM$ and taking values in its normal bundle.

\begin{definition}\label{D:NA}
Let $\cM$ be the center manifold under Assumption \ref{A:NA}. We say that $(\cK, F)$ is an {\emph{$\cM$-normal approximation of $T$}}, if
\begin{enumerate}[\upshape (i)]
    \item there exists a Lipschitz function $$\begin{array}{lrll}
    \cN: & \cM\cap\bC_1 & \to & \cA_{Q}(T^\perp\cM)   \\
    & x & \mapsto &\cN(p):=\sum_{i=1}^{Q}\a{\cN_i(p)},
    \end{array}$$ 
    where $(\cN_i:\cM\cap \bC_1\to T^\perp\cM)_{i\in\{1,\dots, Q\}}$ are measurable sections of the normal bundle satisfying $p + \cN_i(p) \in\Sigma$. We then define the Lipschitz function given by $$\begin{array}{lrll}
    F: & \cM\cap\bC_1 & \to & \cA_{Q}(T^\perp\cM_k)   \\
    & x & \mapsto &\left(\cN \oplus \mathrm{id}\right)(x).
\end{array}$$

    \item $\cK \subset \cM$ is closed and $\mathbf{T}_{F} \res \bp^{-1} (\cK\cap \cM) = T \res \bp^{-1} (\cK\cap \cM)$, where $\mathbf T_{F}:=(F)_\sharp \a{\cM}$, according to \emph{\cite[Definition 1.3]{DS2}},

    \item $\bPhi(\bS)\cap \Gamma \subset \cK$, $\cN |_{\cK}\equiv Q\a{0}$, and then $F(p) =  Q \a{p}$ on $\cK$.
\end{enumerate}
\end{definition}

We finally state the existence and local properties of a normal Lipschitz approximation of $T$. 

\begin{theorem}[Local behavior of $\cM$-normal approximations. Theorem 10.21 of \cite{delellis2021allardtype}]\label{T:LocalNA}
Under Assumption \ref{A:NA}, there is a constant $\alpha_\bL=\alpha_\bL(m, n, \bar{n},  Q)>0$ such that there exists an $\cM$-normal approximation $(\cK, F)$ satisfying the following estimates on any Whitney region $\cL\subset \cM$ associated to a cube $L \in \sW$:
\begin{align}
    \mathrm{Lip} ({\cN}|_{\cL})  \leq C\mbold{0}^{\alpha_\bL} \ell (L)^{\alpha_\bL} \mbox{ and }&
    \|{\cN}|_{\cL}\|_0 \leq C \mbold{0}^{\sfrac{1}{2m}} \ell (L)^{1+\alpha_\bh},\nonumber\\
    \cH^m(\cL\setminus \cK) + \|\mathbf{T}_F -T\| (\bp^{-1} (\cL)) & \leq C \mbold{0}^{1+\alpha_\bL} \ell(L)^{m+2+\alpha_\bL},\nonumber\\
    \int_{\cL} |D{\cN}|^2 & \leq C \mbold{0} \ell (L)^{m+2-2\alpha_\be}, \label{E:LocalNA:Dir}
\end{align}
for a constant $C = C (\alpha_\be,\alpha_\bh,M_0,N_0,C_\be,C_\bh)>0$. Moreover, for any $a>0$ and any Borel $\mathcal{V} \subset \cL$,
\begin{equation*}\label{e:cm_app5}
\begin{aligned}
    \int_{\mathcal{V}} |\eta\circ {\cN}|\leq C \mbold{0} \left(\ell(L)^{m+3+\frac{\alpha_\bh}{3}} + a \ell (L)^{2+\frac{\alpha_\bL}{2}} \mathcal{H}^m(\mathcal{V})\right)
    + \frac{C}{a} \int_{\mathcal{V}} \mathcal{G} ({\cN}, Q \a{\eta\circ {\cN}})^{2+\alpha_\bL}.
\end{aligned}
\end{equation*}
\end{theorem}

Next, we state global estimates that are derived from summing up all the local estimates in the previous theorem over all cubes of the Whitney decomposition.

\begin{corollary}[Global estimates for $\cM$-normal approximations on Whitney regions. Corollary 10.22 of \cite{delellis2021allardtype}]\label{T:GlobalNA}
Under the same assumptions and notation of \Cref{T:LocalNA}, and setting $\cM^\prime = \bPhi(\halfobball{1}\cap [-\sfrac72 ,\sfrac72 ]^m)$, we have: 
\begin{align*}
    \mathrm{Lip} ({\cN}|_{\cM^\prime}) & \leq C\mbold{0}^{\alpha_\bL} \mbox{ and } \|{\cN}|_{\cM^\prime}\|_0 \leq C \mbold{0}^{\sfrac{1}{2m}},\\
    \cH^m(\cM^\prime\setminus \cK) + \|\mathbf{T}_F -T\| (\bp^{-1} (\cM^\prime)) & \leq C \mbold{0}^{1+\alpha_\bL} \mbox{ and }
    \int_{\cM^\prime} |D{\cN}|^2 \leq C \mbold{0}.
\end{align*}
for a constant $C = C (\alpha_\be,\alpha_\bh,M_0,N_0,C_\be,C_\bh)>0$. Moreover, $\cN \equiv Q\a{0}$ on $\bPhi(\bS)\cap\cM^\prime$ and $  \int_{\cM^\prime}|\cN|^2 \leq C \mbold{0}$.
\end{corollary}

We now state some properties on the normal Lipschitz approximations concerning their ``splitting" behavior on stopping cubes.

\begin{proposition}[Splitting, Proposition 14.4 of \cite{delellis2021allardtype}]\label{P:SplitingNA}
Under the same assumptions and notation as \Cref{T:LocalNA}, there is a constant $C= C(m,n,\bar{n}, M_0, \delta_1)>0$ with the following property. If $L\in\sW^{e}$, $q\in \pi_0$ with $\dist(q,L) \leq 4\sqrt{m}\ell(L)$, $\bball{q}{\ell(L)/4}\subset\obball{2}$, and $\Omega:=\bPhi(\bball{q}{\ell(L)/4})$, then we have that
\begin{align}
    C\mbold{0} \ell(L)^{m+2-2\delta_1} \leq \ell(L)^m\bE(T,\oball{L}^{\flat}) \leq C\int_\Omega |D\cN|^2,\label{E:SplittingNA1}\\ 
    \int_{\cL} |D\cN|^2 \leq C\ell(L)^m\bE(T,\oball{L}^{\flat}) \leq C\ell(L)^{-m}\int_\Omega|\cN|^2.\label{E:SplittingNA2}
\end{align}
\end{proposition}

\subsection{Intervals of flattening and definition of \texorpdfstring{$\bI$}{Lg}}

We begin by defining the intervals of flattening, which allow us to transition between center manifolds whenever they cease to be good approximations for the ``average" of the sheets of $T$. Before that, we make a few remarks that are important for comparison with the interior setting and that also highlight the significance of the excess decay available for one-sided points.

\begin{remark}
Our approach closely follows the strategy in \cite{delellis2021allardtype}, as they also do not have an \emph{a priori} almost-quadratic decay of the excess. In contrast, when dealing with two-sided boundary points, the authors of \cite{DDHM} and \cite{nardulli2022density} establish almost quadratic excess decay through a blow-up argument combined with strong properties of the limiting two-sided Dir-minimizing functions. The availability of such a decay allows them to bypass the need for intervals of flattening.
\end{remark}

\begin{remark}
In the interior case (\cite{DS5}), analogous to the boundary setting, the center manifold can only be constructed when the excess is below a certain threshold, namely $\varepsilon_{cm}$ given in \Cref{A:CM}; for the interior, see \cite[Eq. 2.4]{DS5}. This requirement explains the gaps that may exist between the intervals of flattening in \cite{de2023fineI}: At scales within these gaps, the excess may be too large to construct the center manifold. However, in our setting, this issue does not arise, as we rely on an excess decay theorem (with exponent $\alpha \in (0,1)$) established in \cite{ian2024uniqueness}).
\end{remark}

\begin{remark}
We note that \cite{delellis2021allardtype} will be frequently referenced, where the setting assumes $\Sigma = \mathbb{R}^{m+n}$. Nevertheless, it is straightforward to verify that all the results remain valid for an arbitrary $\Sigma \subset \mathbb{R}^{m+n}$ of class $C^{3,\kappa}$, taking into account the errors we explicitly state below.
\end{remark}

\subsubsection{Interval of flattening}\label{S:IntervalsFlattening}

Assume \Cref{A:CM}, \Cref{A:ParametersCM}, and \Cref{A:NA}, with $p\in\Gamma$ in place of $0$, and fix 
\begin{equation*}
    \max\left\{\bE^\flat (T, \ball{p}{2}),\bA_{p,\Sigma}^2, \bA_{p,\Gamma}^2\right\} =: \mbold{p} < \varepsilon_{cm}.
\end{equation*}

We set $t_1 = 2$, $\overline{\cM}_{p,1} = \cM_{p,1}$ and $\overline{\cN}_{p,1} = \cN_{p,1}$ to be the center manifold and the $\cM_{p,1}$-normal approximation given by \Cref{T:CM} and \Cref{T:LocalNA} when applied the translation of $T$ by $p\in\Gamma$, i.e., $(\iota_{p,1})_\sharp T$, respectively. 

Denote by $\sW^{p,1}$ the Whitney decomposition as described in \Cref{S:1sided-CM-NA}. Notice that the construction of the intervals of flattening below depend on $p$, we will however omit this dependence for the sake of notation. 

We now assume by induction that $t_{j-1}>0$ is defined with $t_{i-1} < t_i$, for any $i\leq j-1$ and consider:
\begin{itemize}
    \item $T_{p,j} := \left(\left( \iota_{p,t_j}\right)_\sharp T\right)\res\oball{2}$, $\overline{\Sigma}_{p,j}:=\iota_{p,t_j}\left( \Sigma\right)\cap\oball{2}$, $\overline{\Gamma}_{p,j} :=\iota_{p,t_j}\left( \Gamma\right)\cap\oball{2}$, $\overline{\cM}_{p,j}$ and $\overline{\cN}_{p,j}$ to be the center manifold and the normal approximation given by \Cref{T:CM} and \Cref{T:LocalNA} now applied to the current $T_{p,j}$, the boundary $\overline{\Gamma}_{p,j}$, and the ambient manifold $\overline{\Sigma}_{p,j}$. We also denote by $\sW^{p,j}$ the Whitney decomposition coupled with $\overline{\cM}_{p,j}$ as described above.
\end{itemize}

Next, we inductively define:
\begin{equation*}
    t_{j} = t_{j-1} \max\biggl(\left \{ 64\sqrt{m}\ell(L): L\in\sW^{p,j}\text{ and } 64\sqrt{m}\ell(L) \geq \dist(0,L)\right\}\cup \left\{0\right\}\biggr).
\end{equation*}

Observe that $\sfrac{t_j}{t_{j-1}} \leq 2^{-5}$ which ensures that $[t_{j-1}, t_j]$ is a nontrivial interval. If $t_j = 0$, we stop the induction; otherwise we keep going and end up with infinitely many intervals. We also fix the following notations:
\begin{equation*}
    \cM_{p,j} := t_j\overline{\cM}_{p,j}, \quad  \cN_{p,j}(q) := t_j \overline{\cN}_{p,j}\left(\frac{q}{t_j}\right),  \quad \Gamma_{p,j} := t_j\overline{\Gamma}_{p,j},  \mbox{ and } \Sigma_{p,j} := t_j\overline{\Sigma}_{p,j}.
\end{equation*}

Notice that $\Gamma_{p,j} = \left(\Gamma\cap \ball{p}{t_j}\right) - p$ and $\Sigma_{p,j} = \left(\Sigma\cap \ball{p}{t_j}\right) - p$.
\begin{lemma}
Let $L$ be a cube such that $L\in\sW^{p,j}\text{ and } 64\sqrt{m}\ell(L) \geq \dist(0,L)$ and $64\sqrt{m}\ell(L) = \max\left(\left \{ 64\sqrt{m}\ell(L): L\in\sW^{p,j}\text{ and } 64\sqrt{m}\ell(L) \geq \dist(0,L)\right\}\cup \left\{0\right\}\right)$. Then $L\in(\sW^{p,j})^{\be}$, i.e., it has to stop for the excess.
\end{lemma}
\begin{proof}
Observe that if $L$ is in $(\sW^{p,j})^{\mathbf{n}}$, then there is a neighboring cube $H$ of double side length that also belongs to $\sW^{p,j}$ and it is straightforward to check that it satisfies $H\in\sW^{p,j}\text{ and } 64\sqrt{m}\ell(H) \geq \dist(0,H)$ too, thus violating the maximility of $L$. Since $L$ satisfies $64\sqrt{m}\ell(L)\geq \dist(0,L)$, $L$ is a boundary cube and thus cannot belong to $(\sW^{p,j})^{\bh}$, see \cite[Lemma 10.4]{delellis2021allardtype}. 
\end{proof}

\subsubsection{Universal frequency function at one-sided flat points}\label{S:UniversalFreqDef}

Before defining the frequency function, we first introduce a well-suited distance function $d$ that accounts for the curvature of $\Gamma$ in our calculations. Specifically, $d$ serves as the direction of the vector field used in the inner and outer variations, requiring $\nabla d$ to be tangent to $\Gamma$. The function $d$ is constructed in the following lemma.

\begin{lemma}[Good distorted distance function]\label{L:DistFnc}
There exists a continuous function $d:\oball{1}\times\oball{1}\to \R^+$ which we will call \emph{good distorted distance function} satisfying, for any $q\in\Gamma_{p,1}$, the following:
\begin{enumerate}[\upshape (i)]
    \item $x\mapsto d(x, q)$ is $C^2$ on $\oball{1}\setminus\{q\}$,

    \item $d(x,q) = |x-q| +  O\left(\mbold{0} |x-q|^2\right)$,

    \item $\nabla_x d(x,q) = \frac{x-q}{|x-q|} +  O(\mbold{0} |x-q|)$,

    \item $D^2_x d(x,q) = |x-q|^{-1}\left(\mathrm{Id} - |x-q|^{-2}(x-q)\otimes( x-q)\right) + O(\mbold{0} )$,

    \item $\nabla_x d(x,q)$ is tangent to $\Gamma_{p,1}$ at $x\in \Gamma_{p,1}$.
\end{enumerate}
\end{lemma}
\begin{proof}
The proof is given in \cite[Lemma 4.25]{DDHM}. Even though slightly worse estimates are stated, i.e., there is no $\mbold{0}$ in \cite[Lemma 4.25]{DDHM}, their proof is robust enough and immediately gives the improvements that we claim in this lemma. 
\end{proof}

We first fix a cut-off function as follows:
\begin{equation*}
\phi (t) :=
\left\{
\begin{array}{ll}
1, &\mbox{for $0\leq t \leq \frac{1}{2}$,}\\
2 (1-t), &\mbox{for $\frac{1}{2}\leq t \leq 1$,}\\
0, &\mbox{for $t\geq 1$.}
\end{array}\right.
\end{equation*}

Having fixed a good distorted distance function $d$, a cut-off $\phi$, and the intervals of flattening, for any $q\in\Gamma_{p,j}$ and $r\in (t_{j+1}, t_j]$, we define the following three quantities:
\begin{align*}
    \bD_p(q,r)=\bD_p(\phi,d,q,r) & := \hphantom{-}\int_{\cM_{p,j}} \phi\left(\frac{d(x,q)}{r}\right)|D\cN_{p,j}|^2(x)\operatorname{d}\cH^m(x),\\
    \bH_p(q,r) = \bH_p(\phi,d,q,r) &:=-\int_{\cM_{p,j}} \phi'\left(\frac{d(x,q)}{r}\right) |\nabla_{\cM_{p,j}} d (x,q)|^2 \frac{|\cN_{p,j}(x)|^2}{d(x,q)}\operatorname{d}\cH^m(x),\\
    \bS_p(q,r) = \bS_p(\phi,d,q,r) &:=\int_{\cM_{p,j}} \phi\left(\frac{d(x,q)}{r}\right) |\cN_{p,j}(x)|^2\operatorname{d}\cH^m(x).
\end{align*}

Next, we finally define the \emph{universal frequency function} as follows:
\begin{equation*}
    \bI_p(q,r) = \bI_p(\phi,d,q,r) :=\frac{r\bD_p(q,r)}{\bH_p(q,r)}, \quad \mbox{ for any }r\in(0,2], q\in\oball{2}.
\end{equation*}

We will also fix $d_j (x,q) := t_j^{-1}d(t_jx, t_jq)$ and $q\in\oball{2}$ to define the rescaled functions:
\begin{align*}
    \obD_{p,j}(q,r)=\obD_{p,j}(\phi,d,q,r) & := \hphantom{-}\int_{\ocM_{p,j}} \phi\left(\frac{d_j(x,q)}{r}\right)|D\ocN_{p,j}|^2(x)\operatorname{d}\cH^m(x),\\
    \obH_{p,j}(q,r) = \obH_{p,j}(\phi,d,q,r) &:=-\int_{\ocM_{p,j}} \phi'\left(\frac{d_j(x,q)}{r}\right) |\nabla_{\ocM_{p,j}} d_j (x,q)|^2 \frac{|\ocN_{p,j}(x)|^2}{d_j(x,q)}\operatorname{d}\cH^m(x),\\
    \obS_{p,j}(q,r) = \obS_{p,j}(\phi,d,q,r) &:=\int_{\ocM_{p,j}} \phi\left(\frac{d_j(x,q)}{r}\right) |\ocN_{p,j}(x)|^2\operatorname{d}\cH^m(x).
\end{align*}

We can then straightforwardly check that:
\begin{align}
    \obD_{p,j}\left(t_j^{-1}q,t_j^{-1}r\right) = t_j^{-m}\bD_p(q,r), \qquad \partial_r\obD_{p,j}\left(t_j^{-1}q,t_j^{-1}r\right) = t_j^{-m+1}\partial_r\bD_p(q,r), \label{E:rescalingD}\\
    \obH_{p,j}\left(t_j^{-1}q,t_j^{-1}r\right) = t_j^{-m-1}\bH_p(q,r), \qquad \partial_r\obH_{p,j}\left(t_j^{-1}q,t_j^{-}r\right) = t_j^{-m}\partial_r\bH_p(q,r),\label{E:rescalingH}\\
    \obS_{p,j}\left(t_j^{-1}q,t_j^{-1}r\right) = t_j^{-m-2}\bS_p(q,r), \qquad \partial_r\obS_{p,j}\left(t_j^{-1}q,t_j^{-}r\right) = t_j^{-m-1}\partial_r\bS_p(q,r). \label{E:rescalingS}
\end{align}

We also define the following quantities:
\begin{align*}
    \bE_p(q,r) &:= \Scale[0.999]{-\frac1r\int_{\cM_j}\phi^\prime\left(\frac{d(x,q)}{r}\right)\sum_i\langle (\cN_{p,j})_i(x), D(\cN_{p,j})_i(x)\nabla_{\cM_{p,j}} d(x,q)\rangle\operatorname{d}\cH^m(x),} \\
    \bG_{p,j}(q,r) &:= \Scale[0.999]{-\frac{1}{r^2}\int_{\cM_{p,j}}\phi\left(\frac{d(x,q)}{r}\right)\frac{d(x,q)}{|\nabla_{\cM_{p,j}} d(x,q)|^2}|D\cN_{p,j}\cdot\nabla_{\cM_{p,j}}d(x,q)|^2\operatorname{d}\cH^m(x),}\\
    \obE_{p,j}(q,r) &:= \Scale[0.999]{ -\frac1r\int_{\ocM_{p,j}}\phi^\prime\left(\frac{d_j(x,q)}{r}\right)\sum_i\langle (\ocN_{p,j})_i(x), D(\ocN_{p,j})_i(x)\cdot\nabla_{\ocM_{p,j}} d_j(x,q)\rangle\operatorname{d}\cH^m(x),}\\
    \obG_{p,j}(q,r) &:= \Scale[0.999]{-\frac{1}{r^2}\int_{\ocM_{p,j}}\phi\left(\frac{d_j(x,q)}{r}\right)\frac{d_j(x,q)}{|\nabla_{\ocM_{p,j}} d(x,q)|^2}|D\ocN_{p,j}\cdot\nabla_{\ocM_{p,j}}d_j(x,q)|^2\operatorname{d}\cH^m(x),}
\end{align*}
then, we also have the following rescalings:
\begin{equation}\label{E:rescalingE-G}
    \obE_{p,j}(t_j^{-1}q,t_j^{-1}r) = t_j^{-m}\bE_p(q,r) \mbox{ and }\obG_{p,j}(t_j^{-1}q,t_j^{-1}r) = t_j^{-m+1}\bG_p(q,r).
\end{equation}

\subsection{Radial derivatives in the realm of the frequency function}

The main goal of this subsection is to derive a series of inequalities involving the radial derivatives of the various quantities defined above. These inequalities will be central in computing the radial derivative of $\bI$ and in establishing further properties of $\bI$ (such as its almost monotonicity), which in turn encapsulates significant information about the current $T$.

\begin{lemma}[Radial variations and list of bounds]\label{L:FirstEqsFreqFnc}
Let $T, \Gamma$, and $\Sigma$, satisfy Assumptions \ref{A:CM}, \ref{A:ParametersCM}, and \ref{A:NA}, with $p$ in place of $0$. Then either $T = Q\a{p+\cM_{p,j}}$ in a neighborhood of $p$, for some $j\in\mathbb{N}$, or $\bH_p(p,s) >0 $ and $\bD_p(p,s)>0$, for every $s\in (0,2]$. In the latter case, there exist $C>0$ and $\tau>0$ (which are not dependant on $p, j$, $T$, nor $\varepsilon_{na}$) such that, for all $r>0$, we have
\begin{equation}\label{E:L:FirstEqsFreqFnc:1}
\begin{aligned}
    \bH_p(p,r) &\leq C r\bD_p(p,r), \quad \bI_p(p,r) \geq C^{-1},  \\ \partial_r\bS_p(p,r) &\leq Cr\bD_p(p,r),
    \mbox{ and } \bS_p(p,r)\leq C r^2\bD_p(p,r),    
\end{aligned}    
\end{equation}

Additionally, we obtain that
\begin{equation}\label{E:Dbound}
    \bD_p(p,r) \leq C r^{m+2\kappa}.
\end{equation}

We also have the following formulas for any $r\in (t_{j+1},t_j)$:
\begin{align}
    \partial_r\bD_p(p,r) = -\int \phi^\prime\left(\frac{d(x,p)}{r}\right)\frac{d(x,p)}{r^2}|D\cN_{p,j}(x)|^2\operatorname{d}\cH^m(x),\label{E:Dprime}\\
    \left| \partial_r\bH_p(p,r)- \frac{m-1}{r}\bH_p(p,r) + 2 \bE_p(p,r) \right| \leq C \mbold{p,j} \bH_p(p,r).\label{E:Hprime}
\end{align}

Moreover, for any $r\in (t_{j+1},t_j)$, we have:
\begin{align}
    |\partial_r\bD_p(p,r)| &\leq Cr^{-1}\bD_p(p,r),\label{E:D'bound}\\
    |\bD_p(p,r) - \bE_p(p,r) | &\leq  C\mbold{p} ^{\tau}\bD_p(p,r)^{1 + \tau} + C\mbold{p} t_j^{-2+2\kappa}\bS_p(p,r) \label{E:D-E},\\
    |\partial_r\bD_p(p,r) -\frac{m-2}{r}\bD_p(p,r)-2\bG_p(p,r)| &\leq C\mbold{p} ^{\tau}\bD^{\tau}_p(p,r)\left( \frac{\bD_p(p,r)}{r}+ \partial_r\bD_p(p,r)\right) \nonumber \\
    &\quad + C\mbold{p} t_j^{2\kappa-1}\bD_p(p,r), \label{E:Dprime-D-G}\\
    \biggl |\frac{1}{\bD_p(p,r)} - \frac{1}{\bE_p(p,r)} \biggr| &\leq C\mbold{p}^{\tau}\bD_p(p,r)^{\tau - 1} + C\mbold{p} t_j^{-2+2\kappa}\frac{\bS_p(p,r)}{\bD_p(p,r)^2}.\label{E:1/D-1/Ebound}
\end{align}
\end{lemma}
\begin{proof}
These estimates depend on the outer variation and an inner variation on specific vector fields as is well-known. We denote by $\bp_j$ the projection onto the center manifold $\cM_{p,j}$.
The vector fields are the following:
\begin{align*}
    X_o^j(q):=\phi\left(\frac{d_j(\bp_j(q))}{r}\right) (q- \bp_j(q)), X_i^j(q):=-Y_j(\bp_j(q))  \mbox{ with } Y_j:= \frac{1}{2}\phi\left(\frac{d_j}{r}\right)\frac{\nabla d_j}{|\nabla d_j|^2},
\end{align*}
where in the formula above, $\nabla$ denotes the tangent gradient to $\cM_{p,j}$. The vector field $X_o^j$ vanishes at $\Gamma$ and the vector field $X_i^j$ is tangent to $\cM_{p,j}$ and to $\Gamma_{p,j}$ by properties of $d$.
Let us drop the dependence on $p$ in the notation and assume without loss of generality that $p=0$. The proof of \eqref{E:L:FirstEqsFreqFnc:1} for the rescaled quantities is given in \cite[Proposition 9.4]{DDHM} from which, together with the rescalings above, i.e., \eqref{E:rescalingD}, \eqref{E:rescalingS}, and \eqref{E:rescalingH}, we get \eqref{E:L:FirstEqsFreqFnc:1}. In order to prove \eqref{E:Dbound}, we note that \cite[Eq. (16.29)]{delellis2021allardtype} gives $\obD_j(r) \leq C_0t_j^{2\kappa}r^{m+2-2\alpha_\be}$, thus
\begin{equation*}
    \bD(r) \overset{\eqref{E:rescalingD}}{=} t_j^{m}\obD_j(t_j^{-1} r)\leq C_0t_j^{m+2\kappa}(t_j^{-1}r)^{m+2-2\alpha_\be} \overset{r<t_j}{<} C_0r^{m+2\kappa}.
\end{equation*}

The proof of \eqref{E:Dprime} is a direct computation. For \eqref{E:Hprime}, we can mimic the proof of \cite[Proposition 16.2]{delellis2021allardtype} using the slightly improved estimates for $d$ given by \Cref{L:DistFnc}. Inequality \eqref{E:D'bound} is a consequence of the definition of $\bD$ and \eqref{E:Dprime}. To prove \eqref{E:D-E} we will do as in \cite[Section 9]{DDHM}. Indeed,
\begin{equation*}
    |\obD_j(r) - \obE_j(r)| \leq \sum_{j=1}^5 |\mathrm{Err}_j^o|,
\end{equation*}
where $\mathrm{Err}_k^o$ are as in \cite[Section 9]{DDHM} and 
\begin{equation*}
    |\obD_j^\prime(r) - O(t_j^{\tau_1})\obD_j(r) - 2\obG_j(r)| \leq \frac{2}{r}\sum_{k=1}^5 |\mathrm{Err}_k^i|,
\end{equation*}
where $\mathrm{Err}_k^i$ are as in \cite[Section 9]{DDHM}. Now, revisiting the argument of \cite[Proposition 9.14]{DDHM}, one can easily keep track that $\mbold{0}$ can be carried throughout the proof to ensure 
\begin{equation}\label{E:D-E:with-bars}
    |\obD_j(r) - \obE_j(r)| \leq C\mbold{0}^{\tau_1}\obD_j(r)^{1+\tau_1} + C\mbold{0} t_j^{2\kappa}\obS_j(r), \mbox{ and }
\end{equation}
\begin{equation}\label{E:Dprime-D-G:with-bars}
    |\obD_j^\prime(r) - \frac{m-2}{r}\obD_j(r) - 2\obG_j(r)| \leq C\mbold{0}^{\tau_1}\obD_j(r)^{\tau_1}\left( \frac{\obD_j(r)}{r}+ \obD_j^\prime(r)\right) + C\mbold{0} t_j^{2\kappa}\obD_j(r).
\end{equation}

We use now the rescalings \eqref{E:rescalingD}, \eqref{E:rescalingS}, and \eqref{E:rescalingE-G}, in \eqref{E:D-E:with-bars} to get
\begin{equation*}
\begin{aligned}
    t_j^{-m}|\bD(r) - \bE(r)| &= |\obD_j(t_j^{-1}r) - \obE_j(t_j^{-1}r)| \leq C\mbold{0}^{\tau_1}\obD_j(t_j^{-1}r)^{1+\tau_1} + C\mbold{0} t_j^{2\kappa}\obS_j(t_j^{-1}r)\\
    &= C\mbold{0}^{\tau_1}t_j^{-m(1+\tau_1)}\bD(r)^{1+\tau_1} + C\mbold{0} t_j^{-m-2+2\kappa}\bS(r).
\end{aligned}
\end{equation*}

Using \eqref{E:Dbound} and $r\leq t_j$, we get $t_j^{-m}\bD(r) \leq Ct_j^{-m}r^m\bD(r)^{1-\frac{m}{m+2\kappa}} \leq C\bD(r)^{1-\frac{m}{m+2\kappa}}$ which together with the last displayed inequality provide
\begin{equation}\label{E:bound|D-E|}
\begin{aligned}
    |\bD(r) - \bE(r) | &\leq C\mbold{0}^{\tau_1}t_j^{-m\tau_1}\bD(r)^{1+\tau_1} + C\mbold{0} t_j^{-2+2\kappa}\bS(r)\\
    &= C\mbold{0} ^{\tau_1}\left(t_j^{-m}\bD(r)\right)^{\tau_1}\bD(r) + C\mbold{0} t_j^{-2+2\kappa}\bS(r)\\
    &\leq C\mbold{0} ^{\tau_1}\bD(r)^{1 + \tau_1\left(1-\frac{m}{m+2\kappa}\right)} + C\mbold{0} t_j^{-2+2\kappa}\bS(r) .
\end{aligned}
\end{equation}

We can choose choose $\tau_1$ slightly smaller accordingly with the exponent $\tau_1\left(1-\frac{m}{m+2\kappa}\right)$ to verify \eqref{E:D-E}. We argue analogously to get \eqref{E:Dprime-D-G} from \eqref{E:Dprime-D-G:with-bars}, we use \eqref{E:Dprime-D-G:with-bars} and the rescalings \eqref{E:rescalingD}, \eqref{E:rescalingE-G} to obtain that
\begin{equation*}
\begin{aligned}
    t_j^{-m+1}|\bD^\prime(r) &- \frac{m-2}{r}\bD(r) - 2\bG(r)| = |\obD_j^\prime(t_j^{-1}r) - \frac{(m-2)t_j^{2\kappa}}{r}\obD_j(t_j^{-1}r) - 2\obG_j(t_j^{-1}r)| \\
    &\leq C\mbold{0} ^{\tau_1}\obD(t_j^{-1}r)^{\tau_1}\left( \frac{\obD(t_j^{-1}r)}{t_j^{-1}r}+ \obD^\prime(t_j^{-1}r)\right) + C\mbold{0} t_j^{2\kappa}\obD(t_j^{-1}r)\\
    &= t_j^{-m+1}\left(C\mbold{0} ^{\tau_1}t_j^{-m\tau}\bD(r)^{\tau_1}\left( \frac{\bD(r)}{r}+ \bD^\prime(r)\right) + C\mbold{0} t_j^{2\kappa-1}\bD(r)\right).
\end{aligned}
\end{equation*}

We argue as in \eqref{E:bound|D-E|} to obtain \eqref{E:Dprime-D-G}, for a possibly smaller $\tau_1>0$. To prove \eqref{E:1/D-1/Ebound}, we apply \eqref{E:L:FirstEqsFreqFnc:1} and the fact that $r\leq t_j$ to achieve, for any $\sigma\in (0,1)$, that
\begin{equation*}
    t_j^{2\sigma -2}\bS(r) \leq t_j^{2\sigma-2} Cr^2\bD(r) \leq Cr^{2\sigma}\bD(r).
\end{equation*}

The last displayed inequality guarantees
\begin{equation*}
\begin{aligned}
    |\bD(r) - \bE(r) | &\overset{\eqref{E:bound|D-E|}}{\leq} C\mbold{0} ^{\tau}\bD(r)^{1 + \tau} + C\mbold{0} t_j^{-2+2\kappa}\bS(r) \\
    &\leq C\mbold{0} ^{\tau}\bD(r)^{1 + \tau}+C\mbold{0} r^{\tau}\bD(r)\\
    &\overset{\eqref{E:Dbound}}{\leq} C\mbold{0} ^{\tau}r^{\tau(m+2\kappa)}\bD(r)+C\mbold{0} r^{\tau}\bD(r)\\
    &\leq C\mbold{0} ^{\tau}r^\tau \bD(r),
\end{aligned}
\end{equation*}
which in turn ensure that $\bD(r)/2 \leq \bE(r) \leq 2\bD(r)$ provided $r$ is small enough. Hence, we get that \eqref{E:bound|D-E|} can be turned into
\begin{equation*}
    \biggl |\frac{1}{\bD(r)} - \frac{1}{\bE(r)} \biggr| \leq C\mbold{0} ^{\tau}\bD(r)^{\tau - 1} + C\mbold{0} t_j^{-2+2\kappa}\frac{\bS(r)}{\bD(r)^2} .
\end{equation*}
\end{proof}

\subsection{Almost monotonicity of \texorpdfstring{$\bI$}{Lg} at one-sided flat singular points}

We are now in a position to prove the almost monotonicity of the universal frequency function $\bI$, with improved bounds compared to \cite{delellis2021allardtype} and in arbitrary dimensions.

\begin{theorem}[Almost monotonicity and jumps of $\bI$]\label{T:AlmMonotFreq}
Let $T, \Gamma$, and $\Sigma$, satisfy Assumptions \ref{A:CM}, \ref{A:ParametersCM}, and \ref{A:NA}, with $p$ in place of $0$. Then either $T = Q\a{p+\cM_{p,j}}$ in a neighborhood of $q$, for some $j\in\mathbb{N}$, or $\bH_p(p,r) >0 $ and $\bD_p(p,r)>0$, for every $r>0$. In the latter case, $\bI$ is absolutely continuous on each $(t_{j+1}, t_j)$ and there exist $C,\tau>0$ (not depending on $p,j,T$, nor $\varepsilon_{na}$) such that for a.e. $r\in (t_{j+1}, t_j)$ it holds
\begin{equation}\label{E:LowerBoundDerivI}
\frac{\mathrm{d}}{\mathrm{d}r}\left(
    \log \bI_p(p,r) + C\mbold{p} ^\tau\bD_p(p,r)^\tau -C\mbold{p} t_j^{-2+2\kappa}\frac{\bS_p(p,r)}{\bD_p(p,r)} \right) \geq - C\mbold{p} ^{\tau}r^{-1+\tau}.
\end{equation}

Morever, we have a bound on the jumps of $\bI$ given by
\begin{align}
    \biggl| \lim_{s\downarrow t_j}\bI_p(p,s) - &\lim_{s\uparrow t_j}\bI_p(p,s)\biggr| \leq C\mbold{p}^{\tau}t_j^{\tau}, \quad \sum_j \biggl| \lim_{s\downarrow t_j}\bI_p(p,s) - \lim_{s\uparrow t_j}\bI_p(p,s)\biggr| \leq C\mbold{p}^{\tau} ,\label{E:jumpsI}\\
    &\mbox{ and }\bI_p(p,r) \geq e^{- C\mbold{p} ^{\tau}t_j^\tau} \bI_p(p,a),\label{E:Almost-Monot-I}
\end{align}
where the last inequality holds for $r\in (t_{j+1}, t_j)$ and any $a\in (0,r)$.
\end{theorem}
\begin{remark}
Due to the excess decay (\cite{ian2024uniqueness}) at one-sided points, we obtain excellent estimates for the jumps of $\bI$, as stated in \eqref{E:jumpsI}, which mitigate several difficulties compared to the interior case.  
In the interior setting, where the uniqueness of tangent cones and excess decay are not known \emph{a priori}, the authors must establish a significantly more delicate and labor-intensive BV-estimate for $\log(\bI + 1)$, ensuring similar bounds on the jumps of $\bI$ (cf. \cite[Proposition 6.2]{de2023fineI}). Still, in the interior case, the authors in \cite{krummel2023analysisI} introduced the \emph{planar frequency}, for which they are able to establish a lower bound strictly greater than one. This, in turn, yields an excess decay estimate (cf. \cite[Corollary 4.3]{krummel2023analysisI}), and represents a key distinction arising from the use of the planar frequency.
\end{remark}
\begin{proof}
We drop the dependences on $p$ and reduce to $p=0$ as in \Cref{L:FirstEqsFreqFnc}. Applying now \eqref{E:Hprime}, we get
\begin{equation*}
    \frac{\mathrm{d}}{\mathrm{d}r}
    \log \bI(r) = \frac{1}{r} + \frac{\bD^\prime(r)}{\bD(r)} - \frac{\bH^\prime(r)}{\bH(r)} = \frac{m-2}{r}\frac{\bD^\prime(r)}{\bD(r)} - \frac{2\bE(r)}{\bH(r)} - C\mbold{0}.
\end{equation*}

Plugging the displayed equality above to \eqref{E:1/D-1/Ebound} and \eqref{E:Dprime-D-G}, we get
\begin{equation*}
\begin{aligned}
    \frac{\mathrm{d}}{\mathrm{d}r}
    \log \bI(r) &\overset{\eqref{E:1/D-1/Ebound}}{\geq} \frac{\bD^\prime(r)}{\bE(r)} - C\mbold{0} ^{\tau}\bD(r)^{\tau - 1}\bD^\prime(r) - C\mbold{0} t_j^{-2+2\kappa}\frac{\bS(r)\bD^\prime(r)}{\bD(r)^2} \\
    &\quad - \frac{2\bE(r)}{\bH(r)} - C\mbold{0} \\
    &\overset{\eqref{E:Dprime-D-G}}{\geq} \frac{2\bG(r) - C\mbold{0} ^{\tau}\bD(r)^{\tau}\left( \frac{\bD(r)}{r}+ \bD^\prime(r)\right) - C\mbold{0} t_j^{2\kappa-1}\bD(r)}{\bE(r)} \\
    &\quad - C\mbold{0} ^{\tau}\bD(r)^{\tau - 1}\bD^\prime(r) - C\mbold{0} t_j^{-2+2\kappa}\frac{\bS(r)\bD^\prime(r)}{\bD(r)^2} \\
    &\quad - \frac{2\bE(r)}{\bH(r)} - O(\mbold{0} ) \\
    &\geq \frac{- C\mbold{0} ^{\tau}\bD(r)^{\tau}\left( \frac{\bD(r)}{r}+ \bD^\prime(r)\right) - C\mbold{0} t_j^{2\kappa-1}\bD(r)}{\bE(r)} \\
    &\quad - C\mbold{0} ^{\tau}\bD(r)^{\tau - 1}\bD^\prime(r)  - C\mbold{0} t_j^{-2+2\kappa}\frac{\bS(r)\bD^\prime(r)}{\bD(r)^2} - C\mbold{0},
\end{aligned}
\end{equation*}
where in the last inequality we applied Cauchy-Schwarz to have $\bG(r)\bH(r) \geq \bE(r)^2$, and hence $2\bG(r)/\bE(r) - 2\bE(r)/\bH(r)\geq 0$. By \eqref{E:D-E}, we have $\bD(r)/\bE(r) \leq C_0$ which ensures, by the last displayed inequality, that 
\begin{equation*}
\begin{aligned}
    \frac{\mathrm{d}}{\mathrm{d}r}
    \log \bI(r) &\geq - C\mbold{0} ^{\tau}\biggl( \frac{\bD(r)^\tau}{r}+ \frac{\bD^\prime(r)}{\bD(r)^{1-\tau}}\biggr) - C\mbold{0} t_j^{2\kappa-1} - C\mbold{0} ^{\tau}\frac{\bD^\prime(r)}{\bD(r)^{1-\tau}} \\
    &\quad - C\mbold{0} t_j^{-2+2\kappa}\frac{\bS(r)\bD^\prime(r)}{\bD(r)^2} - C\mbold{0}.
\end{aligned}
\end{equation*}

We absorb the $C\mbold{0}$ term in the $\mbold{0}t_j^{2\kappa-1}$ term, $-\frac{\bS(r)\bD^\prime(r)}{\bD(r)^2} = \frac{\mathrm{d}}{\mathrm{d}r}\frac{\bS(r)}{\bD(r)} - \frac{\bS^\prime(r)}{\bD(r)}$ and $\frac{\mathrm{d}}{\mathrm{d}r}\bD(r)^\tau = \tau\frac{\bD^\prime(r)}{\bD(r)^{1-\tau}}$, we derive
\begin{equation*}
\begin{aligned}
    \frac{\mathrm{d}}{\mathrm{d}r}\biggl(
    \log \bI(r) + C\mbold{0} ^\tau\bD(r)^\tau - C\mbold{0} t_j^{-2+2\kappa}\frac{\bS(r)}{\bD(r)} \biggr) &\geq - C\mbold{0} ^{\tau}\frac{\bD(r)^\tau}{r} - C\mbold{0} t_j^{-2+2\kappa}\frac{\bS^\prime(r)}{\bD(r)} \\
    &\quad - C\mbold{0}t_j^{2\kappa-1}.
\end{aligned}
\end{equation*}

Using \eqref{E:Dbound} and \eqref{E:L:FirstEqsFreqFnc:1}, possibly taking $\tau$ smaller, we get 
\begin{equation*}
\begin{aligned}
    \frac{\mathrm{d}}{\mathrm{d}r}\left(
    \log \bI(r) + C\mbold{0} ^\tau\bD(r)^\tau -C\mbold{0} t_j^{-2+2\kappa}\frac{\bS(r)}{\bD(r)} \right) &\geq - C\mbold{0} ^{\tau}r^{-1+\tau} - C\mbold{0}t_j^{2\kappa-1}\\
    &\quad - C\mbold{0} t_j^{-2+2\kappa}r \\
    &\overset{t_j > r}{\geq }  - C\mbold{0} ^{\tau}r^{-1+\tau}.
\end{aligned}
\end{equation*}

This finishes the proof of \eqref{E:LowerBoundDerivI}.
The proof of the first estimate in \eqref{E:jumpsI} is given in \cite[Theorem 15.5, Part II]{delellis2021allardtype}, where the authors do not use $m=2$ in their proof. 
More precisely the desired estimate is \cite[Equation 17.7]{delellis2021allardtype}, which is a consequence of \cite[Proposition 17.1]{delellis2021allardtype} and \cite[Proposition 17.2]{delellis2021allardtype}. The only difference in our case is that we have the additional term coming from $\Sigma$, precisely $|\bI(t_j^+)-\bI(t_j^-)|$ is controlled by $\mathbf{E}(T,6t_j)^{\tau}+ \mathbf{A}_{\Sigma}^{2\tau}$ which is bounded by $\mbold{0}^{\tau}t_j^{\tau\alpha}$ due to the excess decay in \cite{ian2024uniqueness}. Here, the superscripts $+$ and $-$ denote the right- and left-hand side limits, respectively. 

The second estimate of the jumps in \eqref{E:jumpsI} follows immediately from $t_{j+1}/t_{j}\leq \frac{1}{2}$.

We now prove \eqref{E:Almost-Monot-I} using \eqref{E:jumpsI}. Fix $r\in (t_{{j_r}+1}, t_{j_r}]$, $a\in (t_{j_a+1},t_{j_a}]$, $r>a$, and integrate \eqref{E:jumpsI} as follows
\begin{equation}\label{E:proof-of-small-improvements-monoton-I}
\begin{aligned}
    \int_a^r\frac{\mathrm{d}}{\mathrm{d}s}\left(
    \log \bI(s) + C\mbold{0} ^\tau\bD(s)^\tau -C\mbold{0} t_{i(s)}^{-2+2\kappa}\frac{\bS(s)}{\bD(s)} \right) \mathrm{d}s\geq  - C\mbold{0} ^{\tau}(r^\tau-a^\tau),
\end{aligned}
\end{equation}
where $i(s)$ denote the index for which $s\in (t_{j(s)+1}, t_{j(s)})$. We now notice that
\begin{equation*}
\begin{aligned}
    \int_a^r\frac{\mathrm{d}}{\mathrm{d}s}\left(-t_{i(s)}^{-2+2\kappa}\frac{\bS(s)}{\bD(s)} \right) \mathrm{d}s &= - \sum_{i= j_r}^{j_a-1}\left( - t_{i}^{-2+2\kappa}\frac{\bS(t_i^-)}{\bD(t_i^-)} + t_{i+1}^{-2+2\kappa}\frac{\bS(t_{i+1}^+)}{\bD(t_{i+1}^+)}\right)\\
    &\quad - t_{j_r}^{-2+2\kappa}\frac{\bS(r)}{\bD(r)} + t_{j_r}^{-2+2\kappa}\frac{\bS(t_{j_r}^+)}{\bD(t_{j_r}^+)}\\
    &\quad - t_{j_a}^{-2+2\kappa}\frac{\bS(t_{j_a}^-)}{\bD(t_{j_a}^-)} + t_{j_a}^{-2+2\kappa}\frac{\bS(a)}{\bD(a)}.
\end{aligned}
\end{equation*}

We discard the negative terms and we use \eqref{E:L:FirstEqsFreqFnc:1} to turn the last displayed equation into
\begin{equation}\label{E:proof-of-small-improvements-monoton-II}
\begin{aligned}
    \int_a^r\frac{\mathrm{d}}{\mathrm{d}s}\left(-C\mbold{0} t_{i(s)}^{-2+2\kappa}\frac{\bS(s)}{\bD(s)} \right) \mathrm{d}s \leq C\mbold{0} \left(\sum_{i= j_r}^{j_a-1}t_{i}^{2\kappa} +t_{j_r}^{2\kappa} +t_{j_a}^{2\kappa}\right) \leq C\mbold{0}t_{j_r}^{2\kappa},
\end{aligned}
\end{equation}
where we used that $t_{j_a} \leq t_i < r \leq t_{j_r}, i\in\{j_r+1, \ldots, j_a\}$ and $t_{i+1}/t_i < 2^{-5}$ (by construction of the intervals of flattening). We argue analogously, this time using \eqref{E:Dbound}, to obtain 
\begin{equation}\label{E:proof-of-small-improvements-monoton-III}
\begin{aligned}
    \int_a^r\frac{\mathrm{d}}{\mathrm{d}s}\left(C\mbold{0} ^\tau\bD(s)^\tau \right) \mathrm{d}s &= C\mbold{0}^\tau \left(\bD(r)^\tau - \left(\bD(t_{j_r}^+)^{\tau}\right) + \sum_{i= j_r}^{j_a-1}(\bD(t_i^-)^\tau) -(\bD(t_{i+1}^+)^\tau)\right.\\
    &\quad \biggl. (\bD(t_{j_a}^-)^\tau) - \bD(a)^\tau \biggr)\\
    &\leq C\mbold{0}^\tau \left(r^{m+2\kappa} + \sum_{i= j_r}^{j_a} t_{i}^{m+2\kappa}\right) \leq C\mbold{0}^\tau t_{j_r}^{m+2\kappa}.
\end{aligned}
\end{equation}

We plug \eqref{E:proof-of-small-improvements-monoton-II} and \eqref{E:proof-of-small-improvements-monoton-III} into \eqref{E:proof-of-small-improvements-monoton-I} to obtain (possibly reducing $\tau>0$) that 
\begin{equation*}
\begin{aligned}
    \int_a^r\frac{\mathrm{d}}{\mathrm{d}s}\left(
    \log \bI(s) \right) \mathrm{d}s \geq  - C\mbold{0} ^{\tau}(r^\tau-a^\tau) - C\mbold{0}^\tau t_{j_r}^{2\kappa} \geq - C\mbold{0} ^{\tau}t_{j_r}^{\tau}.
\end{aligned}
\end{equation*}

Finally, we have that
\begin{equation*}
    \log \bI(r) - \log \bI(a) =  \int_a^r\frac{\mathrm{d}}{\mathrm{d}s}\left(
    \log \bI(s) \right) \mathrm{d}s + \sum_{i=j_r}^{j_a} \left( \log \bI(t_i^+)- \log \bI(t_i^-) \right).
\end{equation*}

Since $|\log(x)-\log(y)| \leq C\frac{|x-y|}{\min\{|x|,|y|\}}$ and, by \eqref{E:L:FirstEqsFreqFnc:1}, $\bI(s) \geq C^{-1}$, we derive that 
\begin{equation*}
\begin{aligned}
    \sum_{i=j_r}^{j_a} \log \bI(t_i^+)- \log \bI(t_i^-) &\geq - C \sum_{i=j_r}^{j_a} |\bI(t_i^+)- \bI(t_i^-)| \\
    \overset{\eqref{E:jumpsI}}&{\geq} -C \mbold{0}^{\tau}  \sum_{i=j_r}^{j_a} t_i^{\tau} \geq - C\mbold{0}^{\tau}t_{j_r}^{\tau}.    
\end{aligned}
\end{equation*}

Putting the last three displayed inequalities together, we guarantee that $\log \bI(r) - \log \bI(a) \geq - C\mbold{0}^{\tau}t_j^{\tau}$ which in turn ensures \eqref{E:Almost-Monot-I}.
\end{proof}

\section{Uniform lower bound and the frequency regimes}

\subsection{Uniform lower bound for \texorpdfstring{$\bI$}{Lg} at one-sided flat singular points}

In this section, the main result is \Cref{T:UniversalLowerBound-FreqFnc}, which, among other implications, establishes a \emph{uniform} lower bound of $1+\alpha/2$ for the frequency values, where $\alpha = \alpha(m,n,\bar{n},Q) > 0$ is the constant from the excess decay result in \cite{ian2024uniqueness}. This lower bound marks a significant distinction between the interior setting studied in \cite{de2023fineI,de2023fineII,de2023fineIII} and the boundary setting considered here.  

We now lay the groundwork for the proof of this crucial lower bound. For a small constant $\bar{\varepsilon} > 0$ (to be specified later) satisfying $\bar{\varepsilon} > \max\{\bA_{\Sigma}, \bA_{\Gamma}\}$, and adopting the notation from \Cref{S:UniversalFreqDef}, we set  
\begin{equation*}
    \mbold{p,j} := \max\left\{\bE^\flat(T_{p,j}, \oball{2}), \bar{\varepsilon}t_j^{2-2{\alpha_\be}}\right\}.
\end{equation*}

\begin{assumption}\label{A:AreaMin+0FlatSing}
Let $T,\Sigma$ and $\Gamma$ be as in \Cref{A:general} and $p$ be a flat one-sided singular point for $T$.
\end{assumption}
\begin{remark}\label{R:1sided-flat-assumption} Due to \cite{ian2024uniqueness}, we can guarantee for a sufficiently small scale $\bar{r}$ all Assumptions \ref{A:CM}, \ref{A:ParametersCM}, and \ref{A:NA} are satisfies for $r\leq \bar{r}$. We can also ensure that $\bar{\varepsilon}\bar{r}^{2-2{\alpha_\be}} < \varepsilon_{na}$. 
\end{remark}

\subsubsection{Definition of fine blowups}

Let $T$ be as in \Cref{A:AreaMin+0FlatSing} and let $\{t_{j}\}_{k\in\N}$ be the endpoints of the intervals of flattening as constructed in \Cref{S:UniversalFreqDef}. Now, assume that $r_k\leq \bar{r}$, $r_k\in [t_{j(k)+1},t_{j(k)}]$. 
We will also use all notation for the rescaled objects of \Cref{S:UniversalFreqDef}. 

Next, we define the following sequence of $Q$-valued functions: 
\begin{equation*}
    u_k:= \frac{\overline{N}_k \circ \textbf{exp}_{0,\overline{\cM_k}}}{\mathbf{h}_k}, \mbox{ where }\mathbf{h}_k:=\left\|\overline{N}_k \right\|_{L^2\left(\textbf{exp}_{0,\overline{M_k}}\left(\oball{3/2}\cap\pi_k^+\right)\right)},
\end{equation*}
where $\textbf{exp}_{0,\overline{M_k}}$ is the exponential map and $\pi_k$ is the optimal $m$-plane for $\bE^\flat(T_{p,j(k)}, \oball{2})$. 

Under these conditions, there exists a subsequence of $\{r_k\}_k$ and a Dirichlet minimizing map $u \in W^{1,2}(\halfobball{3/2})$ with $\etaa\circ u=0$ and $\left\|u \right\|_{L^2(\halfobball{3/2})}=1$ satisfying
\begin{equation*}
    u_k \rightarrow u \mbox{ strongly on } W^{1,2}_{loc}\left(\halfobball{3/2},\cA_{Q}(\R^n)\right) \cap L^2\left(\halfobball{3/2},\cA_{Q}(\R^n)\right).
\end{equation*}

Here, convergence is understood after applying a suitable rotation that maps the domains of $u_k$ into $\pi_0$.

We also have that $\Gamma_k$ converges to a flat boundary and the function $u$ vanishes on that flat boundary (indeed, this is a direct consequence of the convergence of $\Gamma_k$ and \Cref{D:NA}(iii)). 

\begin{definition}[Fine blowups]\label{D:fine-blowups}
Any function $u$ generated by the procedure above applied to $T_{p,r_k}$ will be called \emph{a fine blowup with respect to $r_k\downarrow 0$ at $p$}. Moreover, we define \emph{the singularity degree of $T$ at the flat singular point $p$} as 
\begin{equation*}
    \bI(T,p) := \inf\left\{ I_u(0):u\mbox{ is a fine blowup at }p\mbox{ w.r.t. to some }r_k\downarrow 0 \right\}.
\end{equation*}
\end{definition}

In the definition above, $I_u(0):=\lim_{r\downarrow 0}I_u(r)$ and $I_u(r)$ is the value of the smoothed frequency function of $u$ at $0$, i.e., 
\begin{equation*}
\begin{aligned}
    I_u(r) := \frac{rD_u(r)}{H_u(r)},& \ D_u(r):= \int |Du(y)|^2\phi\left(\frac{d(y)}{r}\right), \\
    &\mbox{ and } H_u(r) := - \int \frac{|u(y)|^2}{d(y)}\phi^\prime\left(\frac{d(y)}{r}\right)|\nabla d(y)|^2,
\end{aligned}
\end{equation*}
where the cut-off $\phi$ is chosen as before, i.e.,
\begin{equation*}
    \phi(t) := \begin{cases}
        1, & t\in [0,1/2],\\
        2-2t, & t\in [1/2,1],\\
        0, &\mbox{otherwise}.
    \end{cases}
\end{equation*}

\begin{remark}
We also consider a sequence of currents $T_k$ that does not consist solely of rescalings of the same current, as required in some compactness arguments in \cite{de2023fineII}. For our purposes (and also for \cite{de2023fineII}), it suffices that the limit $u$ is a solution to the corresponding linear problem and satisfies the following properties: it is Dirichlet minimizer, average-free (i.e., $\etaa \circ u \equiv 0$), nontrivial, and has a flat boundary with zero boundary data.
\end{remark}

\subsubsection{Main lower bound for the universal frequency function}

We now provide a necessary and sufficient condition for having finitely many intervals of flattening, based on the frequency value. We also prove the uniform lower bound for $\bI(T,p)$ and demonstrate how to derive the decay rate of the excess, depending on $\bI(T,p)$.

\begin{theorem}[Universal lower bound for the frequency function]\label{T:UniversalLowerBound-FreqFnc}
Let $T,\Gamma, \Sigma$, and $p$, satisfy \Cref{A:AreaMin+0FlatSing}. Then, the following holds:
\begin{enumerate}[\upshape (i)]
    \item\label{I:HomogFineBU} for any fine blowup $u$ at $p$, $u$ is $\bI(T,p)$-homogeneous and $\lim_{r\downarrow 0}\bI_p(p,r) = \bI(T,p) = I_u(0)$,
    
    \item\label{I:FreqFnc:HighFreqIFFfiniteIntervalsFlattening} if $\bI(T,p) > 2-{\alpha_\be}$, then $t_{J+1} = 0$ for some $J\in\N$. Conversely, if $t_{J+1} = 0$ for some $J\in\N$, then $\bI(T,p) \geq 2 -\alpha_\be$.
    
    \item\label{I:FreqFnc:UniversalLowerBound} $\bI(T,p) \geq 1 + \alpha/2$, where $\alpha = \alpha(m,n,\bar{n}, Q)>0$ is the excess decay exponent in \cite{ian2024uniqueness}, and $\inf_j\frac{t_{j+1}}{t_j} > 0$ unless $\{t_j\}_j$ is finite.
\end{enumerate}
\end{theorem}

In order to prove the above theorem, we first establish an improved excess decay estimate and an intermediate lower bound (by $1$) for the frequency values.
\begin{lemma}\label{L:Freq-geq-than-1}
Let $T,\Gamma, \Sigma$, and $p$, satisfy \Cref{A:AreaMin+0FlatSing}. For any fine blow-up $u$ at $p$, we have $\lim_{r\to 0}\bI_p(p,r) = I_u(0) = \bI(T,p) \geq 1$ and $u$ being $\bI(T,p)$-homogeneous. Moreover, if $\bI(T,p) > 1$, then
\begin{equation*}
    \bE(T,\oball{r})\leq C t_j ^{\beta}, \mbox{ for any }\beta < \min\{2(\bI(T,p) - 1), 2-2{\alpha_\be}\}\mbox{ and }r\in (t_{j+1}, t_j).
\end{equation*}
\end{lemma}
\begin{remark}
The moreover part of \Cref{L:Freq-geq-than-1} is not used to prove \eqref{I:FreqFnc:UniversalLowerBound} in \Cref{T:UniversalLowerBound-FreqFnc}, we only use \cite{ian2024uniqueness} instead. Hence, \emph{a fortiori}, using \eqref{I:FreqFnc:UniversalLowerBound} in \Cref{T:UniversalLowerBound-FreqFnc}, we are always under the assumption $\bI(T,p) > 1$.
\end{remark}
\begin{proof}
Assume $p=0$. Due to \cite[Lemma 8.16]{ian2024uniqueness}, for any fine blow-up $u$, we have $I_u(0) \geq 1$, thus $\bI(T,0) \geq 1$. 
Notice that $I_0:=\lim_{r\to 0}\bI_p(p,r)$ exists due to \eqref{E:Almost-Monot-I} in \Cref{T:AlmMonotFreq}. We take any sequence $s_k \in (\frac{t_{j(k)+1}}{r_k},\frac{t_{j(k)}}{r_k}]$, if $s_k \rightarrow s$, we must have $\bI_p(p,s_kr_k) \rightarrow I_{u}(s)$ and thus, since $s_kr_k\to 0$, we have $I_{u}(s)=I_0$. This implies that if
\begin{equation*}
s^{+}:=\limsup_{k \rightarrow \infty}\frac{t_{j(k)}}{r_k}, \: \:  s^{-}:=\liminf_{k \rightarrow \infty}\frac{t_{j(k)+1}}{r_k}.
\end{equation*}
then $I_{u}(s)=I_0$ for any $s \in [s^-,s^+]$. By construction of the intervals of flattening, we have that $\frac{s^-}{s^+} \leq \frac{1}{2}.$ By \Cref{L:constantfrequencyannulus}, we must have that $u$ is $I_0$-homogeneous, since its frequency is constant on the radii in $[s^-,s^+]$.
For the moreover part, we fix $0< \beta < \min\{2(I_0 - 1), 2-2{\alpha_\be}\}$. We show that it is enough to prove
\begin{equation}\label{E:improved-decay-I>1}
    \bE(T,\oball{2t_{j+\kappa }}) \leq \left(\frac{t_{j+\kappa}}{t_j}\right)^{\beta} \mbold{0,j}.
\end{equation}

Indeed, if $\mbold{0,j} = \bar{\varepsilon}^2t_j^{2-2{\alpha_\be}}$, since $2-2{\alpha_\be} > \beta$, we obtain
\begin{equation*}
    \mbold{0,j+\kappa} = \max\{\bE(T,\oball{2t_{j+\kappa}}), \bar{\varepsilon}^2t_{j+\kappa}^{2-2{\alpha_\be}}\} \overset{\eqref{E:improved-decay-I>1}}{\leq} \left(\frac{t_{j+\kappa}}{t_j}\right)^{\beta}\bar{\varepsilon}^2t_{j}^{2-2{\alpha_\be}} = \left(\frac{t_{j+\kappa}}{t_j}\right)^{\beta} \mbold{0,j}. 
\end{equation*}

If $\mbold{0,j} = \bE(T,\oball{2t_{j}})$, then $\bE(T,\oball{2t_j})\geq \bar{\varepsilon}^2t_j^{2-2{\alpha_\be}}$ and, since $2-2{\alpha_\be} > \beta$, thus
\begin{equation*}
    \bar{\varepsilon}^2t_{j+\kappa}^{2-2{\alpha_\be}} = \bar{\varepsilon}^2\left(\frac{t_{j+\kappa}}{t_j}\right)^{\beta}\left(\frac{t_{j+\kappa}}{t_j}\right)^{2-2{\alpha_\be}-\beta}t_j^{2-2{\alpha_\be}} \leq \left(\frac{t_{j+\kappa}}{t_j}\right)^{\beta}\bE(T,\oball{2t_j}) = \left(\frac{t_{j+\kappa}}{t_j}\right)^{\beta}\mbold{0,j}.
\end{equation*}

The proof of \eqref{E:improved-decay-I>1} goes along the same lines as the proof of \cite[Equation (88)]{de2023fineI}. 
\end{proof}
We now state the following inequality that will be crucial to prove \Cref{T:UniversalLowerBound-FreqFnc} and is a consequence of a compactness argument demonstrated below.

\begin{lemma}\label{L:LowerBound_DirNj}
Let $T,\Gamma, \Sigma$, and $p$, satisfy \Cref{A:AreaMin+0FlatSing}. For any $\beta >\bI(T,p)$, we have, for $j$ sufficiently large and any $\rho,r\in (t_{j+1}, t_j]$ with $\rho < r$, that
\begin{equation}\label{E:UandL-bound-DirNj}
    \left(\frac{\rho}{r}\right)^{m+2(\beta - 1)}\int_{\cM_{p,j}\cap\oball{r}}|D \cN_{p,j}|^2 \leq \int_{\cM_{p,j}\cap\oball{\rho}}|D \cN_{p,j}|^2.
\end{equation}   
\end{lemma}
\begin{proof} 
Assume $p=0$. It is enough to show that this inequality holds for $r$ small enough with $t_{j+1}<\rho=r/2 < r \leq t_j$. Assume by contradiction that there exists $\beta_0>\bI(T,0)$ such that for every $k\in\N$, there exists $j(k)\geq k$ and $r_{j(k)}\in (t_{{j(k)}+1}, t_{j(k)}]$ such that
\begin{equation*}
    \left(\frac12\right)^{m+2\left(\beta_0 - 1\right)}\int_{\cM_{j(k)}\cap\oball{r_{j(k)}}}|D \cN_{j(k)}|^2 > \int_{\cM_{j(k)}\cap\oball{r_{j(k)}/2}}|D \cN_{j(k)}|^2.
\end{equation*}

Take the fine blowup $u_0$ with respect to $r_{j(k)}$, by \Cref{L:Freq-geq-than-1}, we obtain that $u_0$ is $I_{u_0}(0)$-homogeneous and $I_{u_0}(0) = \bI(T,0)$. By the natural rescaling of the Dirichlet energy, see \cite{DS1,ian2024uniqueness}, and taking the limit in the last displayed inequality, we would force that $m+2(\beta_0-1) \leq m+2(I_{u_0}(0)-1)$ which gives $\beta_0 \leq I_{u_0}(0)=  \bI(T,0) < \beta_0$ which gives the contradiction.
\end{proof}

We are now ready to apply the estimates from the construction of the center manifolds and normal approximations in \Cref{S:1sided-CM-NA}, and combine them with the estimates from the previous lemma to prove the main theorem of this subsection (\Cref{T:UniversalLowerBound-FreqFnc}).

\begin{proof}[Proof of \Cref{T:UniversalLowerBound-FreqFnc}]
Assume $p=0$. The item \eqref{I:HomogFineBU} is settled in \Cref{L:Freq-geq-than-1}. We first prove \eqref{I:FreqFnc:HighFreqIFFfiniteIntervalsFlattening} assuming that the intervals of flattening are finitely many. So, take $J\in\N$ such that $t_{J+1} = 0$. We argue as in \cite[Remark 3.4]{DS5} (using \Cref{T:LocalNA}) to obtain that
\begin{equation*}
    \int_{\oball{\rho}\cap\cM_J}|D\cN_J|^2 \leq C \rho^{m+2-2{\alpha_\be}}\mbold{0,J}, \mbox{ for every }\rho\in \left(0,t_J\right].
\end{equation*}

By \Cref{L:LowerBound_DirNj}, we get $2 - 2{\alpha_\be} < 2(\beta -1)$, for any $\beta > \bI(T,0)$. This then gives $\bI(T,0) \geq 2-{\alpha_\be}$. We prove the opposite implication to conclude the proof of \eqref{I:FreqFnc:HighFreqIFFfiniteIntervalsFlattening}. We assume $\bI(T,0)> 2-{\alpha_\be}$ and that the intervals of flattening are infinitely many. By \eqref{E:SplittingNA1}, we obtain
\begin{equation}\label{E:LowerBound-DirNj-HighFrequency}
    \int_{\oball{t_{j+1}}\cap\cM_j}|D \cN_j|^2 \geq C \mbold{0,j} \left ( \frac{t_{j+1}}{t_j}\right) ^{m+2-2{\alpha_\be}}.
\end{equation}

By \Cref{L:Freq-geq-than-1}, we have $\bE(T,\oball{r})\leq Cr^{\beta}$ and, since $\bI(T,0) > 2-{\alpha_\be}$, we can choose $\beta > 2-2{\alpha_\be}$. This in turn guarantees that, for a sufficiently large $j$, $\mbold{0,j} = \bar{\varepsilon}^2t_j^{2-2{\alpha_\be}}$. Then, we derive
\begin{equation*}
    \bE(T,\oball{r}) \leq C r^{\beta- (2-2{\alpha_\be})}\frac{r^{2-{\alpha_\be}}}{(\bar{\varepsilon}^2t_j)^{2-2{\alpha_\be}}}\mbold{0,j}  \leq C r^{\beta- (2-2{\alpha_\be})}\left(\frac{t_{j+1}}{t_j}\right)^{2-2{\alpha_\be}}\mbold{0,j}, \forall r\in (t_{j+1},t_j].
\end{equation*}

Additionally, the last displayed inequality implies
\begin{equation*}
\begin{aligned}
    \int_{\cM_j\cap\oball{t_{j+1}}}|D\cN_j|^2 \overset{\eqref{E:SplittingNA2}}{\leq} \tilde{C}\left(\frac{t_{j+1}}{t_j}\right)^{m} \bE(T_j, \oball{L}) \leq C t_{j+1}^{\beta- (2-2{\alpha_\be})}\left(\frac{t_{j+1}}{t_j}\right)^{m+2-2{\alpha_\be}}\mbold{0,j} .
\end{aligned}
\end{equation*}

By the last displayed inequality and \eqref{E:LowerBound-DirNj-HighFrequency}, we get that
\begin{equation*}
    C \mbold{0,j} \left ( \frac{t_{j+1}}{t_j}\right) ^{m+2-2{\alpha_\be}} \leq \int_{\oball{t_{j+1}}\cap\cM_j}|D \cN_j|^2 \leq C r^{\beta- (2-2{\alpha_\be})}\left(\frac{t_{j+1}}{t_j}\right)^{m+2-2{\alpha_\be}}\mbold{0,j},
\end{equation*}
for any $r\in (t_{j+1},t_j]$. Due to the choice of $\beta$, which ensures $\beta - (2-2{\alpha_\be})>0$, we obtain a contradiction from the last displayed inequality for $j$ sufficiently large. Thus, $\{t_j\}_j$ has to be finite, and this concludes the proof of \eqref{I:FreqFnc:HighFreqIFFfiniteIntervalsFlattening}.

Let us prove \eqref{I:FreqFnc:UniversalLowerBound}, note that if we have finitely many intervals of flattening, \eqref{I:FreqFnc:HighFreqIFFfiniteIntervalsFlattening} implies \eqref{I:FreqFnc:UniversalLowerBound}. Assume then that $t_j \neq 0$ for all $j$. We apply \eqref{E:SplittingNA1} and \cite{ian2024uniqueness} to get
\begin{equation*}
    \int_{\cM_j\cap\oball{t_{j+1}}} |D\cN_j|^2 \leq \int_{\cM_j\cap\oball{\frac{t_{j+1}}{t_j}}} |D\cN|^2 \leq C\left(\frac{t_{j+1}}{t_j}\right)^m\bE(T,\oball{L}^{\flat}) \leq \tilde{C}\mbold{0,j}\left(\frac{t_{j+1}}{t_j}\right)^{m+\alpha}.
\end{equation*}

We apply \eqref{E:SplittingNA2} and \Cref{L:LowerBound_DirNj} to obtain, for any $\rho\in (t_{j+1}, t_j]$ and any $\beta > \bI(T,0)$, that
\begin{equation}\label{E:lowerbound-beta-dirNj}
    C\mbold{0,j} \leq \int_{\cM_j\cap\oball{t_j}} |D\cN|^2 \leq \left(\frac{\rho}{t_j}\right)^{-m-2(\beta-1)}\int_{\cM_j\cap\oball{\rho}} |D\cN|^2.
\end{equation}

Combining the last two displayed inequalities, we arrive at
\begin{equation*}
    C\mbold{0,j}\left(\frac{\rho}{t_j}\right)^{m+2(\beta-1)} \leq \tilde{C}\mbold{0,j}\left(\frac{t_{j+1}}{t_j}\right)^{m+\alpha}, \mbox{ for any }\rho \in (t_{j+1},t_j]\mbox{ and any }\beta > \bI(T,0), 
\end{equation*}
which, for $j$ large enough, guarantees that $\beta \geq 1+ \alpha/2$ for any $\beta>\bI(T,0)$. Thus, $\bI(T,0) \geq 1 +\alpha/2$. We now need to show $\inf_j\frac{t_{j+1}}{t_j}>0$ assuming that $\{t_j\}_j$ is infinite. In light of \eqref{I:FreqFnc:HighFreqIFFfiniteIntervalsFlattening} and the first part of \eqref{I:FreqFnc:UniversalLowerBound}, we have $1+\alpha/2 \leq \bI(T,0) < 2-{\alpha_\be}$. Let us assume by contradiction that $\inf_j\frac{t_{j+1}}{t_j}=0$. Arguing as in \eqref{E:lowerbound-beta-dirNj} and \cite[Remark 3.4]{DS5} (same argument performed above), we obtain that
\begin{equation*}
    C\mbold{0,j}\left(\frac{t_{j+1}}{t_j}\right)^{m+2(\beta-1)} \leq \int_{\oball{t_{j+1}}\cap\cM_j}|D\cN_j|^2 \leq \tilde{C}\mbold{0,j}\left(\frac{t_{j+1}}{t_j}\right)^{m+2-2{\alpha_\be}},
\end{equation*}
for any $\beta>\bI(T,0)$, hence $\bI(T,0) \geq 2-\alpha_\be$. This is a contradiction for $j$ sufficiently large, since $\inf_j\frac{t_{j+1}}{t_j}=0$ and $\bI(T,0) < 2 -\alpha_\be$. Thus, we conclude the proof of \eqref{I:FreqFnc:UniversalLowerBound}. 
\end{proof}

\subsection{High-frequency regime}

This subsection addresses the high-frequency regime, where the frequency is at least almost quadratic and, by \Cref{L:Freq-geq-than-1}, an almost-quadratic excess decay holds. In this regime, we define the frequency pinching and prove the main lemma of this subsection (\Cref{L:BigPinching:HighFreq}).

To this end, we introduce a new exponent, denoted by $\delta_\be$ (which corresponds to $\delta_3$ in the notation of \cite{de2023fineII}), and require that $2-2\delta_\be < 2-2\alpha_\be$. This exponent $\delta_\be$ will replace $\alpha_\be$ in the ``stopping by the excess condition" (see \Cref{S:Stopping-condition}), while $\alpha_\bh$ in the ``stopping by the height condition" remains unchanged. 

The exponent $\delta_\be$ will then be used to construct the center manifold and the normal approximation of \Cref{S:1sided-CM-NA}, namely $\cM_{p,J}^{\delta_\be}$ and $\cN_{p,J}^{\delta_\be}$. For convenience, we will drop the dependence on $\delta_\be$ in the notation. 

\begin{assumption}\label{A:High-Frequency}
Let $T,\Gamma, \Sigma$, and $p$, satisfy \Cref{A:AreaMin+0FlatSing}, $\bI(T,p) > 2-\delta_\be$, and let $J+1$ be the index of \Cref{T:UniversalLowerBound-FreqFnc} \eqref{I:FreqFnc:HighFreqIFFfiniteIntervalsFlattening} for which $t_{J+1}=0$. Assume also that $q\in\Gamma_{p,J}$.    
\end{assumption}

In the high-frequency regime, given the quasi-quadratic excess decay in \Cref{L:Freq-geq-than-1}, there are no boundary-stopping cubes. In other words, the contact set contains the boundary, as stated below. 

\begin{lemma}[Boundary cubes never stop in the high-frequency regime]\label{L:BdrCubesDontStop:HighFreq}
 $\mathscr{C}^\flat_j \cap \sW_i = \varnothing$ for every $i, j\geq N_0$ and, in particular, $\Gamma_{p,J} \subset \bPhi_{p,J}(\bS_{p,J})$. 
\end{lemma}
\begin{proof}
This proof is entirely analogous to \cite[Lemma 6.5]{nardulli2022density} using $\bI(T,p) > 2-\delta_\be$ and the almost quadratic excess decay provided by \Cref{L:Freq-geq-than-1}.
\end{proof}

\begin{lemma}[Radial variations]\label{L:FirstEqsFreqFnc-HighFreq}
Under \Cref{A:High-Frequency}, all the consequences stated in \Cref{L:FirstEqsFreqFnc} are valid for $p$, $q$, and every radius $r\in (0,t_J)$.
\end{lemma}
\begin{proof}
The proof will be omitted here, since it consists of a mere repetition of the arguments in \Cref{L:FirstEqsFreqFnc}, \Cref{T:AlmMonotFreq}, and \cite{delellis2021allardtype}, taking into consideration that they only depend on the following property: every cube $L\in\sW$ such that $L\cap \bball{\bp(q)}{r}\neq\varnothing$ cannot have side length larger than $c_0 r$ for a dimensional constant $c_0 >0$. This property in turn follows directly from \Cref{L:BdrCubesDontStop:HighFreq} and \eqref{E:separation-Whitney}.
\end{proof}

We define frequency pinching, which estimates the error in the frequency when measured at two different scales at the same boundary point $q$.

\begin{definition}[Frequency pinching]
Under \Cref{A:High-Frequency}. Let $\rho, r$ be two radii which satisfy the inequalities $0 < \rho \leqslant r\leq t_J$.
We then define the \emph{frequency pinching $W_\rho^r(p,q)$ around $q$ between the scales $\rho$ and $r$} to be the following quantity
\begin{equation*}
    W_\rho^r(p, q):=\left|\bI_{p}\left(q, r\right)-\bI_{p}\left(q, \rho\right)\right|, \mbox{ for any }q\in\Gamma_{p,J}.
\end{equation*}
\end{definition}

With the very same argument of \cite[Lemma 4.1]{de2023fineII}, we can ensure the existence of $\Lambda$ depending on $T$ such that
\begin{equation}\label{E:UpperBound-Freq}
    C^{-1} \leq \bI_p(q,r) \leq \Lambda, \quad \forall q\in \Gamma_{p,J}, \quad \forall r \in (0, t_J]. 
\end{equation}

We will now prove a lemma that will be useful in establishing the dimensional bound for the set of one-sided flat singular boundary points in \Cref{S:DimBound} and also for proving the rectifiability results in \Cref{S:Rectfiability}. To do this, we define the following.

\begin{definition}
We say that a set $S=\left\{x_0, x_1, \ldots, x_k\right\} \subset \ball{x}{r}$ is \emph{$\rho r$-linearly independent} if
\begin{equation*}
    d\left(x_i, \mathrm{span}\left(\left\{x_0, \ldots, x_{i-1}\right\}\right)\right) \geqslant \rho r, \quad \text { for all } i=1, \ldots, k.
\end{equation*}

We say that a set $F \subset \ball{x}{r}$ \emph{$\rho r$-spans an $k$-dimensional affine subspace $V$} if there is a $\rho r$-linearly independent set of points $S=\left\{x_i\right\}_{i=0}^k \subset F$ such that $V=\mathrm{span}(S)$.
\end{definition}

We now demonstrate that it is not possible to have $m-1$ quantitatively linearly independent points that exhibit small pinching simultaneously.

\begin{lemma}\label{L:BigPinching:HighFreq}
Assume \Cref{A:High-Frequency} and $\rho \in (0,1]$. There exists a suitable threshold $\delta_1 = \delta_1(m, n, Q, \rho)>0$, with the following property. Take $r\in (0,\bar{r})$ (where $\bar{r}$ is from \Cref{R:1sided-flat-assumption}) and an $\rho r$-linearly independent set 
\begin{equation*}
    S = \left\{x_i\right\}_{i=0}^{m-2} \subset \ball{p}{r}\cap\Singb(T)\cap\left\{q: \Theta^m(T,q) = Q/2\mbox{ and }q\mbox{ flat}\right\}.
\end{equation*}

Then, there exist $i\in\{0,\ldots, m-2\}$ and $r_i\in (0,r)$ such that, for any $y\in \ball{x_i}{r_i}\cap\Singb(T)\cap\left\{q: \Theta^m(T,q) = Q/2\mbox{ and }q\mbox{ flat}\right\}$, we have
\begin{equation*}
    W_{r_i}^{2 r_i}\left(p, y\right) \geq \delta_1.
\end{equation*}
\end{lemma}
\begin{proof}
The proof of the lemma follows the same lines as the contradiction argument in \cite[Lemma 6.2]{de2023fineII}. The main difference being that the authors contradict the \emph{interior} regularity for the linear problem (in \cite{DS1,de2018rectifiability}) with their argument, we instead run the very same contradiction argument in order to get a fine blow-up $u$ that is $I_0$-homogeneous with respect to every point in $S$ and $I_0\geq 1+\alpha/2$. Then \Cref{l:twovariables} gives us a function $f$ which is $I_0$-homogeneous $2$d Dirichlet minimizer with zero boundary data on a flat boundary and $I_0\geq 1 + \alpha/2$. This contradicts \Cref{t:homogeneous2d}, which says that the frequency of $f$ should be exactly $1$.
\end{proof}

\subsection{Low-frequency regime}

We now address the case of low-frequency, where we do not have an almost-quadratic decay. The proofs in the low-frequency regime follow a similar approach to those in the high-frequency regime from the previous section. However, one of the main challenges in this regime is that the intervals of flattening are infinitely many (\Cref{T:UniversalLowerBound-FreqFnc}\eqref{I:FreqFnc:HighFreqIFFfiniteIntervalsFlattening}). Although this introduces several technical difficulties, which will be explained and addressed in the following, it does not fundamentally alter the conceptual approach for proving rectifiability.

\subsubsection{Setting up uniform bounds in a suitable partition} 

We define the following sets for $K\in\N\setminus \{0\}$:
\begin{equation*}
    \mathfrak{G}_K := \Singb(T) \cap \left\{q: \Theta^m(T,q)=Q/2, q\mbox{ is flat, and }\bI_q(T,q) \leq 2 - \alpha_\be  - \frac{1}{2^{K}}\right \}\cap\ball{p}{2}.
\end{equation*}

Since the main goal of this work is to prove the rectifiability of the set of one-sided flat singular boundary points (\Cref{P:Rectf-LowFreq}), which is equivalent to proving the rectifiability of $\cup_{1\leq K\in\N}\mathfrak{G}_K$, we will begin by studying each $\mathfrak{G}_K$ individually.

\begin{assumption}\label{A:Low-Frequency}
Let $T,\Gamma$, and $\Sigma$ satisfy \Cref{A:general}, $p\in\mathfrak{G}_K$, and $\eta$ be a given small number to be specified later.
\end{assumption}

We take $\gamma \in (0,1/2]$ to be a geometric constant, depending on $m,n,Q$, as in \cite[Section 10.1]{de2023fineII}, and denote by $\cM_{p,\gamma^j}$ and $\cN_{p,\gamma^j}$ the center manifold and normal approximation, respectively, with respect to $T_{p,\gamma^j}, \Gamma_{p,\gamma^j}$, and the intervals $(\gamma^{j+1},\gamma^j]$ (which replace the intervals of flattening; see also \cite[Section 10.1]{de2023fineII}). 

This constant $\gamma$ ensures the uniformity of certain estimates used below. For example, an important feature now is that $\inf_j \frac{\gamma^{j+1}}{\gamma^j} = \gamma > 0$, as opposed to \Cref{T:UniversalLowerBound-FreqFnc}\eqref{I:FreqFnc:UniversalLowerBound}, which only guarantees that $\inf_j \frac{s_{j+1}}{s_j} > 0$ with no uniform lower bound.

\begin{lemma}[Radial variations]\label{L:RadialVariations:LowFreq}
Under \Cref{A:Low-Frequency}. All assertions stated in \Cref{L:FirstEqsFreqFnc} are valid for $p$, $q$, and 
\begin{equation*}
    4\eta \gamma^{j+1} < \rho \leq r \leq 2\gamma^j \mbox{ and } q\in \Gamma_{p,\gamma^j}.
\end{equation*}
\end{lemma}
\begin{proof}
The proof will be omitted here, since it consists of a mere repetition of the arguments in \Cref{L:FirstEqsFreqFnc}, \Cref{T:AlmMonotFreq}, and \cite{delellis2021allardtype}, taking into consideration that they only depend on the following property: every cube $L\in\sW$ such that $L\cap \bball{\bp(q)}{r}\neq\varnothing$ cannot have side length larger than $c_0 r$ for a dimensional constant $c_0 >0$. This property in turn follows directly from the assumption that $r > 4\eta \gamma^{j+1}$ and \eqref{E:separation-Whitney}.
\end{proof}

We state here the counterpart, for the low-frequency regime, of \Cref{L:BigPinching:HighFreq}. In order to do this, we need to define the frequency pinching in this regime.

\begin{definition}[Frequency pinching]
Under \Cref{A:Low-Frequency}. Let $\rho, r$ be two radii which satisfy the inequalities $4\eta\gamma^{j+1} < \rho \leqslant r\leq \gamma^j$.
We then define the \emph{frequency pinching $W_\rho^r(p,j,q)$ around $q$ between the scales $\rho$ and $r$} to be the following quantity
\begin{equation*}
    W_\rho^r(p, j, q):=\left|\bI_{p}\left(q, r\right)-\bI_{p}\left(q, \rho\right)\right|, \mbox{ for any }q\in\Gamma_{p,\gamma^j}.
\end{equation*}
\end{definition}

The proof of the following lemma is entirely analogous to \Cref{L:BigPinching:HighFreq}, this time following the contradiction argument in \cite[Lemma 12.2]{de2023fineII}.

\begin{lemma}\label{L:BigPinching:LowFreq}
Assume \Cref{A:High-Frequency} and $\rho \in (0,1]$. There exists a suitable threshold $\delta_2 = \delta_2(m, n, Q, \gamma, K, \rho)>0$, with the following property. Take $r\in (\gamma^{j+1},\gamma^j)$ and an $\rho r$-linearly independent set $S = \left\{x_i\right\}_{i=0}^{m-2} \subset \ball{p}{r}\cap\mathfrak{G}_K$. Then, there exist $i\in\{0,\ldots, m-2\}$ and $r_i\in (0,r)$ such that, for any $y\in \ball{x_i}{r_i}\cap\mathfrak{G}_K$, we have
\begin{equation*}
    W_{r_i}^{2 r_i}\left(p, y\right) \geq \delta_1.
\end{equation*}
\end{lemma}

\section{Dimensional bound and countability of the singular set in 3d}\label{S:DimBound}

In this section, we show how to obtain a Hausdorff dimensional bound for the set of one-sided singular boundary points, i.e., with density $Q/2$. We will split the proof into two cases: $m>3$ (\Cref{T:DimBound}) and $m=3$ (\Cref{T:Countability}). Additionally, we note that when proving the rectifiability of the set of one-sided singular points for $m>3$ in \Cref{S:Rectfiability}, the Hausdorff dimensional bound follows as a by-product of the finer analysis conducted therein, which is entirely independent of this section.

We first recall a standard lemma in geometric measure theory regarding the weak $\epsilon$-approximation property, for which we provide a proof for the sake of completeness.

\begin{lemma}[Weak $\epsilon$-approximation property]\label{L:weak-epsilon-approx}
Let $0\in S \subset \R^{m+n}, k \in \{1,\ldots, m+n-1\}$, and $\epsilon \in\left(0, 1/4 \right)$. Assume that $S$ satisfies the weak $\epsilon$-approximation property in $\oball{1}$, meaning that for each $r \in(0,1]$ and each $x \in S \cap \oball{1}$, there exists a $k$-plane $L=L(r,x)$ containing $x$ and satisfying
\begin{equation}\label{E:weak-epsilon-approx}
S \cap \ball{x}{r} \subset(\epsilon r)\mbox{-neighborhood of } L \cap \ball{x}{r}.
\end{equation}

Then, there exists a function $g : (0, 1) \rightarrow(0, +\infty)$ with $\lim _{s \rightarrow 0} g(s)=0$ such that $\cH^{k+g(\epsilon)}\left(S \cap \oball{1}\right)=0$.
\end{lemma}

\begin{proof}
Given a ball $\ball{x}{r}$, by \eqref{E:weak-epsilon-approx}, we can cover the set $S \cap \ball{x}{r}$ with at most $C(k)\epsilon^{-k}$ balls of radius $\epsilon r$. We fix the cover $\mathfrak{C}_0= { \oball{1}}$. The cover $\mathfrak{C}_{j+1}$ is then obtained by replacing each ball $B$ in $\mathfrak{C}_j$ with $C(k)\epsilon^{-k}$ balls of radius $\epsilon r(B)$.

For any $\alpha>0$, we observe that
\begin{equation*}
\cH^{\alpha}_{\infty}(S \cap \oball{1}) \leq \sum_{B \in \mathfrak{C}{j+1}} \omega_{\alpha} r(B)^{\alpha} \leq \sum_{B \in \mathfrak{C}{j}} \omega_{\alpha}C(k)\epsilon^{-k}(\epsilon r(B))^{\alpha}= C(k)\epsilon^{\alpha-k}\sum_{B \in \mathfrak{C}{j}} \omega_{\alpha}r(B)^{\alpha}.
\end{equation*}

By iterating on $j$, we obtain $\cH^{\alpha}(S \cap \oball{1}) \leq (C(k)\epsilon^{\alpha-k})^{j} \omega_{\alpha}$. If $C(k)\epsilon^{\alpha-k}<1$, it follows that $\cH^{\alpha}(S \cap \oball{1})=0$.

We define $g$ for any $s>0$ by setting $C(k)s^{g(s)}=2^{-1}$. It is straightforward to verify that $g$ is continuous, satisfies $\lim_{s\rightarrow 0}g(s)=0$, and ensures $\cH^{k+g(\epsilon)}\left(S \cap \oball{1}\right)=0$.
\end{proof}

We now assume that $m>3$ to prove the following dimensional bound.

\begin{theorem}\label{T:DimBound}
Let $T,\Sigma$, and $\Gamma$, be as in \Cref{A:general}. Suppose $m>3$, then
\begin{equation*}
    \mathrm{dim}_{\cH}\left( \Singb^1(T) \right) \leq m-3.
\end{equation*}
\end{theorem}
\begin{proof}
In the proof of \Cref{T:Rectifiability-1sided} (see \Cref{S:main-results-proofs}), we constructed a decomposition of the singular set into countable pieces, each consisting of flat boundary singularities for $T$. By applying the same induction argument as in that proof, we readily reduce the problem to proving \Cref{T:DimBound} for flat boundary singularities.

For the remainder of the proof, we consider the set $S$ (which will be a Borel set) to consist of either the one-sided flat singularities of high frequency or the one-sided flat singularities of low frequency, i.e.,
\begin{equation*}
    S:= \mathfrak{G}_K\mbox{ or } S:=\left(\Singb^1(T) \cap \left\{q\in\Gamma: q\mbox{ is flat}\right \}\right)\setminus\mathfrak{G}_K.
\end{equation*} 

Assume, for the sake of contradiction, that $\dim_{\cH}(S) =: M > m-3$. Let $\epsilon_0>0$ be sufficiently small so that if a set satisfies the weak $\epsilon_0$-approximation property, then $\cH^{M}(S)=0$. Removing a set of $\cH^{M}$-measure zero if necessary, we may assume that, for all $q \in S$, $\Theta^{M}(S,q)>0$. Consequently, for every $q\in S$ and $r>0$, the set $S$ fails to satisfy the weak $\epsilon_0$-approximation property in $\ball{q}{r}$. This, in turn, implies the existence of a point $x_0 \in S\cap \ball{q}{r}$ and a radius $r_0>0$ such that \emph{no} $(m-3)$-plane $L$ satisfies \eqref{E:weak-epsilon-approx}, i.e.,
\begin{equation*}
S\cap\ball{x_0}{r_0} \not\subset (\epsilon_0 r_0)\mbox{-neighborhood of } L, \mbox{ for any }(m-3)-\mbox{plane }L.
\end{equation*}

We now aim to construct a set of $m-1$ points in $S$ that are $(r_0\epsilon_0)$-linearly independent. Beginning with $x_0$, we proceed inductively: suppose we have already found an $(r_0\epsilon_0)$-linearly independent set $\{x_0,\dots,x_k\} \subset S$. Then, we define $x_{k+1} \in S$ as a point satisfying
\begin{equation*}
\dist\left(x_{k+1}, \textup{span}(x_0,x_1,..,x_{k}) \right)>\epsilon_0 r_0.
\end{equation*}

If this process terminates before reaching $k+1 = m-2$, i.e., we fail to find such an $x_{j+1} \in S\cap\ball{x_0}{r_0}$ for some $j<m-2$, then every $x\in S\cap\ball{x_0}{r_0}$ must satisfy
\begin{equation*}
\dist\left(x, \textup{span}(x_0,x_1,..,x_{j}) \right)\leq \epsilon_0 r_0.
\end{equation*}

Since $L_j := \textup{span}(x_0,x_1,..,x_{j})$ has dimension $j < m-2$, this contradicts our earlier assumption about $x_0$ and $r_0$, as $S\cap \ball{x_0}{r_0}$ would then satisfy \eqref{E:weak-epsilon-approx}.

Thus, we obtain an $(r_0\epsilon_0)$-linearly independent set $\{x_0,\dots, x_{m-2}\} \subset S\cap \ball{x_0}{r_0}$. To derive a contradiction, we proceed as follows. Since $\Theta^M(S,x_i)>0$ for each index $i$ given by \Cref{L:BigPinching:HighFreq} (for $S=\mathfrak{G}_K$) or \Cref{L:BigPinching:LowFreq} (otherwise), we can repeat the above process in $\ball{x_i}{r_i}$ to find another set $\{x_0^{(1)},\dots, x_{m-2}^{(1)}\} \subset S\cap\ball{x_i}{r_i}$ that is $(r_i\epsilon_0)$-linearly independent. Applying \Cref{L:BigPinching:HighFreq} (if $S=\mathfrak{G}_K$) or \Cref{L:BigPinching:LowFreq} (otherwise) with $\rho = \epsilon_0$, we obtain an index $i_1 \in {0,\dots, m-2}$ and a radius $r_{i_1} > 0$ such that
\begin{equation}\label{E:Proof:BigPinching1}
W_{r_{i_1}/2}^{2r_{i_1}}(p,y) \geq \delta(\epsilon_0) >0, \quad \forall y\in S\cap\ball{x^{(1)}_{i_1}}{r_{i_1}}.
\end{equation}

Since the frequency function is uniformly bounded above (see \eqref{E:UpperBound-Freq}) and non-negative, the iterative process must terminate after finitely many steps, contradicting the assumption that $\Theta^M(S,x) > 0$ for all $x \in S$ up to a $\cH^M$-null set. This contradiction implies that $M \leq m-3$, completing the proof.
\end{proof}

In the case $m=3$, we establish the countability of the one-sided singular set. The argument used in \Cref{T:DimBound} must be replaced with a different approach, as there exist uncountable sets with Hausdorff dimension equal to zero.

\begin{theorem}\label{T:Countability}
Let $T,\Sigma$, and $\Gamma$, be as in \Cref{A:general}. Suppose $m=3$, then
\begin{equation*}
    \Singb^1(T) \mbox{ is a countable set}.
\end{equation*}\end{theorem}
\begin{proof}
In the proof of \Cref{T:Rectifiability-1sided} (see \Cref{S:main-results-proofs}), we provide a decomposition of the boundary singular set into countable pieces, each consisting of flat boundary singularities for $T$. By applying the same induction argument used therein, we can reduce the proof of \Cref{T:Countability} to establishing countability for flat boundary singularities. Specifically, by demonstrating that each piece in the decomposition is countable, we conclude \Cref{T:Countability}.

For the remainder of the proof, we consider the set $S$ to be either the one-sided flat singularities of high frequency or those of low frequency, i.e.,
\begin{equation*}
    S:= \mathfrak{G}_K\mbox{ or } S:=\left(\Singb^1(T) \cap \left\{q\in\Gamma: q\mbox{ is flat}\right \}\right)\setminus\mathfrak{G}_K.
\end{equation*} 

If the set of one-sided flat singularities were uncountable, then either the high-frequency or the low-frequency points would form an uncountable set. Without loss of generality, we may assume that every point in $S$ is an accumulation point of $S$. Indeed, the set of isolated points in $S$ is at most countable, and removing them does not affect our argument.

Next, take two points $x_0, x_1\in S$ and set $|x_0-x_1|=: \epsilon_0 >0$. We can now apply either \Cref{L:BigPinching:HighFreq} ($S=\mathfrak{G}_K$) or \Cref{L:BigPinching:LowFreq} (otherwise) with $\rho = \epsilon_0$, to obtain that, without loss of generality,
\begin{equation*}
    W_{r_0}^{2r_0}(p,y) \geq \delta_1(\epsilon_0) > 0, \ \forall y\in S\cap\ball{x_0}{r_0},\mbox{ for some }r_0 \in (0, \epsilon_0/4).
\end{equation*}

Since $x_0$ is an accumulation point of $S$, we can now reapply this procedure to obtain two points $x_0^{(1)},x_1^{(1)}\in S\cap \ball{x_0}{r_0}$ with $|x_0^{(1)} - x_1^{(1)}| =: \epsilon_0 \rho_1 >0 $. Again, without loss of generality, we have
\begin{equation}\label{E:Proof:BigPinching2}
    W_{r_1}^{2r_1}(p,y) \geq \delta(\epsilon_0) > 0, \ \forall y\in S\cap\ball{x_0}{r_1},\mbox{ for some }r_1 \in (0, \rho_1/4).
\end{equation}

Since the frequency function is uniformly bounded from above (see \eqref{E:UpperBound-Freq}) and is non-negative, this process can only be iterated finitely many times. However, this contradicts \eqref{E:Proof:BigPinching2}, as the assumption that all points in $S$ are accumulation points would allow us to repeat the procedure indefinitely. Hence, $S$ must be countable.
\end{proof}

\section{Rectifiability for the nonlinear problem}\label{S:Rectfiability}

We are now aiming at proving the main rectifiability results of this paper. We mention that the results in this section are independent of the dimensional bound in \Cref{S:DimBound} when $m>3$. In the case that $m=3$, the problem is already settled in \Cref{T:Countability}.

\subsection{Spatial variations}

In previous sections, we computed several different quantities related to $\bI$ in which estimates for their radial derivatives were achieved, see \Cref{L:FirstEqsFreqFnc}. These estimates give us the almost monotonicity of the frequency. However, in order to compare the frequency amongst different points, we will need a better understanding of spatial derivatives.

\subsubsection{Spatial variations in the high-frequency regime}

The relevant quantities and their properties are stated in the following lemma. 

\begin{lemma}[Spatial variations]\label{L:spatial-variations-HighFreq}
Under \Cref{A:High-Frequency}. Then either $T_{p,J} = Q\a{\cM_{p,J}}$ in a neighborhood of $q$, or 
\begin{equation*}
    \bH_p(q,r) >0 \mbox{ and }\bD_p(q,r)>0,\mbox{ for every }r\in (0,t_J].
\end{equation*} 

In the latter case, there exists a constant $\gamma_1 = \gamma_1(m,\delta_\be,\alpha_\bh,\kappa,\tau)>0$ with the following property. Let $v$ be a vector field on $\cM_{p,J}$ that is tangent to $\Gamma_{p,J}$. We obtain
\begin{equation*}
\begin{aligned}
    \partial_{v} \bD_{p}(q, r)&=\frac{2t_J}{r} \int_{\cM_{p,J}} \phi^{\prime}\left(\frac{d(q,z)}{r}\right) \sum_i \partial_{\nabla_z d(q,z)}\left(\cN_{q,J}\right)_i(z) \cdot \partial_{v}\left(\cN_{q,J}\right)_i(z) \mathrm{d}z \\
    &\quad +O\left(\mbold{p,J}^{\gamma_1}\right)r^{\gamma_1}t_J^{\gamma_1-1} \bD_{p}(q, r), \\
    \partial_{v} \bH_{p}(q, r) &= -2 \sum_i \int_{\cM_{p, J}} \frac{|\nabla d(q,z)|^2}{d(q, z)} \phi^{\prime}\left(\frac{d(q,z)}{r}\right)\left\langle\partial_{v}\left(\cN_{p,J}\right)_i(z),\left(\cN_{p,J}\right)_i(z)\right\rangle \mathrm{d}z,
\end{aligned}   
\end{equation*}
for any $r\in (0,t_J)$.
\end{lemma}
\begin{proof}
The proof is analogous to \cite[Lemma 5.5]{de2023fineII} using that $v$ is tangent to the boundary and the machinery from \cite{delellis2021allardtype} which is in turn based upon \cite{DDHM}. We briefly sketch it here to show the main differences. To prove the first claimed equality, we take $r\in (0, 2)$ and define $\bar{v}(z):= t_J^{-1}v(t_J z)$ on $\ocM_{p,J}$ which is tangent to $\Gamma_{p,J}$ and 
\begin{equation*}
    Y(z):= \phi\left(\frac{d_J(q,z)}{r}\right)\bar{v}(z),\mbox{ for }z\in\ocM_{p,J}, \mbox{ where }v\mbox{ is tangent to }\Gamma_{p,J}.
\end{equation*}

We then have that 
\begin{equation*}
\begin{aligned}
    \mathrm{div}_{\cM_{p,J}}&Y(z) = \frac{1}{r}\phi^\prime\left( \frac{d_J(q,z)}{r} \right)\left \langle \nabla d_J(q,z), \bar{v}(z)\right\rangle - \langle Y(z), H_{\cM_{p,J}}\rangle , \\
    D_{\cM_{p,J}}Y(z) &= \frac{1}{r}  \phi^\prime\left( \frac{d_J(q,z)}{r} \right) \bar{v}(z)\otimes \nabla d_J(q,z) - \phi\left( \frac{d_J(q,z)}{r} \right) \sum_k A_{\cM_{p,J}}(e_k, \bar{v}(z))\otimes e_k,
\end{aligned}
\end{equation*}
where $\{e_k\}_k$ is an orthonormal frame of $\ocM_{p,J}$ with $e_1 = \bar{v}$. By a direct computation, we also obtain that
\begin{equation*}
    \partial_{\bar{v}} \obD_p(q,r) = \frac{1}{r}\int_{\ocM_{p,J}} \phi^\prime\left( \frac{d_J(q,z)}{r} \right)\left \langle \nabla d_J(q,z), \bar{v}(z)\right\rangle|D\ocN_{p,J}(z)|^2\mathrm{d}z.
\end{equation*}

Arguing as in \cite[Lemma 5.5]{de2023fineII}, we plug the vector field $Y$ into the inner variation formulas in \cite[Proposition 16.4]{delellis2021allardtype} and using their notation for $\mathrm{Err}^i_{k}, k\in\{1,\ldots, 4\}$, we obtain that
\begin{equation*}
\begin{aligned}
    \partial_{\bar{v}}\obD_p(q,r) &- \frac{2}{r}\int_{\ocM_{p,J}} \phi^\prime\left( \frac{d_J(q,z)}{r} \right)\sum_i\left \langle \partial_{\nabla d_J(q,z)}(\ocN_{p,J})_i(z), \partial_{\bar{v}(z)}(\ocN_{p,J})_i(z)\right\rangle \mathrm{d}z \\
    &=O(\mbold{p,J}^{1/2})r^{\frac{1-\delta_\be}{2}}\obD_p(q,r) + \sum_{k=1}^4\overline{\mathrm{Err}^i_{k}}.
\end{aligned}
\end{equation*}

Now, as in \cite[Proposition 16.7]{delellis2021allardtype}, we control the error terms as follows:
\begin{equation*}
\begin{aligned}
    \sum_{k=1}^4|\overline{\mathrm{Err}^i_{k}}| \leq C\mbold{p,J}^\tau \obD_{p,J}(q,r)^\tau \left( \obD_{p,J}(q,r) + r \partial_r \obD_{p,J}(q,r)\right) + C\mbold{p,J}t_J^{2\kappa}r\obD_{p,J}(q,r).    
\end{aligned}
\end{equation*}

We now use the rescaling \eqref{E:rescalingD} to obtain 
\begin{equation*}\begin{aligned}
    \sum_{k=1}^4|\overline{\mathrm{Err}^i_{k}}| &\leq C\mbold{p,J}^\tau t_J^{-m(1+\tau)} \bD_{p}(t_J q,t_J r)^\tau \left( \bD_{p}(t_J q,t_J r) + r t_J \partial_r \bD_{p}(t_J q,t_J r)\right) \\
    &\quad + C\mbold{p,J}t_J^{-m+2\kappa}r\bD_{p}(t_J q,t_J r).   
\end{aligned}\end{equation*}

The last displayed inequalities together with $\partial_{\bar{v}(z)}\obD_p(q,r) = t_J^{-m}\partial_{v(t_J z)}\bD_p(t_J q, t_J r)$ gives 
our claimed inequality. Indeed:
\begin{equation*}\begin{aligned}
    \partial_{v(z^\prime)}&\bD_p(q^\prime,s) - \frac{2t_J^{m+1}}{s}\int_{\ocM_{p,J}} \phi^\prime\left( \frac{d_J(q,z)}{s} \right)\sum_i\left \langle \partial_{\nabla d_J(q,z)}(\ocN_{p,J})_i(z), \partial_{\bar{v}(z)}(\ocN_{p,J})_i(z)\right\rangle \mathrm{d}z \\
    &\leq O(\mbold{p,J}^{1/2})\left(\frac{s}{t_J}\right)^{\frac{1-\delta_\be}{2}}\bD_p(q^\prime,s) + C\mbold{p,J}t_J^{2\kappa - 1}s\bD_{p}(q^\prime,s)\\
    &\quad + C\mbold{p,J}^\tau t_J^{-(1+\tau)} \bD_{p}(q^\prime,s)^\tau \left( \bD_{p}(q^\prime,s) + \partial_r \bD_{p}(q^\prime,s)\right)\\
    \overset{\eqref{E:Dbound}, \eqref{E:D'bound}, s <t_J}&{\leq} O\biggl(\mbold{p,J}^{\min\left\{1/2, \tau\right\}}\biggr)s^{\min\left\{\frac{1-\delta_\be}{2}, \tau (m-2+2\kappa)\right\}}t_J^{\min\left\{-1+\tau, 2\kappa - 1, -\frac{1-\delta_\be}{2}\right\}}\bD_p(q^\prime,s),
\end{aligned}\end{equation*}
for any $q^\prime \in \Gamma_{p,J}, z^\prime\in \cM_{p,J}$, and $s\in (0,t_J)$. Rescaling the integrand above and choosing $\gamma_1$ appropriately, we conclude the proof. The spatial derivative of $\bH$ is a direct computation, as performed in \cite[Proposition 3.1]{de2018rectifiability}.
\end{proof}

It is possible to measure how much the normal approximation $\cN_{p,J}$ deviates from being homogeneous in terms of the frequency pinching. This is stated in the following lemma. We omit the proof of the lemma since it follows exactly the same argument as in \cite[Proposition 5.2]{de2023fineII}.

\begin{lemma}[Deviation from homogeneity of $\cN_{p,J}$]\label{L:High-Freq:AlmHomog} 
Under \Cref{A:High-Frequency}. There exist $C=C(m,n,Q,\Lambda)>0$ and $\gamma_2=\gamma_2(m,n,Q,\Lambda)>0$ such that the following estimate holds for every pair $\rho, r$ with $0<\rho \leq r \leq t_J$. We define the annuli 
\begin{equation*}
    \cA_{\rho/4}^{2r}(q):= \left(\ball{q}{2r} \setminus \overline{\ball{q}{\rho/4}}\right)\cap \cM_{p,J},
\end{equation*}
then 
\begin{align*}
    \int_{\cA_{\rho/4}^{2r}(q)} \sum_{i}\left | D(\cN_{p,J})_i(z) \frac{d(q,z)\nabla d(q,z)}{|\nabla d(q,z)|} - \bI_p\left(q,d(q,z)\right)(\cN_{p,J})_i(z)|\nabla d(q,z)| \right|^2 \frac{\mathrm{d}z}{d(q,z)} &\\ 
    \leq C \bH_p\left(q,2r\right) \left(W_{\rho/4}^{2r}\left(p,q\right)+\mbold{p,J}^{\gamma_2}r^{\gamma_2}\mathrm{log}\left(\frac{4r}{\rho}\right)\right).
\end{align*}
\end{lemma}

We now show how to estimate the frequency function $\bI$ evaluated at two possibly different boundary points, $p$ and $q$, in terms of the frequency pinching. We omit the proof of the following lemma since it follows entirely analogously to \cite[Lemma 5.4]{de2023fineII} using \Cref{L:spatial-variations-HighFreq}.

\begin{lemma}\label{L:High-Freq:ComparingFreqBetweenPoints}
Under \Cref{A:High-Frequency}. Let $q_1, q_2 \in \Gamma_{p,J}$ with $d(q_1,q_2) \leq r/8$, where $r\in (0,t_J]$. There exist $C=C(m,n,Q,\Lambda)>0$ and $\gamma_3=\gamma_3(m,n,Q,\Lambda)>0$ such that for any $z,w \in [q_1,q_2]$ we have
\begin{equation*}
    |\bI_p(w,r)-\bI_p(z,r)|\leq C \left[\left(W_{r/8}^{4r}(p,q_1)\right)^{1/2}+\left(W_{r/8}^{4r}(p,q_2)\right)^{1/2}+\mbold{p,J}^{\gamma_3}r^{\gamma_3}\right]\frac{d(z,w)}{r}.
\end{equation*}
\end{lemma}

\subsubsection{Spatial variations in the low-frequency regime}

We record here the counterpart for the low-frequency regime of the results in the high-frequency regime.

\begin{lemma}
Under \Cref{A:Low-Frequency}, all the conclusions of \Cref{L:spatial-variations-HighFreq}, \Cref{L:High-Freq:AlmHomog}  and \Cref{L:High-Freq:ComparingFreqBetweenPoints} are valid for 
\begin{equation*}
    4\eta \gamma^{j+1} \leq \rho \leq r <2\gamma^j \mbox{ and } q, q_1, q_2\in \Gamma_{p,\gamma^j}.
\end{equation*}
\end{lemma}

We omit the proof of this lemma here, as it is a repetition of the arguments for the high-frequency regime, along with the observations made in \cite[Section 10.1]{de2023fineII}.

\subsection{Proof of the rectifiability result}

In this section, we prove \Cref{P:Rectifiability-FlatSingular}, which states the $\cH^{m-3}$-rectifiability of the set of one-sided flat singular points. We divide the proof into the high and low-frequency regimes.

\begin{proposition}[Rectifiability of high-frequency points]\label{P:Rectf-HighFreq}
Let $T,\Sigma$, and $\Gamma$, be as in \Cref{A:general}. Then
\begin{equation*}
        \Singb(T) \cap \left\{q\in\Gamma: \Theta^m(T,q)=Q/2, q\mbox{ is flat, and }\bI(T,q)\geq 2 - \alpha_\be \right \}
\end{equation*}
is $\cH^{m-3}$-rectifiable.
\end{proposition}

\begin{proposition}[Rectifiability of low-frequency points]\label{P:Rectf-LowFreq}
Let $T,\Sigma$, and $\Gamma$, be as in \Cref{A:general}. Then
\begin{equation*}
        \Singb(T) \cap \left\{q\in\Gamma: \Theta^m(T,q)=Q/2, q\mbox{ is flat, and }\bI(T,q) < 2 - \alpha_\be \right \}
\end{equation*}
is $\cH^{m-3}$-rectifiable.
\end{proposition}

\begin{remark} We note that the Hausdorff dimensional bound obtained in \Cref{S:DimBound} is not needed to prove either of the two theorems above. In fact, it can be derived as a consequence of the proof. \end{remark}

Paralleling the setting of interior points from \cite{de2023fineII}, we encounter a crucial difference. Singular points for the boundary linear problem are not necessarily points where the multi-valued function vanishes. In contrast, in the interior setting, after applying an argument involving the average of the multi-valued function, we reduce the problem to studying points where the function vanishes.

Instead, we focus on the set of points with frequency greater than $1+\alpha/2$. To prove the bound on $\beta_2$, we must be able to preserve a frequency bound of $1+\alpha/2$ in the limit. While this is straightforward in the high-frequency regime, as the almost monotonicity formula for the frequency holds without changing the center manifold, it is more challenging in the low-frequency regime. In contrast, in the low-frequency regime in \cite{de2023fineII}, the authors only need to preserve the existence of $Q$ points without controlling their frequency.

\subsubsection{Bound on the Jones' $\beta_2$ coefficient in the high-frequency regime}

We now prove \Cref{T:JonesBeta-HighFreq} which gives an upper bound for the $(m-3)$-dimensional Jones' $\beta_2$ coefficient in terms of the frequency pinching.

\begin{theorem}[Frequency pinching controls the Jones' $\beta_2$ coefficient]\label{T:JonesBeta-HighFreq}
Assume $m>3$. There exist constants $\eta_0=\eta_0(m)>0$, $\epsilon_0 = \epsilon_0(m,n,Q,\eta_0,\Lambda)>0$, $\gamma_4=\gamma_4(m,n,Q,\Lambda)>0$, and $C=C(m,n,Q,\Lambda)>0$, such that the following holds. Under \Cref{A:High-Frequency}, and assume further that $\eta < \eta_0$ and $\mbold{p,J} < \varepsilon_0^2$. If $\mu$ is a finite non-negative Radon measure with $\spt(\mu) \subset \Gamma_{p,J}$. Then for all $r \in (0,t_J]$ and every $q \in \oball{r/8}\cap\Gamma_{p,J}$ we have
\begin{equation*}
    \left[\beta_{2,\mu}^{m-3}\left(q,r/8\right)\right]^2 \leq \frac{C}{r^{m-3}}\int_{\ball{q}{r/8}}W_{r/8}^{4r}(p,z)\mathrm{d}\mu(z)+ C\mbold{p,J}^{\gamma_4}\frac{\mu(\ball{q}{r/8})}{ r^{m-3-\gamma_4}}.
\end{equation*}
\end{theorem}

With this control, we are in a position to apply \Cref{T:NaberValtorta-AzzaamTolsa} for Frostman measures, as done in \cite{de2023fineII}, to conclude the proof of the $\cH^{m-3}$-rectifiability in the high-frequency regime, i.e., \Cref{P:Rectf-HighFreq}. We are now left with proving \Cref{T:JonesBeta-HighFreq}.

\begin{proof}[Proof of \Cref{T:JonesBeta-HighFreq}]
We assume without loss of generality that $\mu\left(\ball{q}{r / 8}\right)>0$, otherwise the inequality is trivial. As in \cite[Proposition 7.2]{de2023fineII}, denote the barycenter of $\mu$ in $\ball{q}{r/8}\cap T_{q} \cM_J$ by
\begin{equation*}
    \bar{q}_r:=\frac{1}{\mu\left(\ball{q}{r/8}\right)} \int_{\ball{q}{r/8}} z \mathrm{d} \mu(z), 
\end{equation*}
where $\bp_q$ is the orthogonal projection of $\mathbb{R}^{m+n}$ onto $T_{q} \cM_J$. Next, define the symmetric bilinear form $b_{q, r}: T_{q} \cM_J \times T_{q} \cM_J \rightarrow \R$ as
\begin{equation*}
\begin{aligned}
    b_{q, r}(v, w) & := \int_{\ball{q}{r/8}}\left(\left(z-\bar{q}_r\right) \cdot v\right)\left(\left(z-\bar{q}_r\right) \cdot w\right) \mathrm{d} \mu(z).
\end{aligned}
\end{equation*}

Diagonalizing $b_{q, r}$ yields an orthonormal basis $\left\{v_i\right\}_{i=1}^m$ of eigenvectors and a corresponding family of eigenvalues $0 \leqslant \lambda_m \leqslant \cdots \leqslant \lambda_1$ for the linear map
\begin{equation*}
\begin{aligned}
    L(v) & := \int_{\ball{q}{r/8}}\left(\left(z-\bar{q}_r\right) \cdot v\right) \bp_{q}(z) \mathrm{d} \mu(z),
\end{aligned}
\end{equation*}
which diagonalize $b_{q, r}$, i.e., $L\left(v_i\right)=\lambda_i v_i$ and $b_{q, r}\left(v_i, v_i\right)=\lambda_i$. This gives the following characterization:
\begin{equation*}
    \inf_{\overset{\textup{affine (m-3)-planes }\pi}{\pi\subset T_q\cM_j}} \left[\left(\frac{r}{8}\right)^{-m+1} \int _{\ball{q}{r/8}} \dist(y,L)^2 d\mu(y)\right]^{1/2} = \left(\frac{r}{8}\right)^{-m+1}(\lambda_{m-2}+\lambda_{m-1}+\lambda_m).
\end{equation*}

We now have the following simplification of our claimed inequality:
\begin{equation*}
    \left[\beta_{2,\mu}^{m-3}\left(q,r/8\right)\right]^2 \leq 3 \left(\frac8r\right)^{m-1}\lambda_{m-2}.
\end{equation*}

The argument differs from \cite[Proposition 4.2]{de2023fineII} for two reasons. Firstly, $\nabla d$ needs to be tangent to $\Gamma_{p,J}$, which we ensure through \Cref{L:DistFnc}. Once this is established, we can apply \Cref{L:spatial-variations-HighFreq} with any compactly supported vector field pointing in the same direction as $\nabla d$, following the approach in \cite[Proposition 4.2]{de2023fineII}. Secondly, we need to prove that:
\begin{equation}\label{E:beta_2-Upperbound-Justification}
    \int_{\cA_{r/4}^{2r}(q)}\sum_{h=1}^{m-2}|D\cN(z)\cdot \boldsymbol{\ell}_z(v_h)|^2dz \geq c(\Lambda) \frac{\bH_p(q,2r)}{r}\mbox{ for some }c(\Lambda)>0.
\end{equation}

Note that we are using the $(m-3)$-dimensional Jones' $\beta_2$ coefficient. The sum in the expression above is taken over the $m-2$ vectors derived from the bilinear form $b_{q,r}$. We proceed by contradiction to prove \eqref{E:beta_2-Upperbound-Justification}. We rescale and translate the current such that we have $q = 0$ and $r = 1$. If \eqref{E:beta_2-Upperbound-Justification} fails, we construct a contradiction sequence of currents $T_k$ with $\mbold{p,k} \leq \varepsilon_k \to 0$. From \eqref{E:L:FirstEqsFreqFnc:1} and \eqref{E:Dbound}, the right-hand side of \eqref{E:beta_2-Upperbound-Justification} is bounded above by a constant $C$. Thus, we obtain the inequality: 
\begin{equation}\label{E:contradiction:HighFreq:Beta}
    \int_{\cA_{r/4}^{2r}(0)}\sum_{h=1}^{m-2}|D\cN_k(z)\cdot \boldsymbol{\ell}^k_z(v_h^k)|^2dz < \frac1k.
\end{equation}

For the contradiction sequence, we have the corresponding center manifolds, boundaries, and normalized normal approximations, denoted by $\cM_k$, $\Gamma_k$, and $\cN_k$, respectively. Additionally, we have sequences $\boldsymbol{\ell}_z^k$ and $\{v_h^k\}_k$ that converge to $\mathrm{Id}$ and, for each $h \in \{1,\ldots, m-2\}$, to $v_h$, respectively. Moreover, $\cM_k$ converges to some $m$-dimensional half-plane $\pi_\infty^+$, $\Gamma_k$ converges to $\partial_{\pi_\infty} \pi_\infty^+$, and $\etaa \circ \cN_k \to 0$. 

From this, we can extract a weak limit of $\cN_k$ in $W^{1,2}$, denoted by $u_0 : \obball{2} \cap {x_m \geq 0} \to \cA_Q(\mathbb{R}^n)$, which is a Dirichlet minimizer with Dirichlet energy equal to 1, $\etaa \circ u_0 = 0$, and $u_0$ vanishes identically at its boundary (note that $u_0$ is not necessarily a fine blow-up). The same arguments as in \cite[Section 10.3]{DDHM} (see also \cite[Section 18.1]{delellis2021allardtype}) allow us to upgrade, up to a subsequence, the weak convergence of $\cN_k$ to $u_0$ to \emph{strong} convergence in $W^{1,2}(\obball{2})$. 

Such strong convergence ensures that $ \bI^{T_k}_p(0, s)$ converges to $I_{u_0}(s)$ for any $s \in (0,t_J]$, which implies:
\begin{equation*}
    I_{u_0}(0) \geq 2-\delta_\be \geq 1+\alpha/2.
\end{equation*}

We also obtain using \eqref{E:contradiction:HighFreq:Beta} that
\begin{equation*}
    \int_{\left(\oball{2}\setminus\overline{\oball{1/4}}\right)\cap\{x_m\geq 0\}}\sum_{h=1}^{m-2}|\nabla_{v_h}u_0|^2 = 0.
\end{equation*}

We reach a contradiction in the same manner as in the argument in \Cref{l:linearproblemderivativebound}. Specifically, we can apply \Cref{l:twovariables} to obtain a nontrivial $I_{u_0}(0)$-homogeneous Dirichlet minimizer in $2$ dimensions, with zero boundary value along a straight line. This, however, contradicts \Cref{t:homogeneous2d}, since we have $I_{u_0}(0) \geq 1 + \alpha / 2$.
\end{proof}

\subsubsection{Bound on the Jones' $\beta_2$ coefficient in the low-frequency regime}

We now establish the same control in the low-frequency regime as given in \Cref{T:JonesBeta-HighFreq}. With this control (\Cref{T:JonesBeta-LowFreq}), we are able to apply \Cref{T:NaberValtorta-AzzaamTolsa} for Frostman measures, as done in \cite{de2023fineII}, to conclude the proof of the $\cH^{m-3}$-rectifiability in the low-frequency regime, i.e., \Cref{P:Rectf-LowFreq}. This, in turn, completes the proof of \Cref{P:Rectifiability-FlatSingular}, as explained above.

\begin{theorem}[Frequency pinching controls the Jones' $\beta_2$ coefficient]\label{T:JonesBeta-LowFreq}
Assume $m>3$. There exist $\eta_1=\eta_1(m)>0$, $\epsilon_1 = \epsilon_1(m,n,Q,\eta_1,K)>0$, $\gamma_5=\gamma_5(m,n,Q,\Lambda)>0$, and $C=C(m,n,Q,\Lambda, \gamma)>0$, such that the following holds. 

Under \Cref{A:Low-Frequency}, and assume further that $\eta < \eta_0$ and $\mbold{p,j} < \epsilon_1^2$. If $\mu$ is a finite non-negative upper $(m-3)$-regular Radon measure with $\spt(\mu) \subset \mathfrak{G}_K$. Then for all $r \in (8\eta\gamma^{j+1},\gamma^j]$ and every $q \in \mathfrak{G}_K$ we have
\begin{equation*}
    \left[\beta_{2,\mu}^{m-3}\left(q,r/8\right)\right]^2 \leq \frac{C}{r^{m-3}}\int_{\ball{q}{r/8}}W_{r/8}^{4r}(p,j,z)\mathrm{d}\mu(z)+ C\mbold{p,j}^{\gamma_5}\frac{\mu(\ball{q}{r/8})}{ r^{m-3-\gamma_5}}.
\end{equation*}
\end{theorem}
\begin{proof}
Since the measure $\mu$ is upper $(m-3)$- regular, the coefficient $\beta_{2,\mu}^{m-3}\left(q,r/8\right)$ is always bounded. Therefore, it is sufficient to prove the theorem under the assumption that
\begin{equation*}
    \frac{1}{r^{m-3}}\int_{\ball{q}{r/8}}W_{r/8}^{4r}(p,j,z)\mathrm{d}\mu(z) \leq C\delta_3.
\end{equation*}

This proof follows the same structure as \cite[Proposition 13.3]{de2023fineII}, with the main difference being the proof of the following inequality:
\begin{equation}\label{e:dircontrol}
    \int_{\cA_{r/4}^{2r}(q)}\sum_{h=1}^{m-2}|D\cN_j(z)\cdot \boldsymbol{\ell}_z(v_h)|^2dz \geq c(\Lambda) \frac{\bH_p(q,2r)}{r}\mbox{ for some }c(\Lambda)>0.
\end{equation}

We prove this estimate by contradiction, following the same approach as in \Cref{P:Rectf-HighFreq}, but in this case, we deal with infinitely many center manifolds and intervals of flattening. By rescaling and translating, we assume $q=0$ and $r=1$. If the inequality fails, it must fail for every $\epsilon_1 > 0$ and $\delta > 0$. We can thus take a contradiction sequence of currents $T_k$ such that $\mbold{p}^{(k)} \rightarrow 0$, where $\mbold{p}^{(k)}$ is defined with respect to $T_k$.

Up to rotation, we can assume that $T_0 \Gamma_k = \R^{m-1} \times \left\{0\right\}$ and $\pi_k = \R^m \times \{0\}$, and that
\begin{equation*}
    \int_{\oball{1/8}}W_{1/4}^{2}(p,j(k),z)d\mu_k(z) \leq \frac{1}{k}.
\end{equation*}

This allows us to select points $x_k \in \Gamma_{p,1} \cap \oball{1/8}$ such that $W_{1/4}^{2}(p,j(k),x_k) \leq k^{-1}$. Up to subsequences, $\mu_k \rightharpoonup \mu$ with $\spt(\mu)$ being a finite non-negative Radon measure supported in $\R^{m-1} \times \left\{0\right\}$, and $x_k \to x \in \textup{spt}(\mu)$.

We let $\theta_k = \gamma^{j(k)} \in [1, \frac{1}{8\eta\gamma})$ be such that the modified interval of flattening given by $(\gamma^{j(k)+1}, \gamma^{j(k)}]$ becomes $(8\eta\gamma\theta_k, \theta_k]$. Up to a subsequence, we assume $\theta_k \to \theta$.

For each $T_k$, we construct the corresponding center manifold $\cM_k$ and normal approximations $\cN_k$. Arguing as in \Cref{T:JonesBeta-HighFreq}, we obtain a limiting function $u_0 : (\obball{\theta} \setminus \overline{\obball{8\eta\gamma\theta}}) \cap \{ x_m \geq 0 \} \to \cA_Q(\R^n)$ that is a Dirichlet minimizer with Dirichlet energy equal to $1$, $\etaa \circ u_0 = 0$, and $u_0$ is identically zero at its boundary (note that $u_0$ is not necessarily a fine blow-up).

Moreover, as in \Cref{T:JonesBeta-HighFreq}, the convergence of $\cN_k$ to $u_0$ is \emph{strong} in $W^{1,2}(\obball{\theta} \setminus \overline{\obball{8\eta\gamma\theta}})$ up to subsequences. This strong convergence furnishes
\begin{equation*}
    \bI^{T_k}_p(x_k,\theta_k) \to I_{u_0}(x,\theta), \ \bI^{T_k}_p(x_k,8\eta\gamma\theta_k) \to I_{u_0}(x,8\eta\gamma\theta),\mbox{ and }I_{u_0}(x,\theta)\geq 1+\alpha/2,
\end{equation*}
where $\bI^{T_k}$ denotes the frequency defined with respect to $T_k$. 

Since $W_{1/4}^{2}(p,j(k),x_k) \to 0$, we have that $I_{u_0}(x, 8\eta\gamma\theta) = I_{u_0}(x, \theta)$. Hence, \Cref{L:constantfrequencyannulus} implies that $u_0$ is $I_0$-homogeneous on $\obball{\theta} \cap \{ x_m \geq 0 \}$ with $I_0 \geq 1 + \alpha / 2$. Since the contradiction sequence is violating \eqref{e:dircontrol}, we arrive again at the identity
\begin{equation*}
    \int_{\left(\oball{2}\setminus \oball{1/4}\right)\cap\{x_m\geq 0\}}\sum_{h=1}^{m-2}|\nabla_{v_h}u_0|^2 = 0.
\end{equation*}

Again, by the same argument given in \Cref{l:linearproblemderivativebound}, we conclude, using \Cref{l:twovariables}, that we obtain an $I_0$-homogeneous Dirichlet minimizer $u$ in $2$d with zero boundary value along a straight line and $I_0 \geq 1 + \alpha / 2$, which violates \Cref{t:homogeneous2d}. This concludes the proof of the inequality \eqref{e:dircontrol}, and, arguing as in \cite[Proposition 13.3]{de2023fineII}, we conclude the proof of the proposition.
\end{proof}

        \bibliographystyle{abbrv}
        \bibliography{main}

\begin{thebibliography}{10}

\bibitem{AllPhD}
W.~K. Allard.
\newblock {On boundary regularity for {P}lateau's problem}.
\newblock {\em Bull. Amer. Math. Soc.}, 75:522--523, 1969.

\bibitem{AllB}
W.~K. Allard.
\newblock {On the first variation of a varifold: boundary behavior}.
\newblock {\em Ann. of Math. (2)}, 101:418--446, 1975.

\bibitem{Alm}
J.~F.~J. Almgren.
\newblock {\em {Almgren's big regularity paper}}, volume~1 of {\em {World Scientific Monograph Series in Mathematics}}.
\newblock World Scientific Publishing Co. Inc., River Edge, NJ, 2000.

\bibitem{azzam2015characterization}
J.~Azzam and X.~Tolsa.
\newblock Characterization of n-rectifiability in terms of jones’ square function: Part ii.
\newblock {\em Geometric and Functional Analysis}, 25(5):1371--1412, 2015.

\bibitem{GMT_prob}
J.~E. Brothers.
\newblock Some open problems in geometric measure theory and its applications suggested by participants of the 1984 {AMS} summer institute.
\newblock In J.~E. Brothers, editor, {\em Geometric measure theory and the calculus of variations ({A}rcata, {C}alif., 1984)}, volume~44 of {\em Proc. Sympos. Pure Math.}, pages 441--464. Amer. Math. Soc., Providence, RI, 1986.

\bibitem{de2022regularity}
C.~De~Lellis.
\newblock The regularity theory for the area functional (in geometric measure theory).
\newblock In {\em International Congress of Mathematicians}, 2022.

\bibitem{DDHM}
C.~{De Lellis}, G.~{De Philippis}, J.~{Hirsch}, and A.~{Massaccesi}.
\newblock {On the boundary behavior of mass-minimizing integral currents}.
\newblock {\em Memoirs of the American Mathematical Society}, 291, 166pp, 2023.

\bibitem{iancamillorectifiability}
C.~De~Lellis and I.~Fleschler.
\newblock An elementary rectifiability lemma and some applications.
\newblock {\em Mathematical Research Letters}, 31(4):1029--1045, 2024.

\bibitem{de2018rectifiability}
C.~De~Lellis, A.~Marchese, E.~Spadaro, and D.~Valtorta.
\newblock Rectifiability and upper minkowski bounds for singularities of harmonic $ q $-valued maps.
\newblock {\em Commentarii Mathematici Helvetici}, 93(4):737--779, 2018.

\bibitem{de2023fineIII}
C.~De~Lellis, P.~Minter, and A.~Skorobogatova.
\newblock The fine structure of the singular set of area-minimizing integral currents iii: Frequency 1 flat singular points and hm-2-ae uniqueness of tangent cones.
\newblock {\em arXiv preprint arXiv:2304.11553}, 2023.

\bibitem{delellis2021uniqueness}
C.~De~Lellis, S.~Nardulli, and S.~Steinbr{\"u}chel.
\newblock Uniqueness of boundary tangent cones for 2-dimensional area-minimizing currents.
\newblock {\em Nonlinear Analysis}, 230:113235, 2023.

\bibitem{delellis2021allardtype}
C.~De~Lellis, S.~Nardulli, and S.~Steinbr{\"u}chel.
\newblock An allard-type boundary regularity theorem for $2 d $ minimizing currents at smooth curves with arbitrary multiplicity.
\newblock {\em Publications math{\'e}matiques de l'IH{\'E}S}, pages 1--118, 2024.

\bibitem{de2023fineI}
C.~De~Lellis and A.~Skorobogatova.
\newblock The fine structure of the singular set of area-minimizing integral currents i: the singularity degree of flat singular points.
\newblock {\em arXiv preprint arXiv:2304.11552}, 2023.

\bibitem{de2023fineII}
C.~De~Lellis and A.~Skorobogatova.
\newblock The fine structure of the singular set of area-minimizing integral currents ii: rectifiability of flat singular points with singularity degree larger than $1$.
\newblock {\em arXiv preprint arXiv:2304.11555}, 2023.

\bibitem{DS1}
C.~{De Lellis} and E.~Spadaro.
\newblock {{$Q$}-valued functions revisited}.
\newblock {\em Mem. Amer. Math. Soc.}, 211(991):vi+79, 2011.

\bibitem{DS3}
C.~{De Lellis} and E.~Spadaro.
\newblock {Regularity of area minimizing currents {I}: gradient {$L^p$} estimates}.
\newblock {\em Geom. Funct. Anal.}, 24(6):1831--1884, 2014.

\bibitem{DS2}
C.~{De Lellis} and E.~Spadaro.
\newblock {Multiple valued functions and integral currents}.
\newblock {\em Ann. Sc. Norm. Super. Pisa Cl. Sci. (5)}, 14(4):1239--1269, 2015.

\bibitem{DS4}
C.~{De Lellis} and E.~Spadaro.
\newblock {Regularity of area minimizing currents {II}: center manifold}.
\newblock {\em Ann. of Math. (2)}, 183(2):499--575, 2016.

\bibitem{DS5}
C.~{De Lellis} and E.~Spadaro.
\newblock {Regularity of area minimizing currents {III}: blow-up}.
\newblock {\em Ann. of Math. (2)}, 183(2):577--617, 2016.

\bibitem{DSS3}
C.~{De Lellis}, E.~Spadaro, and L.~Spolaor.
\newblock {Regularity {T}heory for 2-{D}imensional {A}lmost {M}inimal {C}urrents {II}: {B}ranched {C}enter {M}anifold}.
\newblock {\em Ann. PDE}, 3(2):3:18, 2017.

\bibitem{DSS1}
C.~{De Lellis}, E.~Spadaro, and L.~Spolaor.
\newblock {Uniqueness of tangent cones for two-dimensional almost-minimizing currents}.
\newblock {\em Comm. Pure Appl. Math.}, 70(7):1402--1421, 2017.

\bibitem{DSS2}
C.~De~Lellis, E.~Spadaro, and L.~Spolaor.
\newblock Regularity theory for 2-dimensional almost minimal currents i: Lipschitz approximation.
\newblock {\em Transactions of the American Mathematical Society}, 370(3):1783--1801, 2018.

\bibitem{DSS4}
C.~De~Lellis, E.~Spadaro, and L.~Spolaor.
\newblock Regularity theory for $2 $-dimensional almost minimal currents iii: Blowup.
\newblock {\em Journal of Differential Geometry}, 116(1):125--185, 2020.

\bibitem{Fed}
H.~Federer.
\newblock {\em {Geometric measure theory}}.
\newblock {Die Grundlehren der mathematischen Wissenschaften, Band 153}. Springer-Verlag New York Inc., New York, 1969.

\bibitem{ian2024uniqueness}
I.~Fleschler.
\newblock On the uniqueness of tangent cones to area minimizing currents at boundaries with arbitrary multiplicity.
\newblock {\em preprint on ArXiv}, 2024.

\bibitem{ian2024example}
I.~Fleschler.
\newblock An essential one sided boundary singularity for a $3$-dimensional area minimizing current in $\mathbb{R}^5$.
\newblock {\em arXiv preprint arXiv:2502.15047}, 2025.

\bibitem{HS}
R.~Hardt and L.~Simon.
\newblock {Boundary regularity and embedded solutions for the oriented {P}lateau problem}.
\newblock {\em Ann. of Math. (2)}, 110(3):439--486, 1979.

\bibitem{krummel2017fine}
B.~Krummel and N.~Wickramasekera.
\newblock Fine properties of branch point singularities: Dirichlet energy minimizing multi-valued functions.
\newblock {\em arXiv preprint arXiv:1711.06222}, 2017.

\bibitem{krummel2023analysisI}
B.~Krummel and N.~Wickramasekera.
\newblock Analysis of singularities of area minimizing currents: a uniforn height bound, estimates away from branch points of rapid decay, and uniqueness of tangent cones.
\newblock {\em arXiv preprint arXiv:2304.10272}, 2023.

\bibitem{krummel2023analysisII}
B.~Krummel and N.~Wickramasekera.
\newblock Analysis of singularities of area minimizing currents: planar frequency, branch points of rapid decay, and weak locally uniform approximation.
\newblock {\em arXiv preprint arXiv:2304.10653}, 2023.

\bibitem{NV}
A.~Naber and D.~Valtorta.
\newblock The singular structure and regularity of stationary varifolds.
\newblock {\em Journal of the European Mathematical Society}, 2015.

\bibitem{nardulli2022density}
S.~Nardulli and R.~Resende.
\newblock Density of the boundary regular set of 2d area minimizing currents with arbitrary codimension and multiplicity.
\newblock {\em Advances in Mathematics}, 455:109889, 2024.

\bibitem{nardulli2024interior}
S.~Nardulli and R.~Resende.
\newblock Interior regularity of area minimizing currents within a $c^{2,\alpha}$-submanifold.
\newblock {\em arXiv preprint arXiv:2404.17407}, 2024.

\bibitem{simon2014introduction}
L.~Simon.
\newblock Introduction to geometric measure theory.
\newblock {\em Tsinghua Lectures}, 2014.

\bibitem{white1983regularity}
B.~White.
\newblock Regularity of area-minimizing hypersurfaces at boundaries with multiplicity.
\newblock In {\em Seminar on minimal submanifolds}, volume 103, pages 293--301, 1983.

\bibitem{White97}
B.~White.
\newblock {Stratification of minimal surfaces, mean curvature flows, and harmonic maps}.
\newblock {\em J. Reine Angew. Math.}, 488:1--35, 1997.

\end{thebibliography}
\end{document}